\documentclass[12pt]{amsart}
\usepackage{amsmath}
\usepackage{amsthm}
\usepackage{amsfonts} 
\usepackage{hyperref}
\usepackage{epsfig}

\numberwithin{equation}{section}

\allowdisplaybreaks

\newtheorem{theorem}{Theorem}[section]
\newtheorem{proposition}{Proposition}[section]
\newtheorem{lemma}{Lemma}[section]
\newtheorem{remark}{Remark}[section]

\renewcommand{\epsilon}{\varepsilon}

\newcommand{\abs}[1]{\left\vert #1\right\vert}
\newcommand{\1}[1]{{\mathbf 1}{\{#1\}}}
\newcommand{\R}{\mathbb{R}}
\newcommand{\N}{\mathbb{N}}
\newcommand{\Z}{\mathbb{Z}}
 
\newcommand{\PR}{\mathbb{P}}
\newcommand{\ES}{\mathbb{E}}
\newcommand{\C}{\mathcal{C}}

\title[]{The speed of a biased random walk on a percolation cluster at high density}
\date{}
\author[A.~Fribergh]{Alexander FRIBERGH}
\address{Universit\'e de Lyon;
Universit\'e Lyon 1;
INSA de Lyon, F-69621;
Ecole Centrale de Lyon;
CNRS, UMR5208, Institut Camille Jordan,
43 blvd du 11 novembre 1918,
F-69622 Villeurbanne-Cedex, France} 
\email{fribergh@math.univ-lyon1.fr}


\begin{document}

\begin{abstract}

We study the speed of a biased random walk on a percolation cluster on $\Z^d$ in function of the percolation parameter $p$. We obtain a first order expansion of the speed at $p=1$ which proves that percolating slows down the random walk at least in the case where the drift is along a component of the lattice.

\end{abstract}

\maketitle

\section{Introduction}

Random walks in reversible random environments are an important subfield of random walks in random media. In the last few years a lot of work has been done to understand these models on $\Z^d$, one of the most challenging being the model of reversible random walks on percolation clusters, which has raised many questions.

In this model, the walker is restrained to a locally inhomogeneous graph, making it difficult to transfer any method used for elliptic random walks in random media. In the beginning results concerned simple random walks, the question of recurrence and transience (see~\cite{GKZ}) was first solved and latter on a quenched invariance principles was proved in~\cite{BB} and~\cite{MP}. More recently new results (e.g.~\cite{Mathieu} and~\cite{BBHK}) appeared, but still under the assumption that the walker has no global drift.

The case of drifted random walks on percolation cluster features a very interesting phenomenon which was first described in the theoretical physics literature (see~\cite{Dhar} and~\cite{DS}), as the drift increases the model switches from a ballistic to a sub-ballistic regime. From a mathematical point of view, this conjecture was partially addressed in~\cite{Berger} and~\cite{Sznitman}. This slowdown is due to the fact that the percolation cluster contains arbitrarily large parts of the environment which act as traps for a biased random walk. This phenomenon, and more, is known to happen on inhomogeneous Galton-Watson trees, cf.~\cite{LPP},~\cite{BH} and~\cite{FG}.

Nevertheless this model is still not well understood and many questions remain open, the most famous being the existence and the value of a critical drift for the expected phase transition. Another question of interest is the dependence of the limiting velocity with respect to the parameters of the problem i.e.~the percolation parameter and the bias. This last question is not specific to this model, but understanding in a quantitative, or even qualitative way, the behaviour of speed of random walks in random media seems to be a difficult problem and very few results are currently available on $\Z^d$ (see~\cite{Sabot}).

In this article we study the dependence of the limiting velocity with respect to the percolation parameter around $p=1$. We try to adapt the methods used in~\cite{Sabot} which were introduced to study environments subject to small perturbations in a uniformly-elliptic setting. For biased-random walk on a percolation cluster of high density, the walk is subject to rare but arbitrarily big pertubations so that the problem is very different and appears to be more difficult.

The methods rely mainly on a careful study of Kalikow's auxilary random walk which is known to be linked to the random walks in random environments (see~\cite{Zeitouni} and~\cite{SZ2}) and also to the limiting velocity of such walks when it exists (see~\cite{Sabot}). Our main task is to show that the unbounded effects of the removal of edges, once averaged over all configurations, is small. This will enable us to consider Kalikow's auxilary random walk as a small perturbation of the biased random walk on $\Z^d$. As far as we know it is the first time such methods are used to study a random conductance model or even non-elliptic random walks in random media.

\section{The model }

The models presented in~\cite{Berger} and~\cite{Sznitman} are slightly different, we choose to consider the second one as it is a bit more general, since it allows the drift to be in any direction. Nevertheless all the following can be adapted without any difficulty to the model described in~\cite{Berger}.

Let us describe the environment, we consider the set of edges $E(\Z^d)$ of the lattice $\Z^d$ for some $d\geq 2$. We fix $p \in (0,1)$ and perform a Bernoulli bond-percolation, that is we pick a random configuration $\omega \in \Omega:=\{0,1\}^{E(\Z^d)}$ where each edge has probability $p$ (resp.~$1-p$) of being open (resp.~closed) independently of all other edges. Let us introduce the corresponding measure 
\[
{\bf P}_{p}= (p \delta_1 + (1-p) \delta_0)^{\otimes E(\Z^d)}.
\]

Hence an edge $e$ will be called open (resp.~closed) in the configuration $\omega$ if $\omega(e)=1$ (resp.~$\omega(e)=0$). This naturally induces a subgraph of $\Z^d$ which will be denoted $\omega$ and it also yields a partition of $\Z^d$ into open clusters. 

It is classical in percolation that for $p>p_c(d)$, where $p_c(d)\in (0,1)$ denotes the critical percolation probability of $\Z^d$ (see~\cite{Grimmett}), we have a unique infinite open cluster $K_{\infty}(\omega)$, ${\bf P}_p$-a.s.. The corresponding set of configuration is denoted by $\Omega_0$. Moreover the following event has positive ${\bf P}_p$-probability 
\[
\mathcal{I}=\{ \text{there is a unique infinite cluster $K_{\infty}(\omega)$ and it contains $0$} \}.
\]

In order to define the random walk, we introduce a bias $\ell=\lambda \vec \ell$ of strengh $\lambda >0$ and a direction $\vec \ell$ which is in the unit sphere with respect to the Euclidian metric of $\R^d$. On a configuration $\omega \in \Omega$, we consider the Markov chain of law $P_x^{\omega}$ on $\Z^d$ with transition probabilities $p^{\omega}(x,y)$ for $x,y\in \Z^d$ defined by
\begin{enumerate}
\item $X_0=x$, $P_x^{\omega}$-a.s.,
\item $p^{\omega}(x,x)=1$, if $x$ has no neighbour in $\omega$,
\item $\displaystyle{p^{\omega}(x,y) =\frac{c^{\omega}(x,y)}{\sum_{z \sim x} c^{\omega}(x,z)}}$,
\end{enumerate}
where $x\sim y$ means that $x$ and $y$ are adjacent in $\Z^d$ and also we set
\[
\text{for all $x,y \in \Z^d$,} \qquad \ c^{\omega}(x,y)=\begin{cases}
                                            e^{(y+x)\cdot\ell} & \text{ if } x\sim y \text{ and } \omega(\{x,y\})=1, \\
                                            0           & \text{ otherwise.}\end{cases} 
\]                                           

We see that this Markov chain is reversible with invariant measure given by 
\[
\pi^{\omega}(x)=\sum_{y \sim x} c^{\omega}(x,y).
\]  

Let us call $c^{\omega}(x,y)$ the conductance between $x$ and $y$ in the configuration $\omega$, this is natural because of the links existing between reversible Markov chains and electrical networks. We will be making extensive use of this relation and we refer the reader to~\cite{DoyleSnell} and~\cite{LP} for a further background. Moreover for an edge $e=[x,y]\in E(\Z^d)$, we denote $c^{\omega}(e)= c^{\omega}(x,y)$ and $r^{\omega}(e)=1/c^{\omega}(e)$.
 
Finally the annealed law of the biased random walk on the infinite percolation cluster will be the semi-direct product $\PR_p = {\bf P}_p[\,\cdot \mid \mathcal{I}\,] \times  P_0^{\omega}[\,\cdot\,]$.

The starting point of our work is the existence of a constant limiting velocity which was proved in~\cite{Sznitman} and with some additional work Sznitman managed to obtain the following result
\begin{theorem}
\label{LLN}
For any $d\geq 2$, $p\in (p_c(d),1)$ and any $\ell\in \R^d_*$, there exists $v_{\ell}(p) \in \R^d$ such that
\[
\text{for $\omega$}-{\bf P}_p[\,\cdot \mid \mathcal{I}\,]-\text{a.s.},\ \lim_{n \to \infty} \frac{X_n}n = v_{\ell}(p),\ P_{0}^{\omega}-\text{a.s.}.
\]

Moreover there exist $\lambda_1(p,d,\ell),\lambda_2(p,d,\ell) \in \R_+$ such that
\begin{enumerate}
\item for $\lambda=\ell \cdot \vec \ell < \lambda_1(p,d,\ell)$, we have $v_{\ell}(p) \cdot \vec \ell >0$,
\item for $\lambda=\ell \cdot \vec \ell > \lambda_2(p,d,\ell)$, we have $v_{\ell}(p)=0$.
\end{enumerate}
\end{theorem}

Our main result is a first order expansion of the limiting velocity with respect to the percolation parameter at $p=1$. As in~\cite{Sabot}, the result depends on certain Green functions defined for a configuration $\omega$ as
\[
\text{for any $x,y \in \Z^d$,}\qquad  G^{\omega}(x,y):=E_x^{\omega}\Bigl[ \sum_{n\geq 0} \1{X_n \in y}\Bigr].
\]

Before stating our main theorem we recall that $v_{\ell}(1)=\sum_{e\in \nu} p(e) e$, where $\omega_0$ is the environment at $p=1$, $p(e)=p^{\omega_0}(0,e)$, and $\nu$ is the set of unit vectors of $\Z^d$.

\begin{theorem}{\label{asymptoticspeed}}
For $d\geq 2$, $p\in (p_c(d),1)$ and for any $\ell\in \R^d_*$, we have
\[
v_{\ell}(1-\epsilon)=v_{\ell}(1)-\epsilon \sum_{e\in \nu} (v_{\ell}(1)\cdot e)(G^{\omega^{e}_0}(0,0)-G^{\omega^{e}_0}(e,0))(v_{\ell}(1)-d_e)+o(\epsilon),
\]
where for any $e\in \nu$ we denote 
\[
\text{for $f\in E(\Z^d)$,}\qquad \omega^{e}_0(f)=\1{f\neq e} \text{ and }d_e=\sum_{e'\in \nu} p^{\omega_0^{e}}(0,e')e',
\]
are respectively the environment where only the edge $[0,e]$ is closed and its corresponding mean drift at $0$.
\end{theorem}

\begin{proposition}
\label{incr_speed}
Let us denote $J^e=G^{\omega_0}(0,0)-G^{\omega_0}(e,0)$ for $e\in \nu$. We can rewrite the first term of the expansion in the following way
\[
v_{\ell}'(1)=\sum_{e\in \nu} (v_{\ell}(1)\cdot e)  \frac{p(e) J^e}{1-p(e)J^e-p(-e)J^{-e}}(e-v_{\ell}(1)),
\]
so that if for $e\in \nu$ such that $v_{\ell}(1) \cdot e >0$ we have $v_{\ell}(1) \cdot e \geq \abs{\abs{v_{\ell}(1)}}_2^2$ then
\[
v_{\ell}(1)\cdot v_{\ell}'(1)>0,
\]
which in words means that the percolation slows down the random walk at least at $p=1$.

The previous condition is verified for example in the following cases
\begin{enumerate}
\item $\vec \ell \in \nu$, i.e.~when the drift is along a component of the lattice,
\item $\ell = \lambda \vec \ell$, where $\lambda<\lambda_c(\vec \ell)$ for some $\lambda_c(\vec \ell)>0$, i.e.~when the drift is weak.
\end{enumerate}
\end{proposition}

\begin{remark}
 The property of Proposition~\ref{incr_speed} is expected to hold for any drift, but we were unable to carry our the computations. More generaly the previous should be true in a great variety of cases, in particular one could hope it holds in the whole supercritical regime. For a somewhat related conjecture, see~\cite{Chen}.
 \end{remark}
 
\begin{remark}
Another natural consequence which is not completely obvious to prove is that the speed is positive for $p$ close enough to $1$.
\end{remark}

\begin{remark}
Finally, this result can give some insight on the dependence of the speed with respect to the bias. Indeed, fix a bias $\ell$ and some $\mu>1$, Theorem~\ref{asymptoticspeed} implies that for $\epsilon_0=\epsilon_0(\ell , \mu)>0$ small enough we have
\[
v_{\mu \ell}(1-\epsilon)\cdot \vec \ell>v_{\ell}(1-\epsilon)\cdot \vec \ell \text{ for } \epsilon<\epsilon_0.
\]
\end{remark}

Before turning to the proof of this result, we introduce some further notations. Let us also point out that we will refer to the percolation parameter as $1-\epsilon$ instead of $p$ and assume 
\begin{equation}
\label{ass_eps}
\epsilon<  1/2,
\end{equation}
in particular we have $1-\epsilon>p_c(d)$ for all $d\geq 2$.

We denote by $\{x \leftrightarrow y \}$ the event that $x$ and $y$ are connected in $\omega$. If we want to emphasize the configuration we will use $\{ x\stackrel{\omega}{\leftrightarrow} y \}$. Accordingly, let us denote $K^{\omega}(x)$ the cluster (or connected component) of $x$ in $\omega$.

Given a set $V$ of vertices of $\Z^d$ we denote by $\abs{V}$ its cardinality, by $E(V)=\{ [x,y] \in E(\Z^d)\mid~x,y\in V\}$ its edges and 
\[
\partial V= \{x \in V \mid y\in \Z^d \setminus V,~x\sim y\},~\partial_E V= \{[x,y] \in E(\Z^d)\mid~x\in V,\ y \notin V \},
\]
and also for $B$ a set of edges of $E(\Z^d)$ we denote
\[
\partial B=\{x \mid ~\exists~y,z,~[x,y] \in B, [x,z] \notin B\},~\partial_E B=\{ [x,y] \mid x \in \partial B, y\notin V(B)\},
\]
where $V(B)=\{x \in \Z^d \mid ~\exists y \in \Z^d~[x,y]\in B\}$.

Given a subgraph $G$ of $\Z^d$ containing all vertices of $\Z^d$, we denote $d_G(x,y)$ the graph distance in $G$ induced by $\Z^d$ between $x$ and $y$, moreover if $x$ and $y$ are not connected in $G$ we set $d_G(x,y)=\infty$. In particular $d_{\omega}(x,y)$ is the distance in the percolation cluster if $\{x\leftrightarrow y\}$. Moreover for $x\in G$ and $k \in \N$, we denote the ball of radius $k$ by
\[
B_G(x,k)=\{y\in G,\ d_G(x,y) \leq k\} \text{ and } B_G^E(x,k)=E(B_G(x,k)),
\]
where we will omit the subscript when $G=\Z^d$.

Let us denote by $(e^{(i)})_{i=1\ldots d}$ an orthonormal basis of $\Z^d$ such that $e^{(1)}\cdot\vec \ell \geq e^{(2)}\cdot \vec \ell \geq \cdots \geq e^{(d)}\cdot \vec \ell \geq 0$, in particular we have $e^{(1)}\cdot\vec \ell\geq  1 /\sqrt d$.

In order to control volume growth let us define $\rho_d$ such that 
\[
\text{for all $r\geq 1$,}\qquad \abs{B(x,r)}\leq \rho_d r^d \text{ and } \abs{\partial B(x,r)} \leq \rho_d r^{d-1}.
\]

We will need to modify the configuration of the percolation cluster at certain vertices. So given $A_1, A_2 \in E(\Z^d)$, $B_1 \subset A_1$ and $B_2\subset A_2$, let us denote $\omega^{A_1,B_1}_{A_2,B_2}$ the configuration such that
\begin{enumerate}
\item $\omega^{A_1,B_1}_{A_2,B_2}([x,y])= \omega([x,y])$,  if  $[x,y] \notin A_1 \cup A_2$,
\item $\omega^{A_1,B_1}_{A_2,B_2}([x,y])=\1{[x,y] \notin B_1}$, if $[x,y] \in A_1$,
\item $\omega^{A_1,B_1}_{A_2,B_2}([x,y])=\1{[x,y] \notin B_2}$, if $[x,y] \in A_2 \setminus A_1$,
\end{enumerate}
which in words means that we impose that the set of closed edges  (in $\omega^{A_1,B_1}_{A_2,B_2}$) of $A_1$ (resp.~$A_2$) is exactly $B_1$ (resp.~$B_2$), and in case of an intersection between $A_1$ and $A_2$ the condition imposed by $A_1$ is most important. Furthermore given $A\in E(\Z^d)$ and $B\subset A$ the configurations $\omega^{A,B}$ and $\omega_{A,B}$ are such that
\begin{enumerate}
\item $\omega^{A,B}([x,y])=\omega_{A,B}([x,y])= \omega([x,y])$,  if  $[x,y] \notin A$,
\item $\omega^{A,B}([x,y])=\omega_{A,B}([x,y])=\1{[x,y] \notin B}$, if $[x,y] \in A$,
\end{enumerate}
that is equal to $\omega$ everywhere except on the edges of $A$. The closed edges of $A$ (in the configuration $\omega^{A,B}$ or $\omega_{A,B}$) being exactly those in $B$. 

 For $k_1,k_2\geq 1$, $z_1,z_2\in \Z^d$, $B_1 \subset B^E(0,k_1)$ and $B_2 \subset B^E(0,k_2)$, we introduce
 \begin{equation}
\label{notation_omega}
\omega^{(z_1,k_1),B_1}_{(z_2,k_1),B_2}:=\omega^{B^E(z_1,k_1),z_1+B_1}_{B^E(z_2,k_2),z_2+B_2} \text{ and for $B_1,B_2\subset \nu$, } \omega^{z_1,B_1}_{z_2,B_2}:=\omega^{(z_1,1),z_1+B_1}_{(z_2,1),z_2+B_2},
\end{equation}
 to describe configurations modified on balls. We define the same type of notations without the subscript or the superscript in the natural way.

Moreover we will use shortened notations when we impose that all edges of a certain set are open (resp.~closed), for example
 \begin{equation}
\label{notation_omega1}
\omega^{A,1}:=\omega^{A,\emptyset}\text{ and }\omega^{A,0}:=\omega^{A,A},
\end{equation}
 to denote in particular the special cases where all (resp.~no) edges of $A$ are open. We will use of combinations of these notations, for example, $\omega^{(z,k),1}:=\omega^{B^E(z,k),\emptyset}$.

In connection with that, for a given configuration $\omega\in \Omega$, we call configuration of $z$ and denote
\[
\C(z)=\{e \in \nu,\ \omega([z,z+e])=0\},
\]
the set of closed edges adjacent to $z$.

Hence we can denote $e\in \nu$ and $A\subset \nu$
\begin{equation}
\label{shorter}
p^{A}(e)=p^{\omega^{0,A}}(0,e),~c(e)=c^{\omega^{0,1}}(e) \text{ and } \pi^A=\pi^{\omega^{0,A}}(0),
\end{equation}
this means for example $p^{A}$ is the transition probability along the edge $e$ under the configuration $A$.

Furthermore the pseudo elliptic constant $\kappa_0=\kappa_0(\ell,d)>0$ will denote
\begin{equation}
\label{kap0}
\kappa_0=\min_{A\subset\nu, A\neq \nu, e\notin A} p^{A}(e),
\end{equation} 
which is the minimal non-zero transition probability.

Similarly we fix $\kappa_1=\kappa_1(\ell,d)>0$ such that
\begin{equation}
\label{kap1}
\frac 1 {\kappa_1} \pi^{\omega^{z,A}}(z) \leq e^{2\lambda z. \vec \ell} \leq \kappa_1\pi^{\omega^{z,A}}(z),
\end{equation}
 for any $A\subset \nu$, $A\neq \nu$ and $z\in\Z^d$.

Finally $\tau_{\delta}$ will denote a geometric random variable of parameter $1-\delta$ independent of the random walk and the environment moreover for $A\subset \Z^d$ set
\[
T_A=\inf \{ n \geq 0,~X_n\in A\} \text{ and } T_A^+=\inf \{ n \geq 1,~X_n\in A\},
\]
and for $z\in \Z^d$ we denote $T_z$ (resp. $T_z^+$) for $T_{\{z\}}$ (resp. $T_{\{z\}}^+$).

Concerning constants we choose to denote them by $C_i$ for global constants, or $\gamma_i$ for local constants and will implicitely be supposed to be in $(0,\infty)$. Their dependence with respect to $d$ and $\ell$ will not always be specified.

Let us present the structure of the paper. In Section~\ref{s_kalikow}, we will introduce the central tool for the computation of the expansion of the speed: Kalikow's environment and link it to the asymptotic speed. Then, we will concentrate on getting the continuity of the speed, mathematically the problem is simply reduced to giving upper bounds on quantities depending on Green functions, on a more heuristical level our aim is to understand the slowdown induced by unlikely configurations where \lq \lq traps\rq \rq appear. Since getting the upper bound is a rather complicated and technical matter we will first give a quick sketch, as soon as further notations are in place, and try to motivate our approach at the end of the next section.

In Section~\ref{sect_res} and Section~\ref{sect_perco}, we will respectively give estimates on the behaviour of the random walk near traps and on the probability of appearence of such traps in the percolation cluster. Then in Section~\ref{sect_cont} we will put together the previous results to prove the continuity of the speed.

 The proof of Theorem~\ref{asymptoticspeed} will be done in Section~\ref{s_deriv}. In order to obtain the first order expansion, the task is essentially similar to obtaining the continuity, but the computations are much more involved and will partly be postponed to Section~\ref{s_tech}.

 Finally Proposition~\ref{incr_speed} is proved in Section~\ref{s_simplification}.

\section{Kalikow's auxiliary random walk}
\label{s_kalikow}

We denote for $x,y\in \Z^d$, $P$ a Markov operator and $\delta<1$, the Green function of the random walk killed at geometric rate $1-\delta$ by 
\[
G_{\delta}^P(x,y):=E_x^{P}\Bigl[\sum_{k=0}^{\infty}\delta^k\1{X_k=y}\Bigr]\text{ and } G_{\delta}^{\omega}(x,y):=G_{\delta}^{P^{\omega}}(x,y),
\]
where $P^{\omega}$ is the Markov operator associated with the random walk in the environment $\omega$.

Then we introduce the so-called Kalikow environment associated with the point $0$ and the environment ${\bf P}_{1-\epsilon}[\,\cdot \mid \mathcal{I}\,]$, which is given for $z,y\in \Z^d$, $\delta<1$ and $e\in \nu$ by
\[
\widehat{p}_{\delta}^{\epsilon}(z,z+e)= \frac{{\bf E}_{1-\epsilon}[G_{\delta}^{\omega}(0,z) p^{\omega}(z,z+e)|\mathcal{I}]}{{\bf E}_{1-\epsilon}[G_{\delta}^{\omega}(0,z)|\mathcal{I}]}.
\]

The familly $(\widehat{p}_{\delta}^{\epsilon}(z,z+e))_{z\in \Z^d, e\in \nu}$ defines transition probabilities of a certain Markov chain on $\Z^d$. It is called Kalikow's auxiliary random walk and its first appearence in a slightly different form goes back to~\cite{Kalikow}.

This walk has proved to be useful because it links the annealed expectation of a Green function of a random walk in random media to the Green function of a Markov chain. This result is summarized in the following proposition.

\begin{proposition}
\label{kalikow_0}
For $z\in \Z^d$ and $\delta<1$, we have 
\[
{\bf E}_{1-\epsilon}\Bigl[G_{\delta}^{\omega}(0,z)|\mathcal{I}\Bigr]=G_{\delta}^{\widehat{p}_{\delta}^{\epsilon}}(0,z).
\]
\end{proposition}

The proof of this result can be directly adapted from the proof of Proposition 1 in~\cite{Sabot}. We emphasize that in the case $\delta<1$, the uniform ellipticity condition is not needed.

Using the former property we can link the Kalikow's auxiliary random walk to the speed of our RWRE through the following proposition.
\begin{proposition}
\label{speed_green}
For any $0<\epsilon<1-p_c(\Z^d)$, we have
\[
\lim_{\delta \to 1}\frac{\sum_{z\in \Z^d} G_{\delta}^{\widehat{\omega}_{\delta}^{\epsilon}}(0,z)\widehat{d}_{\delta}^{\epsilon}(z)}{\sum_{z\in \Z^d} G_{\delta}^{\widehat{\omega}_{\delta}^{\epsilon}}(0,z)}=\lim_{\delta \to 1} \frac{\ES[X_{\tau_{\delta}}]}{\ES[\tau_{\delta}]}=v_{\ell}(1-\epsilon),
\]
where $\widehat{d}_{\delta}^{\epsilon}(z)=\displaystyle{\sum_{e\in \nu } \widehat{p}_{\delta}^{\epsilon}(z,z+e)e}$.
\end{proposition}

Let $C_{\delta}^{\epsilon}$ be the convex hull of all $\widehat{d}_{\delta}^{\epsilon}(z)$ for $z\in \Z^d$, then an immediate consequence of the previous proposition is the following
\begin{proposition}{\label{accu}}
For $\epsilon>0$ we have that $v_{\ell}(1-\epsilon)$ is an accumulation point of $C_{\delta}^{\epsilon}$ as $\delta$ goes to 1.
\end{proposition}

The proofs of both propositions are contained in the proof of Proposition 2 in~\cite{Sabot} and rely only on the existence of a limiting velocity, which is a consequence of Theorem~\ref{LLN}.

In order to ease notations we will from time to time drop the dependence with respect to $\epsilon$ of the expectation ${\bf E}_{1-\epsilon}[\,\cdot\,]$.

Let us now give a quick sketch of the proof of the continuity of the speed. A way of understanding $\widehat{d}_{\delta}^{\epsilon}(z)$ is to decompose the expression of Kalikow's drift according to the possible configurations at $z$
\begin{align}
\label{drift}
\widehat{d}_{\delta}^{\epsilon}(z)&=\sum_{e\in \nu} \sum_{A \subset \nu} \frac{{\bf E}\Bigl[\1{\mathcal{I}}\1{\C(z)=A}G_{\delta}^{\omega}(0,z) p^{\omega}(z,z+e)  e \Bigr]}{{\bf E}\Bigl[\1{\mathcal{I}}G_{\delta}^{\omega}(0,z)\Bigr]}\\ \nonumber
                                &=\sum_{A \subset \nu,~A\neq \nu} \frac{{\bf E}\Bigl[\1{\mathcal{I}}\1{\C(z)=A}G_{\delta}^{\omega}(0,z)\Bigr]}{{\bf E}\Bigl[\1{\mathcal{I}}G_{\delta}^{\omega}(0,z)\Bigr]}d_{A} \\ \nonumber
                                &=\sum_{A \subset \nu,~A\neq \nu} {\bf P}[\C(z)=A]\frac{{\bf E}\Bigl[\1{\mathcal{I}}G_{\delta}^{\omega}(0,z)|\C(z)=A\Bigr]}{{\bf E}\Bigl[\1{\mathcal{I}}G_{\delta}^{\omega}(0,z)\Bigr]}d_{A},
\end{align}
where $d_{A}=\displaystyle{\sum_{e\in A } p^{A}(e)e}$ is the drift under the configuration $A$.

Since $ {\bf P}[\C(z)=A]\sim\epsilon^{\abs{A}}$ for any $A\in \nu$, if we want to find the limit of $\widehat{d}_{\delta}^{\epsilon}(z)$ as $\epsilon$ goes to $0$, it is natural to conjecture that the term corresponding to $\{\C(z)=\emptyset\}$ is dominant in~(\ref{drift}). For this, recalling the notations from~(\ref{notation_omega}), we may find an upper bound on
\begin{equation}
\label{intro_quotient}
\frac{{\bf E}\Bigl[\1{\mathcal{I}}G_{\delta}^{\omega}(0,z)|\C(z)=A\Bigr]}{{\bf E}\Bigl[\1{\mathcal{I}}G_{\delta}^{\omega}(0,z)\Bigr]}=\frac{{\bf E}\Bigl[\1{\mathcal{I}(\omega^{z,A})}G_{\delta}^{\omega^{z,A}}(0,z)\Bigr]}{{\bf E}\Bigl[\1{\mathcal{I}}G_{\delta}^{\omega}(0,z)\Bigr]},
\end{equation}
for $z\in \Z^d$, $A\in \{0,1\}^{\nu}\setminus \nu$ and $\delta<1$ which is uniform in $z$ for $\delta$ close to 1, to be able to apply Proposition~\ref{accu} and show that $\abs{v_{\ell}(1-\epsilon)-d_{\emptyset}}=O(\epsilon)$.

 Let us show why the terms in~(\ref{intro_quotient}) are upper bounded. It is easy to see that the denominator is greater than $\gamma_1{\bf E}[\1{\mathcal{I}(\omega^{z,1})}G_{\delta}^{\omega^{z,1}}(0,z)]$, so we essentially need to show that closing some edges adjacent to $z$ cannot increase the quantity appearing in~(\ref{intro_quotient}) by a huge amount. That is: for $A\subset \nu $,
\begin{equation}
\label{intro_eq}
{\bf E}\Bigl[\1{\mathcal{I}(\omega^{z,A})}G_{\delta}^{\omega^{z,A}}(0,z)\Bigr]\leq \gamma_2 {\bf E}\Bigl[\1{\mathcal{I}(\omega^{z,1})}G_{\delta}^{\omega^{z,1}}(0,z)\Bigr].
\end{equation}

In order to show that closing edges cannot have such a tremendous effect, let us first remark that the Green function can be written as $G_{\delta}^{\omega}(0,z)=P^{\omega}_0[T_z<\tau_{\delta}]G_{\delta}^{\omega}(z,z)$. When we close some edges we might create a trap, for example a long \lq \lq corridor\rq \rq can be transformed into a \lq\lq dead-end\rq\rq and this effect can, in the quenched setting, increase arbitrarily $G_{\delta}^{\omega}(z,z)$, the number of returns to $z$ .

 The first step is to quantify this effect, we will essentially show in Section~\ref{sect_res} that $G_{\delta}^{\omega^{z,\nu \setminus A}}(z,z)\leq  \gamma_3G_{\delta}^{\omega^{z,1}}(z,z) + L_z(\omega)$ (see Proposition~\ref{insertedge}) where $L_z(\omega)$ is in some sense, to be defined later, a \lq \lq local\rq \rq quantity around $z$ (see Proposition~\ref{threeprop} and Proposition~\ref{perco}). With this random variable we try to quantify how far from $z$ the random walk has to go to find a \lq\lq regular\rq\rq environment without traps where the effect of the modification around $z$ is \lq\lq forgotten\rq\rq. In this upper bound, we may get rid of the term $G_{\delta}^{\omega^{z,1}}(z,z)$ which is, once multiplied by  $\1{\mathcal{I}}P^{\omega}_0[T_z<\tau_{\delta}]\leq \1{\mathcal{I}(\omega^{z,1})}P^{\omega^{z,1}}_0[T_z<\tau_{\delta}]$, of the same type as the terms on the right-hand side of~(\ref{intro_eq}).

The second step is to understand how the \lq \lq local\rq \rq quantity $L_z$ is correlated with the hitting probability. The intuition here is that the hitting probability depends on the environment as a whole but that a very local modification of the environment cannot change tremendously the value of the hitting probability. On a more formal level this corresponds to (see Lemma~\ref{decorrelation}) ${\bf E}[\1{\mathcal{I}}P^{\omega}_0[T_z<\tau_{\delta}] L_z(\omega)]\leq \gamma_4 {\bf E}[\1{\mathcal{I}}P^{\omega}_0[T_z<\tau_{\delta}]] $ where $\gamma_4$ is some moment of $L_z$, which is a sufficient upper bound.

Before turning to the proof, we emphasize that the aim of Section~\ref{sect_res}  and Section~\ref{sect_perco}, is mainly to introduce the so-called \lq\lq local\rq\rq quantities, which is done at the beginning of Section~\ref{sect_res}, and prove some properties on these quantities, see Proposition~\ref{insertedge}, Proposition~\ref{perco} and Proposition~\ref{threeprop}. The corresponding proofs are essentially unrelated to the rest of the paper and may be skipped in a first reading, to concentrate on the actual proof of the continuity which is in Section~\ref{sect_cont}.

\section{Resistance estimates}
\label{sect_res}

In this section we shall introduce some elements of electrical networks theory (see~\cite{LP}) to estimate the variations on the diagonal of the Green function induced by a local modification of the state of the edges around a vertex $x$. Our aim is to show that we can get efficient upper bounds using only the local shape of the environment. 

Let us denote the effective resistance between $x\in \Z^d$ and a subgraph $H'$ of a certain finite graph $H$ by $R^{H}(x\leftrightarrow H')$. Denoting $V(H')$ the vertices of $H'$, it can be defined through Thomson's principle (see~\cite{LP})
\[
R^{H}(x\leftrightarrow H') =\inf \Bigl\{\sum_{e\in H} r(e)\theta^2(e), \theta(\cdot) \text{ is a unit flow from $x$ to } V(H') \Bigr\},
\]
and this infimum is reached for the current flow from $x$ to $V(H')$. Under the environment $\omega$, we will denote the resistance between $x$ and $y$ by $R^{\omega}(x\leftrightarrow y)$. 

For a fixed $\omega \in \Omega$, we add a cemetary point $\Delta$ which is linked to any vertex $x$ of $K_{\infty}(\omega)$ with a conductance such that at $x$ the probability of going to $\Delta$ is $1-\delta$ and denote the associated weighted graph by $\omega(\delta)$. We denote $\pi^{\omega(\delta)}(x)$ the sum of the conductances of edges adjacent to $x$ in $\omega(\delta)$ and we define $R^{\omega(\delta)}(x \leftrightarrow \Delta)$ to be the limit of $R^{\omega(\delta)}(x\leftrightarrow \omega \setminus \omega_n)$ where $\omega_n$ is any increasing exhaustion of subgraphs of $\omega$. We emphasize that the $R^{\omega(\delta)}(x\leftrightarrow \omega \setminus \omega_n)$ is well defined for $n$ large enough since $x\in \omega_n$ for $n$ large enough. In this setting we have,
\begin{equation}
\label{resdelta}
\pi^{\omega(\delta)}(x)=\frac{\pi^{\omega}(x)}{\delta} \text{ and } r^{\omega(\delta)}([x,\Delta])= \frac 1 {\pi^{\omega(\delta)}(x)} \frac 1 {1- \delta} = \frac 1 {\pi^{\omega}(x)}\frac{\delta}{1-\delta}.
\end{equation}

We emphasize that changing the state of an edge $[x,y]$ changes the values of $r^{\omega(\delta)}([x,\Delta])$ and $r^{\omega(\delta)}([y,\Delta])$. It can nevertheless be noted that Rayleigh's monotonicity principle (see~\cite{LP}) is preserved, i.e.~if we increase (resp.~decrease) the resistance of one edge any effective resistance in the graph also increases (resp.~decreases). 

There is no ambiguity to simplify the notations by setting $R^{\omega}(x \leftrightarrow \Delta):=R^{\omega(\delta)}(x \leftrightarrow \Delta)$ for $x\in \Z^d$ and $r^{\omega}(e):=r^{\omega(\delta)}(e)$ for $e$ any edge of $\omega(\delta)$. It is classic (and can be found as an exercice in chapter 2 of~\cite{LP}) that
\begin{lemma}
\label{lemgreen}
For any $\delta<1$, we have
\[
G_{\delta}^{\omega}(x,x)=\pi^{\omega(\delta)}(x) R^{\omega}(x \leftrightarrow \Delta),
\]
for any $\omega\in \Omega_0$, i.e. if there exists a unique infinite cluster.
\end{lemma}

Hence to understand, in a rough sense, how closing edges might increase the number of returns at $z$, we can concentrate on understanding the effect of closing edges on the effective resistance. By Rayleigh's monotonicity principle, given a vertex $x$, the configuration in $A=B^E(x,r)$ which has the lowest resistance between any point and $\Delta$ is the one where all edges are open. Hence, for configurations $B \subset A$, we want to get an upper bound $R^{\omega^{A,B}}(x\leftrightarrow \Delta)$ in terms of $R^{\omega^{A,\emptyset}}(x\leftrightarrow \Delta)$ and of \lq\lq local\rq \rq quantities.

Let us begin with a heuristic description of configurations which are likely to increase strongly the number of returns when we close an edge. There are mainly two situations that can occur (see Figure~1)
\begin{enumerate}
\item the vertex $x\in K_{\infty}(\omega)$ is in a long corridor, which is turned into a \lq\lq dead-end\rq\rq if we close only an edge, hence increasing the number of returns,
\item if closing an edge adjacent to $x$ creates a new finite cluster $K$, the number of returns to $x$ can be tremendously increased. Indeed because of the geometrical killing parameter, when the particle gets stuck in $K$ for a long time it may die (i.e.~go to $\Delta$), hence by closing the edge linking $x$ to $K$, we can remove this escape possibility and increase the number of returns to $x$.
\end{enumerate}

\begin{figure}[htpd]
\centering
\epsfig{file=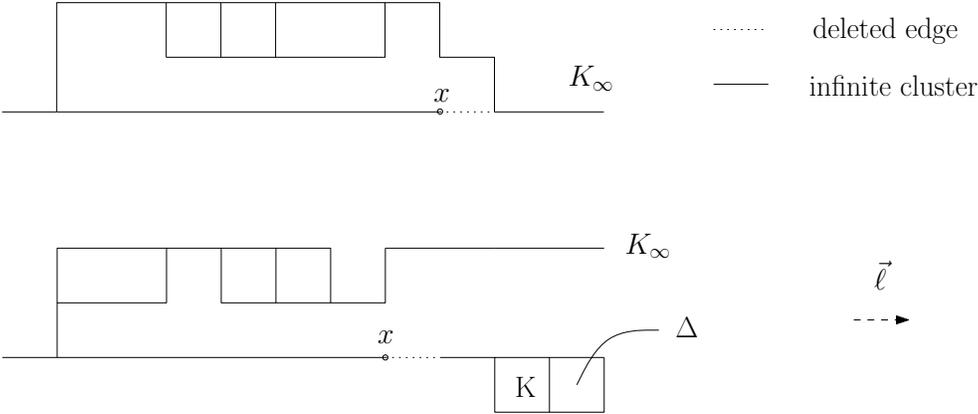, width=13cm}
\caption{Configurations where one deleted edge increases $G_{\delta}(x,x)$}
\end{figure}

We want to find properties of the environment which will quantify how strongly the number of returns will increase for a point in the infinite cluster. In order to find a quantity which controls the effect of the first type of configurations we denote, for any $x\in \Z^d$ and $r\geq 1$, denoting $A=B^E(x,r)$ where $x\in K_{\infty}(\omega)$, we set
\begin{equation}
\label{defL}
M_A(\omega)=\begin{cases}
0 & \text{ if } \forall y\in \partial A, \ y \notin K_{\infty}(\omega^{A,0}), \\
\displaystyle{\max_{y_1,y_2 \in \partial A \cap K_{\infty}(\omega^{A,0})} d_{\omega^{A,0}}(y_1,y_2)} & \text{ otherwise,}
\end{cases}
\end{equation}
which is the maximal distance between vertices of $\partial A\cap K_{\infty}(\omega^{A,0})$ in the infinite cluster of $\omega^{A,0}$. It is important to notice that the notation $K_{\infty}(\omega^{A,0})$ makes sense, since it is classical that ${\bf P}$-a.s.~modifying the states of a finite set of edges does not create multiple infinite clusters.

This quantity will help us give upper bounds on the number of returns to $x$ after having closed some adjacent edges. Indeed even if the \lq\lq best escape way to infinity\rq\rq is closed, $M_A$ tells us in some sense how much more the particle has to struggle to get back onto this good escape route, even though some additional edges are closed.

Let us control the effect the second type of bad configurations has on the expected number of returns to $x\in K_{\infty}(\omega)$. We first want to find out if we are likely to go to $\Delta$ during an excursion into the part we called $K$. For this we introduce a way to measure the size of the biggest finite cluster of $\omega^{A,0}$ which intersects $\partial A$, 
\begin{equation}
\label{defT}
T_A(\omega)=\begin{cases}
                  0 & \text{ if } \forall y\in \partial A,\ y \in K_{\infty}(\omega^{A,0}), \\
                 \displaystyle{\max_{y\in \partial A,\ y\notin K_{\infty}(\omega^{A,0})}} \abs{\partial_E  K_{\omega^{A,0}}(y)} & \text{otherwise,}\end{cases}
                 \end{equation}
 which gives an indication on the time of an excursion into $K$, hence on the probability of going to $\Delta$ during this excursion.

The idea now is to find an alternate place $K'$ close and connected to $x\in K_{\infty}$, from which the particle needs a long time to return to $x$. Thus from this place the particle is likely to go to $\Delta$ before returning to $x$. This means that $K'$ have an effect similar to $K$. This area $K'$ ensures that the number of returns to $x$ cannot be too big even in the case where all the accesses to parts such as $K$ adjacent to $x$ are closed. For this let us denote $\eta\geq 1$ depending on $d$ and $\ell$ such that
\begin{equation}
\label{def_eta}
\text{for all $n\geq 1$,}\qquad e^{2\lambda(\eta-1) n}\geq 3\kappa_1^2 \abs{B(0,n)},
\end{equation}
and $H_A'(\omega)$ the half-space $\{y,\ y\cdot\vec{\ell}\geq x\cdot\vec \ell +\eta T_A(\omega)\}$. From any point of this half-space the particle is very unlikely to return to $x$ in a short time. Indeed to come back from this half-space the particle must go against the drift and for this to happen we have to wait a long amount of time during which the particle is most likely to go to $\Delta$. A relevant quantity to control the effect of the second type of configurations is the distance between $x$ and this half-space, which quantifies the difficulty to reach this half-plane.

In order to define these quantities we need to know the infinite cluster $K_{\infty}$, hence they are not \lq\lq local\rq\rq quantities. Nevertheless we are able to define random variables which are \lq\lq local\rq\rq and fulfill the same functions. 

For $A=B^E(x,r)$, we denote $L_A^1(\omega)$ the smallest integer larger or equal to $r$ such that all $y\in \partial A$ which are connected to $\partial B(x,L_A^1(\omega))$ in $\omega^{A,0}$, are connected to each other using only edges of $B^E(x,L_A^1(\omega))\cap \omega^{A,0}$. 

We always have $L_A^1(\omega)<\infty,$ ${\bf P}_p$-a.s.~by uniqueness of the infinite cluster. Consequently there are two types of vertices in $\partial A$, first those which are not connected to $\partial B(x,L_A^1(\omega))$ in $\omega^{A,0}$ (hence in a finite cluster of $\omega^{A,0}$) and then those which are, the latter being all inter-connected in $B(x,L_A^1(\omega))\cap \omega^{A,0}$.

Set $H_A(\omega)$ to be the half-space $\{y,\ y\cdot\vec{\ell}\geq x\cdot\vec \ell +\eta L_A^1(\omega)\}$ and finally let us define $L_A(\omega)$ the smallest integer larger or equal to $r$ such that 
\begin{enumerate}
\item either $\partial A$ is connected to $H_A(\omega)$ using only edges of $B^E(x,L_A(\omega))\cap \omega^{A,0}$,
\item or $\partial A$ is not connected to $B^E(x,L_A(\omega))$, which can only happen if $\partial A \cap K_{\infty} = \emptyset$,\end{enumerate}
in both cases we can see that 
\begin{equation}
\label{acapk}
\text{on $\partial A \cap K_{\infty} = \emptyset$,} \qquad L_A\leq \min\{n\geq 0,~\partial A \text{ is not connected to } \partial B(x,n)\}
\end{equation}

In order to make the notations lighter we use 
\begin{equation}
\label{notation_L}
L_{z,k}:=L_{B^E(z,k)} \text{ and }L_z=L_{z,1}.
\end{equation}

Using this definition for $L_A$ we get an upper bound for the quantities $M_A$ and $d_{\omega}(x,H_A')$ on the event that $x\in K_{\infty}$ which is the only case we will need to consider. Now we can easily obtain, the proof is left to the reader, the following proposition

\begin{proposition}
\label{threeprop}
For a ball $A=B^E(x,r)$, set $\mathcal{F}_{x,n}$ the $\sigma$-field generated by $\{\omega(e),e\in B^E(x,n)\}$, we have the following
\begin{enumerate}
\item $L_A(\omega)$ does not depend on the state of the edges in $A$,
\item $L_A(\omega)$ is a stopping time with respect to $(\mathcal{F}_{x,n})_{n\geq 0}$, in particular the event $\{L_A(\omega)=k\}$ does not depend on the state of the edges of $B^E(x,k)^c=E(\Z^d)\setminus B^E(x,k)$,
\item $r\leq L_A(\omega)<\infty$, ${\bf P}$-a.s..
\end{enumerate}
\end{proposition}

The second property is one of the two central properties for what we call a \lq \lq local\rq \rq quantity. Recalling the notations~(\ref{notation_omega}) and~(\ref{notation_omega1}), let us prove the following
\begin{proposition}
\label{insertedge}
Set $A=B^E(x,r)$ with $r\geq 1$, $\delta<1$ and $\omega\in \Omega_0$. Suppose that $y\in K_{\infty}(\omega)$ and $\partial A \cap K_{\infty}(\omega) \neq \emptyset$. We have
\[
R^{\omega}(y\leftrightarrow \Delta)\leq 4 R^{\omega^{A,1}}(y \leftrightarrow \Delta) + C_1L_A(\omega)^{C_2} e^{2\lambda (L_A(\omega)-x\cdot\vec \ell)}, 
\]
where $C_1$ and $C_2$ depend only on $d$ and $\ell$.
\end{proposition}

The $4$ appearing is purely arbitrary and could be any constant larger than $1$. Here the correcting term is essentially of the same order as the largest between 
\begin{enumerate}
\item the resistance of paths linking the vertices of $\partial A\cap K_{\infty}(\omega^{A,0})$ inside $B(x,L_A)$,
\item the resistance of paths linking $x$ to $H_A$ inside $B(x,L_A)$.
\end{enumerate}

\begin{proof}
Let us introduce
\[
A^+=B(x,r) \cup \bigcup_{a\in \partial A, a\notin K_{\infty}(\omega^{A,0})}  K_{\omega^{A,0}}(a)\text{ and }A^{+,\delta}=\bigcup_{a\in A^+} \{[a,\Delta]\},
\]
moreover we set
\[
A^-=B(x,r-1) \cup \bigcup_{a\in \partial A, a\notin K_{\infty}(\omega^{A,0})}  K_{\omega^{A,0}}(a)\text{ and }A^{-,\delta}=\bigcup_{a\in A^-} \{[a,\Delta]\}.
\]

Let $\omega_n$ be an exhaustion of $\omega$ and $n_0$ such that $B(x,L_A(\omega))\cap \omega \subset \omega_{n_0}$ and $y$ is connected to $\partial A$ in $\omega_{n_0}$. Set $n\geq n_0$, we denote $\theta(\cdot)$ any unit flow from $y$ to $\omega(\delta) \setminus \omega_n$ using only edges of $\omega_n(\delta)$. By Thomson's principle, we get
\begin{align}
\label{TP}
 & R^{\omega(\delta)}(y\leftrightarrow \omega(\delta) \setminus \omega_n)-R^{\omega^{A,1}(\delta)}(y \leftrightarrow \omega^{A,1}(\delta) \setminus \omega_n^{A,1}) \\ \nonumber
\leq & \sum_{e \in \omega(\delta)} (r^{\omega}(e) \theta(e)^2-r^{\omega^{A,1}}(e) i_0(e)^2),
\end{align}
where $i_0(\cdot)$ denotes the unit current flow from $z$ to $\omega^{A,1}(\delta) \setminus \omega_n^{A,1}$. We want to apply the previous equation with a flow $\theta(\cdot)$ which is close to the current flow $i_0(\cdot)$. Since the latter does not necessarily use only edges of $\omega$ we will need to redirect the part flowing through $A$.

For a vertex $a\in \partial A$, we denote $i_0^A(a)=\sum_{e\in \nu, [a,a+e] \in A} i_0([a,a+e])$ the quantity of current entering $A$ through $a$. Hence we can partition $\partial A$ into
\begin{enumerate}
\item $a_1,\ldots,a_k$ the vertices of $\partial A\cap K_{\infty}(\omega^{A,0})$ such that $i_0^A(a)\geq 0$,
\item $a_{k+1},\ldots,a_l$ the vertices of $\partial A\cap K_{\infty}(\omega^{A,0})$ such that $i_0^A(a)<0$,
\item $a_{l+1},\ldots, a_m$ the vertices of $\partial A\setminus K_{\infty}(\omega^{A,0})$.
\end{enumerate}

Moreover we denote
\[
i_0^+(\Delta)=\sum_{e\in A^{+,\delta}} i_0(e) \text{ and }i_0^-(\Delta)=\sum_{e\in A^{-,\delta}} i_0(e).
\]

Let us first assume $y\in K_{\infty}(\omega^{A,0})$, in particular $y\notin B(x,r-1)$. The following facts are classical (see e.g.~\cite{LP} chapter 2)
\begin{enumerate}
\item for any $e\in E(\Z^d)$, we have $\abs{i_0(e)} \leq 1$,
\item the intensity entering $B(x,r-1)$ is equal to the intensity leaving $B(x,r-1)$, i.e.
\[
\sum_{i\leq k}i_0^A(a_i)=i_0^-(\Delta)-\sum_{j\in[k+1,l]} i_0^A(a_j).
\]
\end{enumerate}

Using the two previous remarks, we see it is possible to find a collection $\nu(i,j)$ with $i\in[1,k]$ and $j\in [k+1,l]\cup \{\Delta\}$ such that
\begin{enumerate}
\item for all $i,j$, we have $\nu(i,j)\in [0,1]$,
\item for all $j\in [k+1,l]$, it holds that $\sum_{i \leq k} \nu(i,j)=-i_0^A(a_j)$,
\item for all $i\in [1,k]$, we have $\sum_{j \in [k+1,l]\cup\{\Delta\}} \nu(i,j)=i_0^A(a_i)$,
\item it holds that $\sum_{i \leq k} \nu(i,\Delta)=i_0^-(\Delta)$,
\end{enumerate}
which should be seen as a way of matching the flow entering and leaving $B(x,r-1)$.

For $i\in[1,k]$ and $j\in [k+1,l]$ we denote $\vec{\mathcal{P}}(i,j)$ one of the directed paths between $a_i$ and $a_j$ in $\omega^{A,0}\cap B^E(x,L_A^1(\omega))$. Let $\vec{\mathcal Q}$ be one of the directed paths from $\partial A$ to $H_A(\omega)$ in $\omega^{A,0}\cap B^E(x,L_A(\omega))$ and let us denote $a_{j_0}$ its starting point and $h_1$ its endpoint. The existence of those paths is ensured by the definitions of $L_A^1$, $L_A$ and $H_A$ and the fact that . Since $\partial A \cap K_{\infty} \neq \emptyset$, then all vertices of $\partial A$ connected to $\partial B^E(x,L_A(\omega))$ are in $K_{\infty}$. Hence  we necessarily have $j_0\leq l$ and $E(A^+)\cap \vec{\mathcal Q} =\emptyset$.

Finally let us notice that the values of the resistances $r^{\omega}([a,\Delta])$ and $r^{\omega^{A,1}}([a,\Delta])$ might differ for $a\in \partial A$ so that to get further simplifications in~(\ref{TP}), it is convenient to redirect the flow using these edges too. We introduce the unique flow (see Figure~2) defined by
\[
\theta_0(\vec e)=\begin{cases} 0  &\text{if } e\in A^{+,\delta},  \\
                               0  &\text{if } e\in E(A^+), \\
                                                  i_0(\vec e)+i_0^+(\Delta)  &\text{if } \vec e = [h_1,\Delta], \\
 i_0(\vec e)+\sum_{i \leq k, j\in [k+1,l]} \nu(i,j) \1{\vec e \in \vec{\mathcal{P}}(i,j)}& \\
~ +\sum_{i\leq k} \nu(i,\Delta) \1{\vec e \in \vec{\mathcal{P}}(i,j_0)} +i_0^+(\Delta)  \1{\vec e \in \vec{\mathcal{Q}}} & \\
~+\sum_{i\leq l} i_0([a_i,\Delta]) \1{\vec e \in \vec{\mathcal{P}}(i,j_0)} &\text{else.}
\end{cases}
\]

\begin{figure}[htpd]
\centering
\epsfig{file=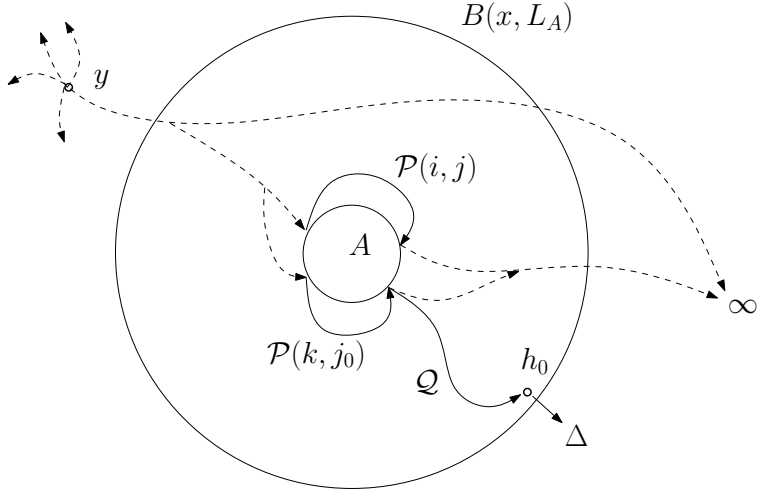, width=10cm}
\caption{The flow $\theta_0(\cdot)$ in the case where $y\in K_{\infty}(\omega^{A,0})$}
\end{figure}

In words, we could say that we have redirected parts of $i_0(\cdot)$ in order to go around $A$ and the flow going from $A$ to $\Delta$ is first sent to $a_{j_0}$, then to $h_1$ and finally to $\Delta$. We have the following properties
\begin{enumerate}
\item $\theta_0(\cdot)$ is a unit flow from $y$ to $\omega(\delta) \setminus \omega_n$,
\item $\abs{\theta_0(e)}\leq 5\abs{\partial A}^2$ for all $e\in E(\Z^d)$,
\item $\theta_0(\cdot)$ coincides with $i_0(\cdot)$ except on the edges of $E(A^+)$, $A^{+,\delta}$, $\mathcal{Q}$, $[h_1,\Delta]$ and $\mathcal{P}(i,j)$ for $i,j\leq k+l$ ,
\item $r^{\omega}(\cdot)$ coincides with $r^{\omega^{A,1}}(\cdot)$ except on the edges of $E(A^+)$ and $A^{+,\delta}$.
\end{enumerate}

Hence recalling~(\ref{TP}) we get
\begin{align} \label{apps_thompson}
 & R^{\omega(\delta)}(y\leftrightarrow \omega(\delta) \setminus \omega_n)-R^{\omega^{A,1}(\delta)}(y \leftrightarrow \omega^{A,1}(\delta) \setminus \omega_n^{A,1}) \\ \nonumber
\leq & \sum_{e\in  \mathcal{P}(i,j) \cup \mathcal{Q}} r^{\omega}(e) (\theta_0(e)^2-i_0(e)^2)+r^{\omega}([h_1,\Delta])(i_0^+(\Delta)+i_0([h_1,\Delta]))^2 \\\nonumber
&\qquad \qquad -\sum_{e\in A^{+,\delta}}r^{\omega}(e)i_0(e)^2-r^{\omega}([h_1,\Delta])i_0([h_1,\Delta])^2 \\\nonumber
\leq &50 \rho_d\abs{\partial A}^6 L_A^{d} e^{2\lambda (L_A-x\cdot\vec \ell)}+r^{\omega}([h_1,\Delta])(i_0^+(\Delta)+i_0([h_1,\Delta]))^2\\\nonumber
&\qquad \qquad -\sum_{e\in A^{+,\delta}}r^{\omega}(e)i_0(e)^2-r^{\omega}([h_1,\Delta])i_0([h_1,\Delta])^2,
\end{align}
where we used that $r^{\omega}(e)\leq e^{2\lambda(L_A-x\cdot\vec \ell)}$ for $e \in \mathcal{P}(i,j) \cup \mathcal{Q}$ and that there are at most $(1+\abs{\partial A}^2)\rho_dL_A^{d}\leq 2 \rho_d\abs{\partial A}^2 L_A^d$ such edges in those paths. These properties being a consequence of the fact that $\mathcal{P}(i,j)$ and $\mathcal{Q}$ are contained in $B^E(x,L_A^1(\omega))$.

Since $\abs{\partial A}\leq \rho_d r^d\leq \rho_d L_A^d$ by the third property of Proposition~\ref{threeprop}, the first term is of the form announced in the proposition, the remaining issue is to control the remaining terms. First, we have by definition
\[
\sum_{e\in A^{+,\delta}}r^{\omega}(e)i_0(e)^2 = \sum_{a\in A^{+}}r^{\omega}([a,\Delta])i_0([a,\Delta])^2,
\]
and since for $a\in K_{\infty}(\omega)$, we have using~(\ref{resdelta}) and~(\ref{kap1}) that
\[
\kappa_1 e^{-2\lambda a\cdot\vec \ell}\frac{\delta}{1-\delta} \geq r^{\omega}([a,\Delta])\geq \frac 1{\kappa_1} e^{-2\lambda a\cdot\vec \ell}\frac{\delta}{1-\delta}.
\]

Furthermore, since for any $a\in A^+$ we have $a\cdot\vec \ell\leq  x\cdot\vec \ell +L_A^1$ and since $h_1\in H_A(\omega)$ we have $h_1\cdot\vec \ell \geq x\cdot \vec \ell + \eta L_A^1 \geq a \cdot \vec \ell +(\eta-1)L_A^1$ so that the definition of $\eta$ at~(\ref{def_eta}) yields 
\[
\frac 1 {\kappa_1}e^{-2\lambda a\cdot\vec \ell}\geq \frac 1 {\kappa_1}e^{-2\lambda h_1\cdot\vec \ell}e^{2\lambda (\eta-1) L_A^1(\omega)} \geq 3\kappa_1\abs{B(0,L_A^1)} e^{-2\lambda h_1\cdot\vec \ell}.
\]

Since $A^+$ is contained in $B(x,L_A^1(\omega))$, the two previous equations yield
\[
r^{\omega}([a,\Delta])\geq 3\kappa_1 \abs{A^+}e^{-2\lambda h_1\cdot\vec \ell}\frac{\delta}{1-\delta}\geq 3\abs{A^+} r^{\omega}([h_1,\Delta]),
\]
and hence
\begin{align}
\label{correct1}
&\sum_{e\in A^{+,\delta}}r^{\omega}(e)i_0(e)^2 +r^{\omega}([h_1,\Delta])i_0([h_1,\Delta])^2 \\ \nonumber
 \geq & r^{\omega}([h_1,\Delta])\Bigl(i_0([h_1,\Delta])^2+3\abs{A^+} \sum_{e\in A^{+,\delta}}i_0(e)^2 \Bigr)\\  \nonumber
 \geq & r^{\omega}([h_1,\Delta])(i_0([h_1,\Delta])^2+3 i_0^+(\Delta)^2),
\end{align}
where we used Cauchy-Schwarz in the last inequality.

Hence the remaining terms in~(\ref{apps_thompson}) verifies if $i_0^+(\Delta)\leq i_0([h_1,\Delta])$
\begin{align*}
&r^{\omega}([h_1,\Delta])(i_0^+(\Delta)+i_0([h_1,\Delta]))^2-\sum_{e\in A^{+,\delta}}r^{\omega}(e)i_0(e)^2-r^{\omega}([h_1,\Delta])i_0([h_1,\Delta])^2 \\
\leq&r^{\omega}([h_1,\Delta])(2i_0([h_1,\Delta]))^2-r^{\omega}([h_1,\Delta])i_0([h_1,\Delta])^2\\
\leq& 3 r^{\omega}([h_1,\Delta])i_0([h_1,\Delta])^2\leq 3 R^{\omega^{A,1}(\delta)}(y \leftrightarrow \omega^{A,1}(\delta) \setminus \omega_n^{A,1}),
\end{align*}
or if $i_0^+(\Delta)> i_0([h_1,\Delta])$, we obtain using~(\ref{correct1})
\begin{align*}
&r^{\omega}([h_1,\Delta])(i_0^+(\Delta)+i_0([h_1,\Delta]))^2-\sum_{e\in A^{+,\delta}}r^{\omega}(e)i_0(e)^2-r^{\omega}([h_1,\Delta])i_0([h_1,\Delta])^2 \\
\leq &r^{\omega}([h_1,\Delta]) \bigl(i_0([h_1,\Delta])^2+2i_0^+(\Delta)i_0([h_1,\Delta]) +i_0^+(\Delta)^2 -(i_0([h_1,\Delta])^2+3 i_0^+(\Delta)^2 )\bigr)\\ \leq& 0.
\end{align*}

In any case we get
\begin{align*}
&r^{\omega}([h_1,\Delta])(i_0^+(\Delta)+i_0([h_1,\Delta]))^2-\sum_{e\in A^{+,\delta}}r^{\omega}(e)i_0(e)^2-r^{\omega}([h_1,\Delta])i_0([h_1,\Delta])^2 \\
\leq &3 R^{\omega^{A,1}(\delta)}(y \leftrightarrow \omega^{A,1}(\delta) \setminus \omega_n^{A,1}),
\end{align*}
 and so we have shown that 
\[
R^{\omega(\delta)}(y\leftrightarrow \omega(\delta) \setminus \omega_n)-4R^{\omega^{A,1}(\delta)}(y \leftrightarrow \omega^{A,1}(\delta) \setminus \omega_n^{A,1})   \leq 50\rho_d^7 L_A^{7d} e^{2\lambda (L_A-x\cdot\vec \ell)},
\]
and letting $n$ go to infinity yields the result in the case where $y\in K_{\infty}(\omega^{A,0})$.

Let us come back to the remaining case where $y \in K_{\infty}(\omega)\setminus K_{\infty}(\omega^{A,0})$. Keeping the same notations, we see that obviously there exists $j_1\leq l$ such that $a_{j_1}\in K_{\infty}(\omega^{A,0})$ which is connected in $\omega$ to $y$ using only vertices of $A^+$, let us denote $\vec{\mathcal{R}}$ path connecting $a_{j_1}$ and $y$ in $\omega \cap A^+$. 

Introducing the flow (see Figure~3) defined by
\[
\theta_0'(\vec e)=\begin{cases}1&\text{ if } \vec e\in \vec{\mathcal{R}}, \\
                                                   0  &\text{ if } e\in A^{+,\delta} \cup E(A^+)\setminus\mathcal{R},  \\
                                                  i_0^+(\Delta)+i_0( [h_1,\Delta])  &\text{ if } \vec e = [h_1,\Delta] \\
 i_0(\vec e)+\sum_{j\leq l} i_0^A(a_j) \1{\vec e \in \vec{\mathcal{P}}(j_1,j)} & \\
~+\sum_{i\leq l} i_0([a_i,\Delta]) \1{\vec e \in \vec{\mathcal{P}}(i,j_0)} \\
~+i_0^-(\Delta) \1{\vec e \in \vec{\mathcal{P}}(j_1,j_0)} +i_0^+(\Delta)  \1{\vec e \in \vec{\mathcal{Q}}} &\text{ else,}
\end{cases}
\]
for which we can get the same properties as for $\theta_0(\cdot)$.

\begin{figure}[htpd]
\centering
\epsfig{file=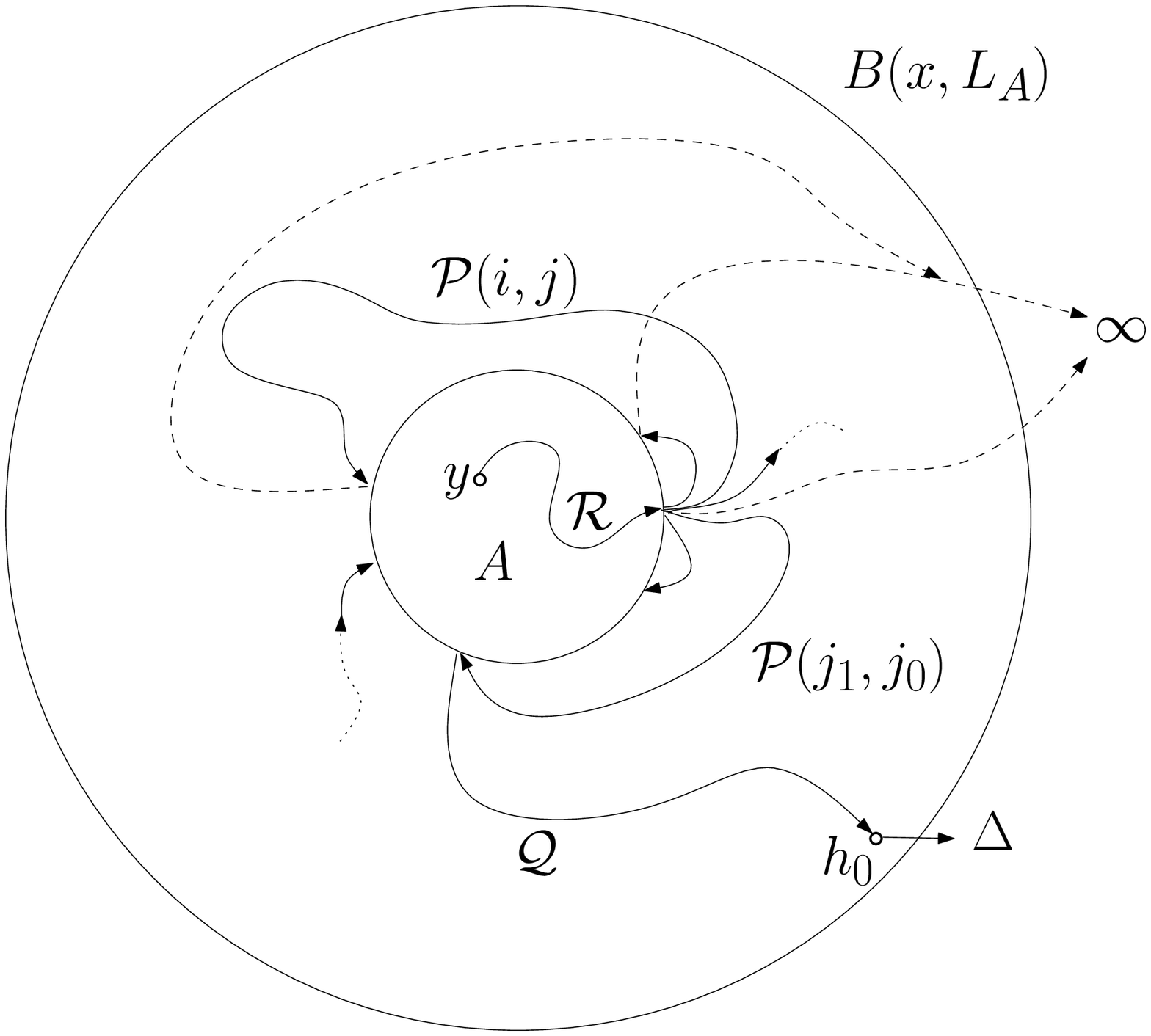, width=8cm}
\caption{The flow $\theta_0'(\cdot)$ in the case where $y \in K_{\infty}(\omega)\setminus K_{\infty}(\omega^{A,0})$}
\end{figure}

 The computation of the energy of $\theta_0'(\cdot)$ is essentially similar to that of $\theta_0(\cdot)$ and we get
\begin{align*}
 & R^{\omega(\delta)}(y\leftrightarrow \omega(\delta) \setminus \omega_n)-R^{\omega^{A,1}(\delta)}(y\leftrightarrow \omega^{A,1}(\delta) \setminus \omega_n^{A,1}) \\ 
\leq &\gamma_1 L_A^{7d} e^{2\lambda (L_A-x\cdot\vec \ell)}+\sum_{e\in \mathcal{R}} r^{\omega}(e) +3R^{\omega^{A,1}(\delta)}(y\leftrightarrow \omega^{A,1}(\delta) \setminus \omega_n^{A,1})\\
\leq &\gamma_2 L_A^{7d} e^{2\lambda (L_A-x\cdot\vec \ell)}+3R^{\omega^{A,1}(\delta)}(y\leftrightarrow \omega^{A,1}(\delta) \setminus \omega_n^{A,1}) ,
\end{align*}
since $\abs{\mathcal{R}}\leq \abs{A^+}\leq \rho_dL_A^d$ and $ r^{\omega}(e)\leq e^{2\lambda(L_A-x\cdot\vec \ell)}$ for $e\in \mathcal{R}$. The result follows.
\end{proof}

We set for $x,y \in \Z^d$ and $Z\subset \Z^d$,
\begin{equation}
\label{fct_green_gen}
G_{\delta,Z}(x,y)=E_x^{\omega}\Bigl[\sum_{k=0}^{T_{Z}} \delta^k \1{X_k=y}\Bigr],
\end{equation}
and similarly we can define $R^{\omega}(x \leftrightarrow Z\cup\Delta)$ to be the limit of $R^{\omega(\delta)}(x\leftrightarrow Z \cup \{\omega(\delta) \setminus \omega_n\})$ where $\omega_n$ is any increasing exhaustion of subgraphs of $\omega$. We can get 
\begin{lemma}
\label{lemgreen2}
For any $\delta<1$, we have for $x,z\in \Z^d$,
\[
G_{\delta,\{z\}}^{\omega}(x,x)=\pi^{\omega(\delta)}(x) R^{\omega}(x \leftrightarrow z\cup\Delta).
\]
\end{lemma}

In a way similar to the proof of Proposition~\ref{insertedge}, we get
\begin{proposition}
\label{insertedge1}
Set $A=B^E(x,r)$, $\delta<1$, $z\in \Z^d$ and $\omega\in \Omega_0$. Suppose that $y,z\in K_{\infty}(\omega)$ and $\partial A \cap K_{\infty}(\omega) \neq \emptyset$. We have
\[
R^{\omega}(y\leftrightarrow z\cup\Delta)\leq 4R^{\omega^{A,1}}(y \leftrightarrow z\cup\Delta) + C_1L_A(\omega)^{C_2} e^{2\lambda (L_A(\omega)-x\cdot\vec \ell)}, 
\]
where $C_1$ and $C_2$ depend only on $d$ and $\ell$.
\end{proposition}

We assume without loss of generality the constants are the same as Proposition~\ref{insertedge}.

\begin{proof}
This time let us denote $i_0(\cdot)$ by the unit current flow from $y$ to $z\cup \{ \omega(\delta)\setminus \omega_n\}$.

The case where $z\in K_{\infty}(\omega^{A,0})$ can be treated using the same flows as in the proof of Proposition~\ref{insertedge} and we will not give further details.

In order to treat the case where $z\notin K_{\infty}(\omega^{A,0})$ and $y\in K_{\infty}(\omega^{A,0})$. We keep the notations of the previous proof for the partition $(a_i)_{1\leq i\leq m}$ of $\partial A$, $i_0^+(\Delta)$, $i_0^-(\Delta)$, $A^+$ and $A^{+,\delta}$. We set
\[
i_0^z=\sum_{e\in \nu} i_0([z+e,z]).
\]

Similarly, we can find a familly $\nu(i,j)$ with $i\in[1,k]$ and $j\in [k+1,l]\cup \{\Delta\}\cup \{z\}$ such that
\begin{enumerate}
\item for all $i,j$, we have $\nu(i,j)\in [0,1]$,
\item for all $j\in [k+1,l]$, it holds that $\sum_{i \leq k} \nu(i,j)=-i_0^A(a_j)$,
\item we have $\sum_{i \leq k} \nu(i,\Delta)=i_0^-(\Delta)$,
\item it holds that $\sum_{i \leq k} \nu(i,z)=i_0^z$,
\item for all $i\in [1,k]$ we have $\sum_{j \in [k+1,l]\cup\{\Delta\}\cup \{z\}} \nu(i,j)=i_0^A(a_i)$.
\end{enumerate}

We use again the same notations for $\mathcal{P}(i,j)$, $\mathcal{Q}$, $j_0$ and $h_1$ and add an index $j_2\leq l$ such that $z$ is connected inside $A^+$ to $a_{j_2}$ and $\vec{\mathcal{S}}$ the corresponding directed path. We set
\[
\theta_0(\vec e)=\begin{cases}i_0^z(\vec e)&\text{ if } \vec e\in \vec{\mathcal{S}}, \\
                              0  &\text{ if } e\in A^{+,\delta}\cup E(A^+)\setminus\mathcal{S},  \\
                              i_0(\Delta)+i_0( [h_1,\Delta])  &\text{ if } \vec e = [h_1,\Delta] \\
 i_0(\vec e)+i_0^+(\Delta)  \1{\vec e \in \vec{\mathcal{Q}}}& \\
\qquad +\sum_{i \leq k, j\in [k+1,l]} \nu(i,j) \1{\vec e \in \vec{\mathcal{P}}(i,j)}& \\
\qquad +\sum_{i\leq k} \nu(i,\Delta) \1{\vec e \in \vec{\mathcal{P}}(i,j_0)} & \\
\qquad +\sum_{i\leq k} \nu(i,z) \1{\vec e \in \vec{\mathcal{P}}(i,j_2)}& \\
\qquad +\sum_{i\leq l} i_0([a_i,\Delta]) \1{\vec e \in \vec{\mathcal{P}}(i,j_0)} &\text{else,}
\end{cases}
\]
which is similar to the flow considered in Proposition~\ref{insertedge} except that the flow naturally supposed to escape at $z$ is, instead of entering $A$, redirected to $a_{j_2}$ and from there sent to $z$. Using this flow with Thomson's principle yields similar computations as in Proposition~\ref{insertedge} and thus we obtain a similar result.

The case where $z\notin K_{\infty}(\omega^{A,0})$ and $0\notin K_{\infty}(\omega^{A,0})$ can be easily adapted from the proof above and the second part of the proof of Proposition~\ref{insertedge}.

\end{proof}

\section{Percolation estimate}
\label{sect_perco}

We want to give tail estimates on $L_A^1$ and $L_A$ for some ball $A=B(x,r)$. More precisely we want to show for any $C>0$, we have ${\bf E}_{1-\epsilon}[e^{C L_A}]<\infty$ for $\epsilon$ small enough, the exact statement can be found in Proposition~\ref{perco}. Let us recall the definitions of $M_A$ and $T_A$ at~(\ref{defL}) and~(\ref{defT}). We see that all vertices of $\partial A$ are either in finite clusters of $\omega^{A,0}$ (which are included in $B(x,r+T_A)$) or in the infinite cluster and all those last ones are inter-connected in $B(x,r+M_A)$. Hence we get  by the two remarks above~(\ref{notation_L}) that
\begin{equation}
\label{upperb_LA1}
L_A^1 \leq r+\max (M_A,T_A).
\end{equation}

Recalling the definitions of $L_A$ and $H_A$ below~(\ref{def_eta}), our overall strategy for deriving an upper-bound on the tail of $L_A$ in the case $\partial A \cap K_{\infty} \neq \emptyset$ is the following: if $L_A$ is large then there are two cases.
\begin{enumerate}
\item The random variable $L_A^1$ is large. This means by~(\ref{upperb_LA1}) that either $M_A$ or $T_A$ is large. The random variable $M_A$ cannot be large with high probability, since the distance in the percolation cluster cannot be much larger than the distance in $\Z^d$ (see Lemma~\ref{corfirstpassage}) and neither can $T_A$ since finite cluster are small in the supercritical regime (see Lemma~\ref{trapsize}).
\item Otherwise the distance from $x$ to $H_A$ in the percolation cluster is large even though it is not large in $\Z^d$. Once again this is unlikely, in fact for technical reasons it appears to be easier to show that the distance to $H_A \cap \mathbb{T}_x$ is small, where $\mathbb{T}_x$ is some two-dimensional cone. For this we will need Lemma~\ref{orientcone}.
\end{enumerate}

The following is fairly classical result about first passage percolation with a minor twist due to the conditioning on the edges in $A$, we will outline the main idea of the proof while skipping a topological argument. To get a fully detailled proof of the topological argument, we refer the reader to the proof of Theorem 1.4 in~\cite{GM}.

\begin{lemma}
\label{firstpassage}
Set $A=B^E(x,r)$ and $y,z\in \Z^d\setminus B(x,r-1)$, there exists a non-increasing function $\alpha_1: [0,1] \to [0,1]$ such that for  $\epsilon<\epsilon_1$ and $n\in \N$,
\[
{\bf P}_{1-\epsilon}\Bigl[y\stackrel{\omega^{A,0}}{\leftrightarrow} z, d_{\omega^{A,0}}(y,z) \geq n+2d_{\Z^d\setminus B(x,r-1)}(y,z)\Bigr] \leq  2\alpha_1(\epsilon)^{n+d_{\Z^d\setminus B(x,r-1)}(y,z)},
\]
and
\[
{\bf P}_{1-\epsilon}\Bigl[y\stackrel{\omega^{A,0}}{\leftrightarrow} z, d_{\omega^{A,0}}(y,z) \geq n+2d_{\Z^d}(y,z)+4dr\Bigr] \leq  2\alpha_1(\epsilon)^{n+d_{\Z^d}(y,z)},
\]
where $\epsilon_1$ and $\alpha_1(\cdot)$ depend only on $d$ and $\displaystyle{\lim_{\epsilon \to 0} \alpha_1(\epsilon) =0}$.
\end{lemma}

The main tool needed to prove Lemma~\ref{firstpassage} is a result of stochastic domination from~\cite{LSS}, next we will state a simplified version of this result which appeared as Proposition 2.1~in~\cite{GM}. We recall that a familly $\{Y_u, u\in \Z^d\}$ of random variables is said to be $k$-dependent if for every $a\in \Z^d$, $Y_a$ is independent of $\{Y_u: \abs{\abs{u-a}}_1 \geq k \}$.

\begin{proposition}
\label{domstoch}
 Let $d,k$ be positive integers. There exists a non-decreasing function $\alpha': [0,1] \to [0,1]$ satisfying $\lim_{\tau \to 1} \alpha'(\tau)=1$, such that the following holds: if $Y=\{Y_u, u \in \Z^d\}$ is a $k$-dependent familly of random variables taking values in $\{0,1\}$ satisfying
\[
\text{for all $u \in \Z^d$,} \qquad\ P(Y_u=1) \geq \tau,
\]
then $P_Y  \succ  (\alpha'(\tau)\delta_1+(1-\alpha'(\tau)) \delta_0)^{\otimes \Z^d}$, where \lq\lq  $\succ$\rq\rq means stochasticaly dominated.
\end{proposition}

Two vertices $u,v$ are $*$-neighbours if $\abs{\abs{u-v}}_{\infty}= 1$, this topology naturally induces a notion of $*$-connected component on vertices.

Let us say that a vertex $u\in \Z^d$ is $\omega^{A}$-wired if all edges $[s,t]\in E(\Z^d)$ with $\abs{\abs{u-s}}_{\infty}\leq 1$ and $\abs{\abs{u-t}}_{\infty}\leq 1$ are open in $\omega^{A,1}$ (recall that $A=B^E(x,r)$), otherwise it is called $\omega^{A}$-unwired. 

We say that a vertex $u\in \Z^d\setminus B(x,r-1)$ is $\omega^{A}$-strongly-wired, if all $y\in \Z^d\setminus B(x,r-1)$ such that $\abs{\abs{u-y}}_{\infty}\leq 2$ are $\omega^{A}$-wired, otherwise $u$ is called $\omega^{A}$-weakly-wired. It is plain that $\1{u \text{ is $\omega^A$-strongly-wired}}$ are $\gamma_1$-dependent $\{0,1\}$-valued random variables where $\gamma_1$ depends only on $d$. We can thus use Proposition~\ref{domstoch} with this familly of random variables since we have
\[
\text{for all $u \in \Z^d$,}\qquad {\bf P}_{1-\epsilon}[\1{u \text{ is $\omega^A$-strongly-wired}}=1] \geq (1-\epsilon)^{\gamma_1},
\]
and that $\displaystyle{\lim_{\epsilon \to 0} (1-\epsilon)^{\gamma_1} =1}$. This yields a function $\alpha'(\cdot)$ which solely depends on $d$.

Let us start the proof of Lemma~\ref{firstpassage}

\begin{proof}
Let $\gamma$ be one of the shortest paths in $\Z^d\setminus B(x,r-1)$ connecting $y$ to $z$. For $u\in \Z^d\setminus B(x,r-1)$, we define $V(u)(\omega^{A})$ to be the $*$-connected component of the $\omega^{A}$-unwired vertices of $u$ and
\[
V(\omega^{A})= \bigcup_{u\in \gamma} V(u)(\omega^{A}).
\]

Since $y$ and $z$ are connected in $\omega^{A,0}$, a topological argument (see Section 3 of~\cite{GM} for details) proves there is an $\omega^{A,0}$-open path $\mathcal{P}$ from $y$ to $z$ using only vertices in $\gamma \cup(V(\omega^{A,0}) + \{-2,-1,0,1,2\}^d)$. On the event $d_{\omega^{A,0}}(y,z) \geq n+ 2d_{\Z^d \setminus B(x,r-1)}(y,z)$, this path $\mathcal{P}$ has $m\geq n+2d_{\Z^d \setminus B(x,r-1)}(y,z)+1$ vertices and all vertices which are not in $\gamma$ are $\omega^{A}$-weakly-wired thus there are at least $m-d_{\Z^d \setminus B(x,r-1)}(y,z)-1$ of them .

Since there are at most $(2d)^k$ paths of length $k$ in $\Z^d\setminus B(x,r-1)$ we get, through a straightforward counting argument, that 
\begin{align*}&
{\bf P}_{1-\epsilon}\Bigl[y \stackrel{\omega^{A,0}}{\leftrightarrow} z, d_{\omega^{A,0}}(y,z)\geq n+d_{\Z^d \setminus B(x,r-1)}(y,z)\Bigr]\\
 \leq & \sum_{m\geq n+2d_{\Z^d\setminus B(x,r-1)}(y,z)+1}(2d)^m (1-\alpha'((1-\epsilon)^{\gamma_1}))^{m-d_{\Z^d\setminus B(x,r-1)}(y,z)-1}\\
 \leq & \sum_{m\geq n+2d_{\Z^d\setminus B(x,r-1)}(y,z)+1}((2d)^3 (1-\alpha'((1-\epsilon)^{\gamma_1})))^{m-d_{\Z^d\setminus B(x,r-1)}(y,z)-1},
\end{align*}
where $\alpha'(\cdot)$ is given by Proposition~\ref{domstoch} and verifies $\displaystyle{\lim_{\epsilon \to 0} 1-\alpha'((1-\epsilon)^{\gamma_1}) =0}$. Thus, the first part of the proposition is verified with $\alpha_1(\epsilon):=1-\alpha'((1-\epsilon)^{\gamma_1})$ and $\epsilon_1$ small enough so that $1-\alpha'((1-\epsilon_1)^{\gamma_1})\leq (2d)^{-3}/2$.

The second part is a consequence of
\[
d(y,z)\leq d_{\Z^d\setminus B(x,r-1)}(y,z) \leq d(y,z) +2dr.
\]
\end{proof}

An easy consequence is the following tail estimate on $M_A$ (defined at~(\ref{defL})).
\begin{lemma}
\label{corfirstpassage}
Set $A=B^E(x,r)$, there exists a non-increasing function $\alpha_1: [0,1] \to [0,1]$ such that for  $\epsilon<\epsilon_1$ and $n\in \N$,
\[
{\bf P}_{1-\epsilon}[M_A \geq n+4dr] \leq  C_3r^{2d} \alpha_1(\epsilon)^{n},
\]
where $C_3$, $\epsilon_1$ and $\alpha_1(\cdot)$ depend only on $d$ and $\displaystyle{\lim_{\epsilon \to 0} \alpha_1(\epsilon) =0}$. The function $\alpha_1(\cdot)$ is the same as in Lemma~\ref{firstpassage}.
\end{lemma}
\begin{proof}
Since $\abs{\partial A} \leq \rho_d r^d$, we have
\begin{align*}
& {\bf P}_{1-\epsilon}[M_A \geq n+4dr]\\
\leq & (\rho_d r^d)^2 \max_{a,b\in \partial A} {\bf P}_{1-\epsilon}\Bigl[a\stackrel{\omega^{A,0}}{\leftrightarrow} b, d_{\omega^{A,0}}(a,b) \geq n+4dr\Bigr] \leq  \gamma_1 r^{2d}\alpha_1(\epsilon)^{n},
\end{align*}
where we used Lemma~\ref{firstpassage} since $d_{\Z^d\setminus B(x,r-1)}(a,b)\leq 4dr$ for $a,b\in \partial A$.
\end{proof}

 A set of $n$ edges $F$ disconnecting $x$ from infinity in $\Z^d$, that is any infinite simple path starting from $x$ uses an edge of $F$, is called a Peierls' contour of size $n$. Asymptotics on the number $\mu_n$ of Peierls' contours of size $n$ have been intensively studied, see for example~\cite{LM}, we will use the following bound proved in~\cite{Ruelle} and cited in~\cite{LM},
\[
\mu_n \leq 3^n.
\]

This enables us to prove the following tail estimate on $T_A$ (defined at~(\ref{defT})).
\begin{lemma}
\label{trapsize}
Set $A=B^E(x,r)$, there exists a non-increasing function $\alpha_2: [0,1] \to [0,1]$ such that for  $\epsilon<\epsilon_2$v
\[
{\bf P}_{1-\epsilon}[T_A \geq n] \leq  C_4 r^d\alpha_2(\epsilon)^{n},
\]
where $C_4$, $\epsilon_2$ and $\alpha_2(\cdot)$ depend only on $d$ and $\displaystyle{\lim_{\epsilon \to 0} \alpha_2(\epsilon) =0}$.
\end{lemma}

\begin{proof}
First we notice that for $n\geq 1$,
\[
{\bf P}_{1-\epsilon}[T_A \geq n] \leq \rho_d r^d \max_{a\in \partial A} {\bf P}_{1-\epsilon}\Bigl[ a \notin K_{\infty}(\omega^{A,0}),\abs{\partial_E K^{\omega^{A,0}}(a)} \geq n\Bigr].
\]

 For any $a\in \partial A$ such that $a\notin K_{\infty}(\omega^{A,0})$, we have that $\partial_E K^{\omega^{A,0}}(a)$ is a finite Peierls' contour of size $\abs{\partial_E K^{\omega^{A,0}}(a)}$ surrounding $a$ which has to be closed in $\omega^{A,0}$.

 Because $A$ is a ball at least half of the edges of $\partial_E K^{\omega^{A,0}}(a)$ have to be closed in $\omega$ as well. Indeed, take $[x,y]\in A\cap \partial_E K^{\omega^{A,0}}(a)$ and denote $x$ its endpoint in $K^{\omega^{A,0}}(a)$, then by definition of a Peierls' contour there is $i\geq 0$ such that $[x+i(x-y),x+(i+1)(x-y)]$ is in  $\partial_E K^{\omega^{A,0}}(a)$, let $i_0(x,y)$ be the smallest one (see Figure~4 for a drawing).
 
  If $[x+i_0(x,y)(x-y),x+(i_0(x,y)+1)(x-y)]$ were in $A$, since $A$ is a ball all edges between $x$ and $x+(i_0(x,y)+1)(x-y)$ would too. This would imply that all edges adjacent to $x$ are in $A$ but since $x$ is connected to $a$ in $\omega^{A,0}$, we have $a=x$. This is a contradiction since $a\in \partial A$ and all edges adjacent to $x=a$ are in $A$. Hence $[x+i_0(x,y)(x-y),x+(i_0(x,y)+1)(x-y)]\notin A$. 

\begin{figure}[htpd]
\centering
\epsfig{file=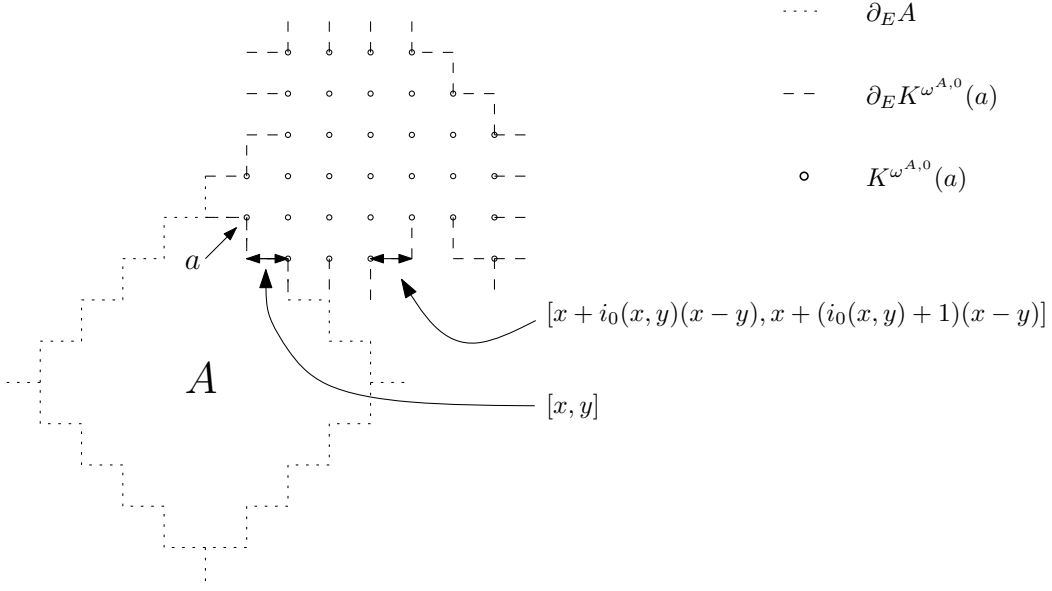, width=14cm}
\caption{Half of the edges of $\partial_E K^{\omega^{A,0}}(a)$ have to be closed in $\omega$}
\end{figure}

Hence
\[
\psi: \begin{cases} A\cap \partial_E K^{\omega^{A,0}}(a) & \to   \partial_E K^{\omega^{A,0}}(a) \setminus A \\
                                   [x,y] & \mapsto  [x+i_0(x,y)(x-y),x+(i_0(x,y)+1)(x-y)] ,\end{cases}
\]
is an injection so that at least half of the edges of $\partial_E K^{\omega^{A,0}}(a)$ are indeed closed in $\omega$. Let us denote $m=\abs{\partial_E K^{\omega^{A,0}}(a)}\geq n$. Then we know that at least $\lceil m/2 \rceil$ edges of $\partial_E K^{\omega^{A,0}}(a)$ are closed. There are at most ${ m \choose \lceil m/2 \rceil }\leq \gamma_1 2^m$ ways of choosing those edges, thus we get for any $a\in \partial  A$
\begin{align*}
{\bf P}_{1-\epsilon}\Bigl[ a \notin K_{\infty}(\omega^{A,0}), \abs{\partial_E K^{\omega^{A,0}}(a)}\geq n\Bigr] & \leq \sum_{m \geq n} { m \choose \lceil m/2 \rceil } \mu_n\epsilon^{m/2} \\
& \leq \gamma_1\sum_{m \geq n} 6^m \epsilon^{m/2}\leq \gamma_2 (\epsilon^{1/2})^{n}.
\end{align*}
\end{proof}

A direct consequence of~(\ref{upperb_LA1}), Lemma~\ref{corfirstpassage} and Lemma~\ref{trapsize} is the following tail estimate on $L_A^1$, defined below~(\ref{def_eta})
\begin{lemma}
\label{tail_l1}
Set $A=B^E(x,r)$, there exists a non-increasing function $\alpha_3: [0,1] \to [0,1]$ such that for  $\epsilon<\epsilon_3$ and $n\in \N$,
\[
{\bf P}_{1-\epsilon}[L_A^1 \geq n+C_5r] \leq  C_6 r^{2d}\alpha_3(\epsilon)^{n},
\]
where $C_5$, $C_6$, $\epsilon_3$ and $\alpha_3(\cdot)$ depend only on $d$ and $\displaystyle{\lim_{\epsilon \to 0} \alpha_3(\epsilon) =0}$.
\end{lemma}

Recalling the definition of $H_A$ above~(\ref{notation_L}), let us introduce
\begin{equation}
\label{notation_LA}
L_A'(\omega)=\begin{cases}
\infty & \text{ if } \forall y\in \partial A,\ y\notin K_{\infty}(\omega^{A,0}) \\
d_{\omega^{A,0}}(\partial A,H_A(\omega)) & \text{otherwise,} \end{cases}
\end{equation}
it is plain that $L_A \leq L_A'+r$.

We need one more estimate before turning to the tail of $L_A'$ (and thus $L_A$). Define the cone $\mathbb{T}=\{ae^{(1)}+be^{(2)}, 0 \leq b \leq a/2 \text{ for } a,b \in \N \}$. It is a standard percolation result that $p_c(\mathbb{T})<1$  (see Section 11.5 of~\cite{Grimmett}) and well-known that the infinite cluster is unique. We denote $K_{\infty}^{\mathbb{T}}(\omega)$ the unique infinite cluster of $\mathbb{T}$ induced by the percolation $\omega$, provided $\epsilon<1-p_c(\mathbb{T})$.

\begin{lemma}
\label{orientcone}
There exists a non-increasing function $\alpha_4: [0,1] \to [0,1]$ so that for  $\epsilon<\epsilon_4$ and $n\in \N$,
\[
{\bf P}_{1-\epsilon}\Bigl[d_{\mathbb{T}}(0,K_{\infty}^{\mathbb{T}}(\omega)) \geq 1+n\Bigl] \leq  C_7\alpha_4(\epsilon)^{n},
\]
where $C_7$, $\alpha_4(\cdot)$ depend only on $d$ and $\displaystyle{\lim_{\epsilon \to 0} \alpha_4(\epsilon) =0}$.
\end{lemma}

\begin{proof}
Choose $\epsilon<1-p_c(\mathbb{T})$, so that $K_{\infty}^{\mathbb{T}}(\omega)$ is well defined almost surely. We emphasize that the following reasoning is in essence two dimensional, so we are allowed to use duality arguments (see~\cite{Grimmett}, Section 11.2). We recall that an edge of the dual lattice (i.e. of $\Z^2+(1/2,1/2)$) is called closed when it crosses a closed edge of the original lattice.

The idea is to show that if $d_{\mathbb{T}}(0,K_{\infty}^{\mathbb{T}}(\omega))=n+1$, there is a closed interface, in the dual lattice, separating the infinite cluster from $0$ in $\mathbb{T}$. The length of this interface grows linearly with $n$ and so this event has very small probability.

If $d_{\mathbb{T}}(0,K_{\infty}^{\mathbb{T}}(\omega))=n+1$, then let $x$ be a point for which this distance is reached, $x$ belongs to a set of at most $\gamma_1n$ points. Consider an edge $e=[x,y]$ where $d_{\mathbb{T}}(0,y)=n$, implying that $y\notin K_{\infty}^{\mathbb{T}}(\omega)$. Let $e'$ denote the corresponding edge in the dual lattice (see Figure~5). From each endpoint of $e'$ there is a closed path in the dual lattice, such that the union of those path and $e'$ separates $K_{\infty}^{\mathbb{T}}$ from $0$. The union of $e'$ and the longest one of these paths has to be at least of length $n/\gamma_2$. Thus there has to be a closed path $\mathcal{P}$ in the dual lattice of length $m\geq n/\gamma_2$ starting from one of the endpoints of $e'$ and exiting $\mathbb{T}$.

\begin{figure}[htpd]
\centering
\epsfig{file=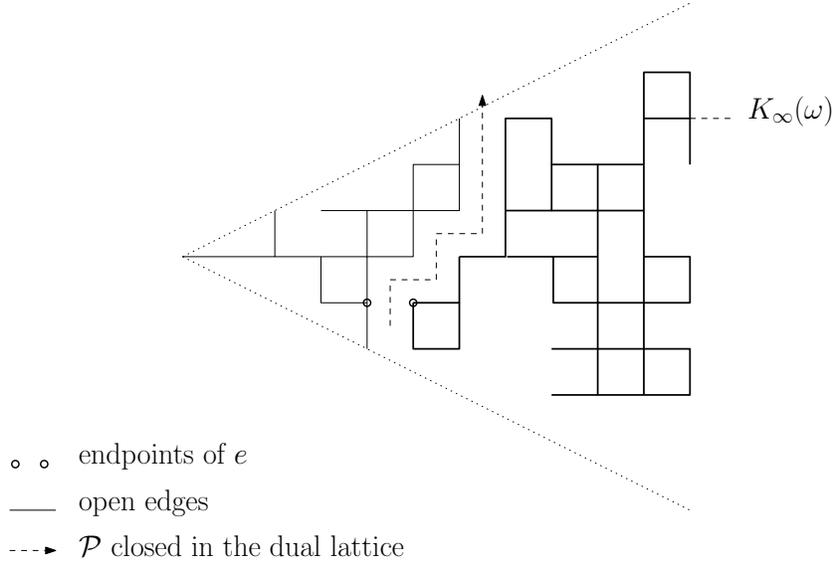, width=11cm}
\caption{The closed path in the dual lattice}
\end{figure}

 Thus since there are at most $4^m$ paths of length $m$ starting at a given point, we get for $\epsilon$ small enough
\[
{\bf P}_{1-\epsilon}\Bigl[d_{\mathbb{T}}(0,K_{\infty}^{\mathbb{T}}(\omega))= 1+n\Bigr] \leq 2\gamma_1n \sum_{m\geq \frac n{\gamma_2}} 4^m \epsilon^m\leq \gamma_3n (4\epsilon)^{n/\gamma_2},
\]
and the result follows since for $n$ large enough $\gamma_3 n \leq 2^{n}$ we have for $n$ large enough
\[
{\bf P}_{1-\epsilon}\Bigl[d_{\mathbb{T}}(0,K_{\infty}^{\mathbb{T}}(\omega))\geq 1+n\Bigr]\leq \sum_{m\geq n } \gamma_3m (4\epsilon)^{m/\gamma_2} \leq \gamma_4(2^{\gamma_5}\epsilon^{1/\gamma_2})^n.
\]
\end{proof}

Now we turn to the study of the asymptotics of $L_A$.

\begin{proposition}
\label{perco}
Set $A=B^E(x,r)$, there exists a non-increasing function $\alpha: [0,1] \to [0,1]$ so that for  $\epsilon<\epsilon_0$ and $n\in \N$,
\[
{\bf P}_{1-\epsilon}[L_A\geq n+C_8r] \leq  C_9r^{2d}n\alpha(\epsilon)^{n},
\]
where $C_8$, $C_9$, $\epsilon_0$ and $\alpha(\cdot)$ depend only on $d$ and $\ell$ and $\displaystyle{\lim_{\epsilon \to 0} \alpha(\epsilon) =0}$.
\end{proposition}

\begin{proof}
Let us notice that two cases emerge. First let us consider that we are on the event $\{\partial A\cap K_{\infty}= \emptyset\}$ in which case we have by~(\ref{acapk})
\[
L_A(\omega) \leq \min\{n \geq 0,~\partial A \text{ is not connected to } \partial B(x,n)\} \leq r+T_A(\omega),
\]
hence because of Lemma~\ref{trapsize} we have for $C_8>1$
\begin{equation}
\label{UB_LA_part1}
{\bf P}[\partial A \cap K_{\infty}=\emptyset,~L_A\geq n+C_8r] \leq C_4r^{2d} \alpha_2(\epsilon)^n.
\end{equation}

We are now interested in the case where $\partial A\cap K_{\infty}\neq \emptyset$. It is sufficent to give an upper bound for $L_A'$ (defined at~(\ref{notation_LA})) since $L_A\leq L_A'+r$. Set $\epsilon<\epsilon_1\wedge\epsilon_2\wedge\epsilon_3\wedge \epsilon_4$, we notice using Lemma~\ref{tail_l1} that
\begin{align}
\label{pstep1}
  &{\bf P}_{1-\epsilon}[\partial A\cap K_{\infty}\neq \emptyset,~L_A' \geq n+(C_8-1)r] \\ \nonumber
 \leq &{\bf P}_{1-\epsilon}[L_A^1 \geq  n/(8\eta d)+C_5r]\\ \nonumber & \qquad + {\bf P}_{1-\epsilon}[\partial A\cap K_{\infty}\neq \emptyset,~L_A^1 \leq  n/(8\eta d)+C_5r,~L_A' \geq  n+(C_8-1)r] \\ \nonumber                                                        \leq& {\bf P}_{1-\epsilon}[\partial A\cap K_{\infty}\neq \emptyset,~L_A^1 \leq  n/(8\eta d)+C_5r,~L_A' \geq  n+(C_8-1)r]\\ \nonumber &\qquad \qquad \qquad \qquad \qquad \qquad+C_6 r^{2d}\alpha_1(\epsilon)^{ n/(8\eta d)}.
\end{align}

We denote $h_m^{x}$ the half-space $\{y,\ y\cdot\vec{\ell}\geq x\cdot\vec \ell + m\}$, we have
\begin{align}
\label{pstep2}
&{\bf P}_{1-\epsilon}[\partial A\cap K_{\infty}\neq \emptyset,~L_A^1 \leq  n/(8\eta d)+C_5r,~L_A' \geq  n+(C_8-1)r] \\ 
\nonumber\leq & {\bf P}_{1-\epsilon}\Bigl[\partial A\cap K_{\infty}\neq \emptyset,~d_{\omega^{A,0}}(\partial A,h_{n/(8  d)+\eta C_5 r}^{x} )\geq  n+(C_8-1)r \Bigr] \\ 
\nonumber \leq & \abs{\partial A} \max_{y\in \partial A}{\bf P}_{1-\epsilon}\Bigl[y \stackrel{\omega^{A,0}}{\leftrightarrow} \infty,~d_{\omega^{A,0}}(y,h_{n/(8 d)+\gamma_1 r}^{x}) \geq n+(C_8-1)r\Bigr].
\end{align}

Set $y\in \partial A$ and let us denote $\gamma_2$ a constant which will be chosen large enough. Using the uniqueness of the infinite cluster we get 
\begin{align}
\label{ppstep1}
 & {\bf P}_{1-\epsilon}\Bigl[d_{\Z^d}\Bigl(y,K_{\infty}(\omega^{A,0})\cap h_{n/(8 d)+\gamma_1 r}^{x}\cap \{y+ \mathbb{T}\}\Bigr)  \geq n/2+\gamma_2 r\Bigr] \\ \nonumber
 \leq & {\bf P}_{1-\epsilon}\Bigl[d_{y+\mathbb{T}}\Bigl(y,K_{\infty}^{y+\mathbb{T}}(\omega^{A,0})\cap h_{n/(8 d)+\gamma_1 r}^{x}\Bigr) \geq  n/2+\gamma_2 r\Bigr] \\ \nonumber
 \leq & {\bf P}_{1-\epsilon}\Bigl[d_{y+\mathbb{T}}\Bigl(y,K_{\infty}^{y+\mathbb{T}}(\omega)\cap h_{n/(8 d)+\gamma_1 r}^{x}\Bigr) \geq  n/2+\gamma_2 r\Bigr],
\end{align}
where we have to suppose that $\gamma_2\geq 2$ for the last inequality. Indeed then $d_{y+\mathbb{T}}(y,K_{\infty}^{y+\mathbb{T}}(\omega))=d_{y+\mathbb{T}}(y,K_{\infty}^{y+\mathbb{T}}(\omega^{A,0}))$ on the event $\{d_{y+\mathbb{T}}(y,K_{\infty}^{y+\mathbb{T}}(\omega))\geq \gamma_2 r\}$ since the distance to the infinite cluster is greater than the radius of $A$.

Moreover since $e^{(1)}\cdot\vec \ell \geq 1/\sqrt d$, we notice that 
\[
\max_{z\in  \partial h_m^x \cap \mathbb{T}}d_{z+\mathbb{T}}(x,y)\leq 2\sqrt{d}m.
\]

 Applying this for $m= n/(8 d)+\gamma_1 r$, we get that 
\begin{align}
\label{ppstep2}
 & {\bf P}_{1-\epsilon}\Bigl[d_{y+\mathbb{T}}\Bigl(y,K_{\infty}^{y+\mathbb{T}}(\omega)\cap h_{n/(8  d)+\gamma_1 r}^{x}\Bigr) \geq  n/2+\gamma_2 r\Bigr]\\ \nonumber
= & {\bf P}_{1-\epsilon}\Bigl[d_{y+\mathbb{T}}\Bigl(y,K_{\infty}^{y+\mathbb{T}}(\omega) \Bigr) \geq  n/2+\gamma_2 r\Bigr],
\end{align}
where $\gamma_2$ is large enough so that $2\sqrt d(n/(8 d)+ \gamma_1 r) \leq n/2+\gamma_2 r$. Indeed, if $d_{y+\mathbb{T}}\bigl(y,K_{\infty}^{y+\mathbb{T}}(\omega) \bigr) \geq  n/2+\gamma_2 r$ then $K_{\infty}^{y+\mathbb{T}}(\omega) \subset h_{n/(8  d)+\gamma_1 r}^{y}$.

The equations~(\ref{ppstep1}) and~(\ref{ppstep2}) used with Lemma~\ref{orientcone} yield that for $\gamma_3$ large enough and any $y\in \partial A$,
\[
{\bf P}_{1-\epsilon}\Bigl[d_{\Z^d}\Bigl(y,K_{\infty}(\omega^{A,0})\cap h_{n/(8d)+\gamma_1 r}^{x}\cap \{y+ \mathbb{T}\}\Bigr) \geq  n/2+\gamma_3 r\Bigr]  \leq \gamma_4 \alpha_4(\epsilon)^{n/2}.
\]

If we use Lemma~\ref{firstpassage} and the previous inequality, for $C_8$ large enough so that $n+(C_8-1)r\geq 2(n/2+\gamma_3 r)+4dr$,
\begin{align}
\label{pstep3}
       & {\bf P}_{1-\epsilon}\Bigl[y \in K_{\infty}^{\omega^{A,0}},~d_{\omega^{A,0}}\Bigl(y,h_{n/(8 d)+\gamma_1 r}^{x}\Bigr) \geq  n+C_8r\Bigr] \\ \nonumber
 \leq  & {\bf P}_{1-\epsilon}\Bigl[d_{\Z^d}\Bigl(y,K_{\infty}(\omega^{A,0})\cap h_{n/(8 d)+\gamma_1 r}^{x}\cap \{y+ \mathbb{T}\}\Bigr) \geq  n/2+\gamma_3 r\Bigr]\\ \nonumber
 +     & \sum_{z \in \partial  B_{\Z^d}(y,  \lceil n/2+\gamma_3 r \rceil)\cap \{y+ \mathbb{T}\}} {\bf P}_{1-\epsilon}\Bigl[z \stackrel{\omega^{A,0}}{\leftrightarrow} y, d_{\omega^{A,0}}(z,y) \geq 2d(y,z)+4dr\Bigr] \\\nonumber
 \leq &\gamma_4 \alpha_4(\epsilon)^{n/2} +\gamma_5( n+\gamma_3 r)  \alpha_1(\epsilon)^{ n/2} \leq \gamma_6 r n \alpha_5(\epsilon)^{n},
\end{align}
where $ \epsilon<\epsilon_5$ depends only on $d$ and $\ell$ for some $\alpha_5(\cdot)$ such that $\displaystyle{\lim_{\epsilon \to 0} \alpha_5(\epsilon) =0}$. 

Adding up~(\ref{pstep1}), (\ref{pstep2}) and~(\ref{pstep3}) we get
\begin{align*}
{\bf P}_{1-\epsilon}[\partial A\cap K_{\infty}\neq \emptyset, L_A' \geq n+C_8r] &\leq \gamma_7 nr^{2d}(\alpha_1(\epsilon)^{ n/(8\eta d)}+ \alpha_5(\epsilon)^{n}) \\
                      &\leq \gamma_8 nr^{2d} \alpha(\epsilon)^{n},
\end{align*}
where $\alpha(\epsilon):= \alpha_1(\epsilon)^{ 1/(8 \eta d)}+\alpha_5(\epsilon)$. As we have $\displaystyle{\lim_{\epsilon \to 0} \alpha(\epsilon) =0}$, this last equation and~(\ref{UB_LA_part1}) completes the proof of Proposition~\ref{perco}.
\end{proof}

Essentially by replacing $(\Z^d,E(\Z^d))$ by $(\Z^d,E(\Z^d\setminus [z,z+e]))$ and $\omega$ by $\omega^{z,e}$ (resp. $\omega^{(z,2)=e}$) along with some minor modifications we obtain

\begin{proposition}
\label{perco2}
Set $A=B^E(x,r)$, $z\in \Z^d$ and $e\in \nu$, there exists a non-increasing function $\alpha: [0,1] \to [0,1]$ so that for  $\epsilon<\epsilon_0$ and $n\in \N$,
\[
{\bf P}_{1-\epsilon}[ L_A(\omega^{z,e})\geq  n+C_{8}r] \leq  C_{9}r^{2d}n\alpha(\epsilon)^{n},
\]
and
\[
{\bf P}_{1-\epsilon}[ L_A(\omega^{(z,2)=e})\geq  n+C_{8}r] \leq  C_{9}r^{2d}n\alpha(\epsilon)^{n},
\]
where $C_{8}$, $C_{9}$, $\epsilon_0$ and $\alpha(\cdot)$ depend only on $d$ and $\ell$ and $\displaystyle{\lim_{\epsilon \to 0} \alpha(\epsilon) =0}$.
\end{proposition}

Also by changing $\omega$ by $\omega^{z,\emptyset}$ (resp. $\omega^{(z,2)=\emptyset}$) we can obtain

\begin{proposition}
\label{perco3}
Set $A=B^E(x,r)$, $z\in \Z^d$, there exists a non-increasing function $\alpha: [0,1] \to [0,1]$ so that for  $\epsilon<\epsilon_0$ and $n\in \N$,
\[
{\bf P}_{1-\epsilon}[ L_A(\omega^{z,\emptyset})\geq  n+C_{8}r] \leq  C_{9}r^{2d}n\alpha(\epsilon)^{n},
\]
and
\[
{\bf P}_{1-\epsilon}[ L_A(\omega^{(z,2)=\emptyset})\geq  n+C_{8}r] \leq  C_{9}r^{2d}n\alpha(\epsilon)^{n},
\]
where $C_{8}$, $C_{9}$, $\epsilon_0$ and $\alpha(\cdot)$ depend only on $d$ and $\ell$ and $\displaystyle{\lim_{\epsilon \to 0} \alpha(\epsilon) =0}$.
\end{proposition}

Here we assume without loss of generality that the constants are the same as in Proposition~\ref{perco}.

\section{Continuity of the speed at high density}
\label{sect_cont}

We now have the necessary tools to study the central quantities which appeared in~(\ref{intro_quotient}).
\begin{proposition}
\label{upperb}
For $0<\epsilon<\epsilon_5$, $A\subset \nu$, $A\neq \nu$ and $\delta \geq  1/2$
\[
\frac{{\bf E}\Bigl[\1{\mathcal{I}}G_{\delta}^{\omega}(0,z)|\C(z)=A\Bigr]}{{\bf E}_{1-\epsilon}\Bigl[\1{\mathcal{I}}G_{\delta}^{\omega}(0,z)\Bigr]}= \frac{{\bf E}_{1-\epsilon}\Bigl[\1{\mathcal{I}(\omega^{z,A})}G_{\delta}^{\omega^{z,A}}(0,z)\Bigr]}{{\bf E}_{1-\epsilon}\Bigl[\1{\mathcal{I}}G_{\delta}^{\omega}(0,z)\Bigr]}<C,
\]
where $C$ and $\epsilon_5$ depend only on $\ell$ and $d$ .
\end{proposition}

This section is devoted to the proof of this proposition. We have
\begin{align*}
{\bf E}\Bigl[\1{\mathcal{I}}G_{\delta}^{\omega}(0,z)\Bigr] &\geq {\bf E}\Bigl[\1{\C(z)=\emptyset}\1{\mathcal{I}}G_{\delta}^{\omega}(0,z)\Bigr]\\
                                          & = {\bf E}\Bigl[\1{\C(z)=\emptyset}\1{\mathcal{I}(\omega^{z,\emptyset})}G_{\delta}^{\omega^{z,\emptyset}}(0,z)\Bigr] \\
                                          &= {\bf P}[\C(z)=\emptyset]{\bf E}\Bigl[\1{\mathcal{I}(\omega^{z,\emptyset})}G_{\delta}^{\omega^{z,\emptyset}}(0,z)\Bigr].
\end{align*}

For $\epsilon< 1/4\leq 1-p_c(d)$, we have $ {\bf P}[\C(z)=\emptyset]>\gamma_1>0$ for $\gamma_1$ independent of $\epsilon$, so that
\begin{equation}
\label{denom0}
 {\bf E}\Bigl[\1{\mathcal{I}}G_{\delta}^{\omega}(0,z)\Bigr] \geq \gamma_1 {\bf E}\Bigl[\1{\mathcal{I}(\omega^{z,\emptyset})}G_{\delta}^{\omega^{z,\emptyset}}(0,z)\Bigr].
\end{equation}

Now we want a similar upper bound for the numerator of Proposition~\ref{upperb}. Let $A \subset \nu$, $A\neq\nu$, then by~(\ref{kap1}) and~(\ref{resdelta}) we obtain
\begin{equation}
\label{res_green}
\frac 1 {\kappa_1}e^{2\lambda z\cdot\vec \ell}\frac 1{\delta}\leq\pi^{\omega^{z,A}(\delta)}(z)\leq \kappa_1 e^{2\lambda z\cdot\vec \ell}\frac 1{\delta}.
\end{equation}

This equation combined with Lemma~\ref{lemgreen} yields
\begin{align}
\label{condit}
 &{\bf E}\Bigl[\1{\mathcal{I}(\omega^{z,A})}G_{\delta}^{\omega^{z,A}}(0,z)\Bigr]\\\nonumber
  \leq &  \frac {\kappa_1 e^{2\lambda z\cdot\vec \ell}}{\delta} {\bf E}\Bigl[\1{\mathcal{I}(\omega^{z,A})}P_0^{\omega^{z,A}}[T_z<\tau_{\delta}]R^{\omega^{z,A}}(z \leftrightarrow \Delta)\Bigr].
\end{align}

If $z\notin K_{\infty}(\omega^{z,A})$ then $P_0^{\omega^{z,A}}[T_z<\tau_{\delta}]=0$. Otherwise we can apply Proposition~\ref{insertedge} to get 
\begin{equation}
\label{apply_lem}
R^{\omega^{z,A}}_{\delta}(z \leftrightarrow \Delta) \leq 4R^{\omega^{z,\emptyset}}_{\delta}(z \leftrightarrow \Delta) + C_1 L_z(\omega)^{C_2}e^{2\lambda(L_z(\omega)-z\cdot\vec \ell)},
\end{equation}
where we used notations from~(\ref{notation_L}). 

Moreover we notice that $P_0^{\omega^{z,A}}[T_z<\tau_{\delta}]\leq P_0^{\omega^{z,\emptyset}}[T_z<\tau_{\delta}]$ and $\1{\mathcal{I}(\omega^{z,A})}\leq \1{\mathcal{I}(\omega^{z,\emptyset})}$. Then inserting~(\ref{apply_lem}) into~(\ref{condit}), using Lemma~\ref{lemgreen} and~(\ref{res_green}) we get since $\delta \geq 1/2$
\begin{align}
\label{inter}
&{\bf E}_{1-\epsilon}\Bigl[\1{\mathcal{I}(\omega^{z,A})}G_{\delta}^{\omega^{z,A}}(0,z)\Bigr] \leq 4\kappa_1^2 {\bf E}\Bigl[\1{\mathcal{I}(\omega^{z,\emptyset})}G_{\delta}^{\omega^{z,\emptyset}}(0,z)\Bigr] \\\nonumber
 & \qquad \qquad \qquad \qquad  +2C_1\kappa_1{\bf E}\Bigl[\1{\mathcal{I}(\omega^{z,\emptyset})}P_0^{\omega^{z,\emptyset}}[T_z<\tau_{\delta}]L_z(\omega)^{C_2}e^{2\lambda L_z(\omega)}\Bigr].
 \end{align}
 
 Now we want to prove that the even though hitting probabilities depend on the whole environment their correlation with \lq \lq local\rq \rq quantities are weak in some sense. Let us now make explicit the two properties which are crucial for what we call \lq\lq local quantity\rq\rq (such as $L_z$) which are \begin{enumerate}
\item the second property of Proposition~\ref{threeprop},
\item the existence of arbitrarily large exponential moments for $\epsilon$ small enough, such as those obtained in Proposition~\ref{perco}.
\end{enumerate}

 We obtain the following decorrelation lemma.
\begin{lemma}
\label{decorrelation}
Set $\delta \geq 1/2$, then
\begin{align*}
 & {\bf E}\Bigl[\1{\mathcal{I}(\omega^{z,\emptyset})}P_0^{\omega^{z,\emptyset}}[T_z<\tau_{\delta}]L_z(\omega)^{C_2}e^{2\lambda L_z(\omega)}\Bigr]\\
 \leq & C_{10} {\bf E}\Bigl[\1{\mathcal{I}(\omega^{z,\emptyset})}P_0^{\omega^{z,\emptyset}}[T_z<\tau_{\delta}]\Bigr] {\bf E}\Bigl[L_z(\omega)^{C_{11}}e^{C_{12}L_z(\omega)}\Bigr],
 \end{align*}
 where $C_{10}$, $C_{11}$ and $C_{12}$ depend only on $d$ and $\ell$.
\end{lemma}

Let us prove this lemma.

\begin{proof}
First let us notice that the third property in Proposition~\ref{threeprop} implies that $L_z$ is finite. Set $k\in \N^*$, recall that the event $\{L_z=k\}$ depends only on edges in $B^E(z,k)$ by the second property of Proposition~\ref{threeprop}. 

We have $ \1{\mathcal{I}(\omega^{z,\emptyset})}P_0^{\omega^{z,\emptyset}}[T_z<\tau_{\delta}]\leq \1{\partial B(z,k) \leftrightarrow \infty}P_0^{\omega^{z,\emptyset}}[T_z<\tau_{\delta}]$. Assume first that $0 \notin B(z,k)$,

\begin{align}
\label{step1}
 &{\bf E}\Bigl[\1{\mathcal{I}(\omega^{z,\emptyset})}P_0^{\omega^{z,\emptyset}}[T_z<\tau_{\delta}]L_z(\omega)^{C_2}e^{2\lambda L_z(\omega)}|L_z=k\Bigr]\\\nonumber
  =&k^{C_2}e^{2\lambda k}{\bf E}\Bigl[\1{\partial B(z,k) \leftrightarrow \infty}P_0^{\omega^{z,\emptyset}}[T_z<\tau_{\delta}]\mid L_z=k\Bigr] \\\nonumber
               \leq & \rho_d k^{\gamma_1}e^{2\lambda k}{\bf E}\Bigl[\1{\partial B(z,k) \leftrightarrow \infty}\max_{x\in \partial B(z,k)}P_0^{\omega}[T_x<\tau_{\delta}, T_x=T_{\partial B(z,k)}]\Bigr],
\end{align}
indeed $\abs{ \partial B(z,k) } \leq \rho_d k^{d}$, here we implicitely used that $0\notin B(z,k)$. Now the integrand of the last term does not depend on the configuration of the edges in $B^E(z,k)$, which allowed us to get rid of the conditioning by the second property of Proposition~\ref{threeprop}.

We denote $x_0(\omega)$ a vertex of $\partial B(z,k)$ connected in $\omega$ to infinity without using edges of $B^E(z,k)$ and accordingly we introduce $\{ a\Leftrightarrow b \} $ the event that $a$ is connected in $\omega$ to $b$ using no edges of $B^E(z,k)$. Again we point out that the random variable $x_0(\omega)$ is measurable with respect to $\{\omega(e), e \notin B^E(z,k)\}$.

In case there are multiple choices in the definition of the random variable $x_0(\omega)$, we pick one of the choices according to some predetermined order on the vertices of $\Z^d$. In case $x_0(\omega)$ is not properly defined, i.e.~when $\partial B(z,k)$ is not connected to infinity, we set $x_0(\omega)=z$. With this definition we have $\{x_0 \Leftrightarrow \infty \}=\{\partial B(z,k) \leftrightarrow \infty\}$.

Let us set $x_1(\omega)$ the point for which the maximum in the last line of~(\ref{step1}) is achieved, this random point also depends only on the set of configurations in $E(\Z^d) \setminus B^E(z,k)$, the same is true for $P_0^{\omega}[T_{x_0}<\tau_{\delta}, T_{x_0}=T_{\partial B(z,k)}]$. Once again, if there are multiple choices in the definition of $x_1(\omega)$, we pick one of the choices according to some predetermined order on the vertices of $\Z^d$.

The definition of $x_1$ implies that 
\[
x_1(\omega)\Leftrightarrow 0 \text{ if } \max_{x\in \partial B(z,k)}P_0^{\omega}[T_x<\tau_{\delta}, T_x=T_{\partial B(z,k)}]>0.
\]

Now let $\mathcal{P}_0$ be a path of $k$ edges in $\Z^d$ between $z$ and $x_0$ and $\mathcal{P}_1$ a path of $k$ edges in $\Z^d$ between $z$ and $x_1$, which are not necessarily disjoint. As those paths are contained in $B^E(z,k)$, we get
\begin{align}
\label{step2}
& {\bf E}\Bigl[\1{\partial B(z,k) \leftrightarrow \infty} \1{x_1\Leftrightarrow 0} P_0^{\omega}[T_{x_1}<\tau_{\delta}, T_{x_1}=T_{\partial B(z,k)}]\Bigr] \\\nonumber
 = & {\bf E}\Bigl[\1{x_0\Leftrightarrow \infty} \1{x_1\Leftrightarrow 0} P_0^{\omega}[T_{x_1}<\tau_{\delta}, T_{x_1}=T_{\partial B(z,k)}]|\mathcal{P}_0 \cup \mathcal{P}_1 \in \omega \Bigr] \\ \nonumber
 \leq  &  \frac 1 {{\bf P}[\mathcal{P}_0 \cup \mathcal{P}_1\in \omega]} {\bf E}\Bigl[\1{\mathcal{P}_0 \cup \mathcal{P}_1 \in \omega} \1{x_0\Leftrightarrow \infty}\1{x_1\Leftrightarrow 0}  P_0^{\omega}[T_{x_1}<\tau_{\delta}]\Bigr].
\end{align}

Then we see that since we have $\epsilon < 1/2$ by assumption~(\ref{ass_eps})
\begin{equation}
\label{step3}
{\bf P}[\mathcal{P}_0 \cup \mathcal{P}_1 \in \omega] \geq (1-\epsilon)^{2k}\geq \frac 1 {4^k}.
\end{equation}

 Moreover, on the event $\mathcal{P}_0 \in \omega$, Markov's property yields
\begin{equation}
\label{step35}
(\delta \kappa_0)^k P_0^{\omega}[T_{x_1}<\tau_{\delta}]\leq  P_0^{\omega}[T_{z}<\tau_{\delta}].
\end{equation}

Since $\delta \geq 1/2$,
\begin{align}
\label{step4}
&{\bf E}\Bigl[\1{\mathcal{P}_0 \cup \mathcal{P}_1 \in \omega} \1{x_0\Leftrightarrow \infty} \1{x_1\Leftrightarrow 0} P_0^{\omega}[T_{x_1}<\tau_{\delta}]\Bigr]\\ \nonumber
\leq & (2/\kappa_0)^k {\bf E}\Bigl[\1{\mathcal{P}_0 \cup \mathcal{P}_1 \in \omega} \1{x_0\Leftrightarrow \infty}\1{x_1\Leftrightarrow 0}P_0^{\omega}[T_{z}<\tau_{\delta}]\Bigr] \\ \nonumber
   \leq  & (2/\kappa_0)^k {\bf E}\Bigl[\1{\mathcal{I}}P_0^{\omega}[T_{z}<\tau_{\delta}]\Bigr],
\end{align}
since on $\1{\mathcal{P}_0 \cup \mathcal{P}_1 \in \omega} \1{x_0\Leftrightarrow \infty}\1{x_1\Leftrightarrow 0}$ we have $0\leftrightarrow x_0 \leftrightarrow z \leftrightarrow x_1 \leftrightarrow \infty$ and which means that $\mathcal{I}$ occurs.

Collecting~(\ref{step1}),~(\ref{step2}),~(\ref{step3}),~(\ref{step4}), noticing that $\1{\mathcal{I}} \leq \1{\mathcal{I}(\omega^{z,\emptyset})}$ and $P_0^{\omega}[T_{z}<\tau_{\delta}]\leq P_0^{\omega^{z,\emptyset}}[T_{z}<\tau_{\delta}]$, we get
\begin{align}
\label{decopart1}
    &{\bf E}\Bigl[\1{\mathcal{I}(\omega^{z,\emptyset})}P_0^{\omega^{z,\emptyset}}[T_z<\tau_{\delta}]L_z(\omega)^{C_2}e^{2\lambda L_z(\omega)}\mid L_z=k\Bigr]\\ \nonumber
 \leq & \rho_d k^{\gamma_1} (8 e^{2\lambda}/\kappa_0)^k {\bf E}\Bigl[\1{\mathcal{I}(\omega^{z,\emptyset})}P_0^{\omega^{z,\emptyset}}[T_{z}<\tau_{\delta}]\Bigr].
 \end{align}

Let us come back to the case where $0 \in B(z,k)$.  We can obtain the same result by saying that $P_0^{\omega^{z,\emptyset}}[T_z<\tau_{\delta}]\leq 1$ in~(\ref{step1}) and formally replacing $P_0^{\omega}[T_x<\tau_{\delta}, T_x=T_{\partial B(z,k)}]$ by $1$ for any $x\in \partial B(z,k)$ and $x_1$ by $0$ in the whole previous proof. The conclusion of this is that~(\ref{decopart1}) holds in any case.

The result follows from an integration over all the events $\{L_z=k\}$ for $k\in \N$ since by~(\ref{decopart1}), we obtain
\begin{align*}
& {\bf E}\Bigl[\1{\mathcal{I}(\omega^{z,\emptyset})}P_0^{\omega^{z,\emptyset}}[T_z<\tau_{\delta}]L_z(\omega)^{C_2}e^{2\lambda L_z(\omega)}\Bigr]\\
\leq& {\bf E}\Bigl[\sum_{k=1}^{\infty}{\bf P}[L_z=k]  {\bf E}\Bigl[\1{\mathcal{I}(\omega^{z,\emptyset})}P_0^{\omega^{z,\emptyset}}[T_z<\tau_{\delta}]L_z(\omega)^{C_2}e^{2\lambda L_z(\omega)}\mid L_z=k\Bigr]\Bigr]\\
\leq& \rho_d {\bf E}\Bigl[\sum_{k=1}^{\infty}{\bf P}[L_z=k] k^{\gamma_1}(8e^{2\lambda}/\kappa_0)^k {\bf E}\Bigl[\1{\mathcal{I}(\omega^{z,\emptyset})}P_0^{\omega^{z,\emptyset}}[T_{z}<\tau_{\delta}]\Bigr]\Bigr]\\
 =& \rho_d {\bf E}\Bigl[L_z^{\gamma_1}(8e^{2\lambda}/\kappa_0)^{L_z}\Bigr]  {\bf E}\Bigl[\1{\mathcal{I}(\omega^{z,\emptyset})}P_0^{\omega^{z,\emptyset}}[T_{z}<\tau_{\delta}]\Bigr].
\end{align*}
\end{proof}

Let us now prove Proposition~\ref{upperb}. 

\begin{proof}

We can apply Proposition~\ref{perco} to get that for $0<\epsilon<\epsilon_6$
\[
{\bf E}\Bigl[L_z(\omega)^{C_{11}}e^{C_{12}L_z(\omega)}\Bigr] \leq  \sum_{k\geq 0}k^{C_{11}}e^{C_{12}k} {\bf P}[L_z\geq  k]<C_{13}<\infty,
\]
where $\epsilon_6$ is such that $\alpha_0(\epsilon_6)<e^{-C_{12}}/2$ and, as $C_{13}$, depends only on $d$ and $\ell$. Then recalling~(\ref{inter}), using Lemma~\ref{decorrelation} with the previous equation we obtain
\begin{align*}
 & {\bf E}\Bigl[\1{\mathcal{I}(\omega^{z,A})}G_{\delta}^{\omega^{z,A}}(0,z)\Bigr] \\
\leq &  4\kappa_1^2 {\bf E}\Bigl[\1{\mathcal{I}(\omega^{z,\emptyset})}G_{\delta}^{\omega^{z,\emptyset}}(0,z)\Bigr] + 2C_1C_{10}C_{13}\kappa_1 {\bf E}\Bigl[\1{\mathcal{I}(\omega^{z,\emptyset})}P_0^{\omega^{z,\emptyset}}[T_z<\tau_{\delta}]\Bigr]
\\
\leq  &  \gamma_2{\bf E}\Bigl[\1{\mathcal{I}(\omega^{z,\emptyset})}G_{\delta}^{\omega^{z,\emptyset}}(0,z)\Bigr].
\end{align*}

Using the preceding equation with~(\ref{denom0}) concludes the proof of Proposition~\ref{upperb}.

\end{proof}

We are now able to prove the following
\begin{proposition}
\label{continuity}
For any $d\geq 2$, $\epsilon<\epsilon_5\wedge \epsilon_6$ and $\ell \in \R^d$ we have
\[
v_{\ell}(1-\epsilon)=d_{\emptyset}+O(\epsilon).
\]
\end{proposition}
\begin{proof}
First notice that
\[
{\bf P}[\C(z)=\emptyset]= 1+O(\epsilon) \text{ and } {\bf P}[\C(z)\neq \emptyset]=O(\epsilon).
\]
using~(3.1) and Proposition~\ref{upperb} we get for $\delta \geq 1/2$,
\begin{equation}
\label{driftpartial}
\widehat{d}_{\delta}^{\epsilon}(z)=d_{\emptyset} \frac{{\bf E}_{1-\epsilon}\Bigl[\1{\mathcal{I}}\1{\C(z)=\emptyset}G_{\delta}^{\omega}(0,z)
\Bigr]}{{\bf E}_{1-\epsilon}\Bigl[\1{\mathcal{I}}G_{\delta}^{\omega}(0,z)\Bigr]} +O(\epsilon),
\end{equation}
where the $O(\cdot)$ depends only on $d$ and $\ell$. But using Proposition~\ref{upperb} again yields
\begin{align*}
&\abs{\frac{{\bf E}_{1-\epsilon}\Bigl[\1{\mathcal{I}}\1{\C(z)=\emptyset}G_{\delta}^{\omega}(0,z)
\Bigr]}{{\bf E}_{1-\epsilon}\Bigl[\1{\mathcal{I}}G_{\delta}^{\omega}(0,z)\Bigr]}-1 }\\
=&
\frac{\sum_{A\subset \nu,~A\neq \emptyset}{\bf E}_{1-\epsilon}\Bigl[\1{\mathcal{I}}\1{\C(z)=A}G_{\delta}^{\omega}(0,z)
\Bigr]}{{\bf E}_{1-\epsilon}\Bigl[\1{\mathcal{I}}G_{\delta}^{\omega}(0,z)\Bigr]} \\
=& \sum_{A\subset \nu,~A\neq \emptyset} {\bf P}_{1-\epsilon}[\C(z)=A] \frac{{\bf E}_{1-\epsilon}\Bigl[\1{\mathcal{I}}G_{\delta}^{\omega}(0,z)\mid \C(z)=A\Bigr]}{{\bf E}_{1-\epsilon}\Bigl[\1{\mathcal{I}}G_{\delta}^{\omega}(0,z)\Bigr]} \leq O(\epsilon),
\end{align*}
and thus
\[
\widehat{d}_{\delta}^{\epsilon}(z)-d_{\emptyset}=O(\epsilon),
\]
where the $O(\cdot)$ depends only on $d$ and $\ell$. Recalling Proposition 3.3, we get
\[
v_{\ell}(1-\epsilon)=d_{\emptyset}+O(\epsilon).
\]
\end{proof}

\section{Derivative of the speed at high density}
\label{s_deriv}

 Next we want to obtain the derivative of the velocity with respect to the percolation parameter. 

In this section we fix $z\in \Z^d$. Using (\ref{drift}) with Proposition~\ref{upperb} we can get the first order of Kalikow's drift 
\begin{equation}
\label{first_order_0}
d_{\delta}^{\widehat{\omega}}(z)-d_{\emptyset}=\epsilon \Bigl(\sum_{e\in \nu} \frac{{\bf E}_{1-\epsilon}[\1{\mathcal{I}(\omega^{z,e})}G_{\delta}^{\omega^{z,e}}(0,z)]}{{\bf E}_{1-\epsilon}[\1{\mathcal{I}(\omega)}G_{\delta}^{\omega}(0,z)]}(d_{e}-d_{\emptyset}) \Bigr)+O_z(\epsilon^2),
\end{equation}
where
\begin{equation}
\label{def_oo1}
\sup_{z\in \Z^d} \abs{O_z(\epsilon^2)}\leq O(\epsilon^2).
\end{equation}
The remaining issue is the dependence of the expectation with respect to $\epsilon$.

For any $A\subset B^E(0,2)$ we denote 
\[
\{(z,2)=A\}=\bigl\{\{ e\in B^E(z,2),\ e\in \omega \}=B^E(z,2)\setminus \{z+A\} \bigr\}.
 \]
 
\subsection{Technical estimate}

Let us prove the following technical lemma, which will simplify the rest of the proof. In words it states that the configuration $B^E(z,2)$ is typically as open as it can be. For example without any condition all edges are open, if $[z,z+e]$ is imposed to be close then it will be the only closed edge in $B^E(z,2)$. One could continue like this, but those two cases are the only ones we need for the rest of the paper.

\begin{lemma}
\label{yet_another_technicality}
We have for $\delta\geq 1/2$, $z\in \Z^d$ and $e\in \nu$,
\[
{\bf E}[\1{\mathcal{I}}G_{\delta}^{\omega}(0,z)]\leq (1+O(\epsilon))
{\bf E}[\1{\mathcal{I}}\1{(z,2)=\emptyset}G_{\delta}^{\omega}(0,z)],
\]
and
\[
{\bf E}[\1{\mathcal{I}(\omega^{z,e})}G_{\delta}^{\omega^{z,e}}(0,z)]\leq (1+O(\epsilon))
{\bf E}[\1{\mathcal{I}(\omega^{z,e})}\1{(z,2)=\emptyset}G_{\delta}^{\omega^{z,e}}(0,z)],
\]
where the $O(\cdot)$ depends only on $d$ and $\ell$.
\end{lemma}

The proof of this lemma is independent of the rest of the paper so it can be skipped in a first reading.

\begin{proof}
Due to the strong similarities with the proof of Lemma~\ref{decorrelation} we will simply sketch the proof of the lemma.

Let us prove the second inequality which is the most complicated. We have
\begin{align}
\label{y_a_t_l}
{\bf E}[\1{\mathcal{I}(\omega^{z,e})}G_{\delta}^{\omega^{z,e}}(0,z)] &= \sum_{\substack{A\in \subset B^E(z,2) \\ A\neq \emptyset}} {\bf P}[(z,2)=A] \\ \nonumber
&           \qquad                                         \times {\bf E}[\1{\mathcal{I}(\omega^{z,e})}G_{\delta}^{\omega^{z,e}}(0,z) \mid (z,2)= A],
\end{align}
let us show that for any $A\subset B^E(z,2)$, $A\neq \emptyset$
\begin{equation}
\label{zzz}
\frac{{\bf E}[\1{\mathcal{I}(\omega^{z,e})}G_{\delta}^{\omega^{z,e}}(0,z) \mid(z,2)=A]}{{\bf E}[\1{\mathcal{I}(\omega^{z,e})}G_{\delta}^{\omega^{z,e}}(0,z)]} <\gamma_1,
\end{equation}
where $\gamma_1$ depends only on $d$ and $\ell$. The method is the same as before
\begin{enumerate}
\item We apply Lemma~\ref{lemgreen} to decompose the Green function into $P_0^{\omega^{z,e}_{(z,2),A}}[T_z<\tau_{\delta}] \times R^{\omega^{z,e}_{(z,2),A}}[z\leftrightarrow \Delta]e^{2\lambda z\cdot \vec \ell}$.
\item With Lemma~\ref{insertedge} we decompose the resistance appearing in $(1)$ into 
\[
R^{\omega^{z,e}}[z\leftrightarrow \Delta]\leq 4 R^{\omega^{z,e}_{(z,2),1}}[z\leftrightarrow \Delta] + C_1L_{z,2}(\omega^{z,e})^{C_2} e^{2\lambda (L_{z,2}(\omega^{z,e})-x\cdot\vec \ell)}.
\]
\item Similarly to~(\ref{step1}) in the case $k=2$ we can obtain
\begin{align*}
&{\bf E}\Bigl[\1{\mathcal{I}(\omega^{z,e}_{(z,2),A})}P_0^{\omega^{z,e}_{(z,2),A}}[T_z<\tau_{\delta}]R^{\omega^{z,e}_{(z,2),1}}[z\leftrightarrow \Delta]\Bigr] \\
\leq& \gamma_2 {\bf E}\Bigl[\1{\partial B(z,2) \leftrightarrow \infty}\max_{x\in \partial B(z,2)}P_0^{\omega}[T_x<\tau_{\delta}, T_x=T_{\partial B(z,2)}]R^{\omega^{z,e}_{(z,2),1}}[z\leftrightarrow \Delta]\Bigr] \Bigr]
\end{align*}
 and repeating the steps~(\ref{step1}),~(\ref{step2}),~(\ref{step3}),~(\ref{step35}) and~(\ref{step4}) for $k=2$ we prove that 
\[
\frac{{\bf E}\Bigl[\1{\mathcal{I}(\omega^{z,e}_{(z,2),A})}P_0^{\omega^{z,e}_{(z,2),A}}[T_z<\tau_{\delta}]R^{\omega^{z,e}_{(z,2),1}}[z\leftrightarrow \Delta]\Bigr]e^{2\lambda z\cdot \vec \ell}}{{\bf E}[\1{\mathcal{I}(\omega^{z,e})}G_{\delta}^{\omega^{z,e}}(0,z)]} <\gamma_3,
\]
the only difference is that we impose $\mathcal{P}_0$ (resp, $\mathcal{P}_1$) to be a path in $B^E(z,2)\setminus [z,z+e]$ of length at most 4 connecting $z$ and $x_0$ (resp. $x_1$) and that~(\ref{step35}) becomes
\[
(\delta \kappa_0)^4 P_0^{\omega}[T_{x_1}<\tau_{\delta}] \leq P_0^{\omega^{z,e}}[T_z<\tau_{\delta}].
\]
\item We can use arguments similar to the ones in the proof of Lemma~\ref{decorrelation} (essentially repeating the steps~(\ref{step1}),~(\ref{step2}),~(\ref{step3}),~(\ref{step35}) and~(\ref{step4})) to prove that
\[
\frac{{\bf E}\Bigl[\1{\mathcal{I}(\omega^{z,e}_{(z,2),A})}P_0^{\omega^{z,e}_{(z,2),A}}[T_z<\tau_{\delta}]L_{z,2}^{C_2}(\omega^{z,e})e^{2\lambda L_{z,2}(\omega^{z,e})}\Bigr]}{{\bf E}[\1{\mathcal{I}(\omega^{z,e})}G_{\delta}^{\omega^{z,\emptyset}}(0,z)]} <\gamma_4,
\]
since $L_{z,2}(\omega^{z,e})$ has arbitrarily large exponential moments under the measure ${\bf P}[\, \cdot \,]$, for $\epsilon$ small enough by Proposition~\ref{perco2}. Here we also need $\mathcal{P}_0$(resp, $\mathcal{P}_1$)  to be a path in $B^E(z,k)\setminus [z,z+e]$ of length at most $k+2$ connecting $z$ and $x_0$ (resp. $x_1$) and that~(\ref{step35}) becomes
\[
(\delta \kappa_0)^{k+2} P_0^{\omega}[T_{x_1}<\tau_{\delta}] \leq P_0^{\omega^{z,e}}[T_z<\tau_{\delta}].
\]
\end{enumerate}

This reasoning yields~(\ref{zzz}) with $\gamma_1=4\gamma_3+C_1\gamma_4$. Now, the equations~(\ref{zzz}) and~(\ref{y_a_t_l}) imply that
\begin{align*}
 &{\bf E}[\1{\mathcal{I}(\omega^{z,e})}G_{\delta}^{\omega^{z,e}}(0,z)]\leq O(\epsilon){\bf E}[\1{\mathcal{I}(\omega^{z,e})}G_{\delta}^{\omega^{z,e}}(0,z)] \\
&\qquad \qquad +{\bf E}[\1{\mathcal{I}(\omega^{z,e})}\1{(z,2)=\emptyset}G_{\delta}^{\omega^{z,e}}(0,z)],
\end{align*}
and it follows that
\[
{\bf E}[\1{\mathcal{I}(\omega^{z,e})}G_{\delta}^{\omega^{z,e}}(0,z)] 
\leq (1+O(\epsilon)){\bf E}[\1{\mathcal{I}}\1{(z,2)=\emptyset}G_{\delta}^{\omega^{z,e}}(0,z)].
\]

This completes the proof of the second inequality of the lemma. The proof for the first inequality is the same except that it uses Proposition~\ref{perco3}.
\end{proof}

\subsection{Another perturbed environment of Kalikow}

We recall that our aim is to compute
\[
\frac{{\bf E}_{1-\epsilon}[\1{\mathcal{I}(\omega^{z,e})}G_{\delta}^{\omega^{z,e}}(0,z)]}{{\bf E}_{1-\epsilon}[\1{\mathcal{I}}G_{\delta}^{\omega}(0,z)]},
\]
and we will start by studying the numerator. Our aim is to relate it to the denominator, for this we need to express the quantities appearing in the environment $\omega^{z,e}$ in terms of similar quantities in the environment $\omega^{z,\emptyset}$, which is the environment that naturally arises for $p=1-\epsilon$ close to $1$.

We can link the Green functions of two Markov operators $P$ and $P'$, since for $n \geq 0$
\begin{equation}
\label{green_exp}
G_{\delta}^{P'}=G_{\delta}^P+\sum_{k=1}^n \delta^k(G_{\delta}^P(P'-P))^kG_{\delta}^P+\delta^{n+1}(G_{\delta}^P(P'-P))^{n+1}G_{\delta}^{P'}.
\end{equation}

In our case we close one edge which changes the transition probabilities at two sites, so that the previous formula applied for $n=0$,
\begin{align}
\label{green_exp1}
 G^{\omega^{z,e}}_{\delta}(0,z)=&G^{\omega^{z,\emptyset}}_{\delta}(0,z)+\delta G^{\omega^{z,\emptyset}}_{\delta}(0,z) \sum_{e'\in \nu} (p^e(e')-p^{\emptyset}(e'))G^{\omega^{z,e}}_{\delta}(z+e',z)\\\nonumber
                     &+\delta G^{\omega^{z,\emptyset}}_{\delta}(0,z+e) \sum_{e'\in \nu} (p^{\omega^{z,e}}(z+e,z+e+e')-p^{\omega^{z,\emptyset}}(z+e,z+e+e')) \\ \nonumber 
& \qquad \qquad \qquad \qquad \times G^{\omega^{z,e}}_{\delta}(z+e+e',z),                 
\end{align}
where we used a notation from~(\ref{shorter}).

Typically the configuration at $z+e$ is $\{-e\}$. This intuition follows from an easy consequence of Lemma~\ref{yet_another_technicality} which is
\[
\text{for all $z\in \Z^d$ and $e\in \nu$,} \qquad \abs{\frac{{\bf E}_{1-\epsilon}[\1{\mathcal{I}(\omega^{z,e})}G_{\delta}^{\omega^{z,e}}(0,z)]}{{\bf E}_{1-\epsilon}[\1{\mathcal{I}(\omega^{z,e})}\1{(z,2)=\emptyset}G_{\delta}^{\omega^{z,e}}(0,z)]}-1}\leq O(\epsilon),
\]
which yields that
\begin{align}
\label{ugly_ass_ub}
& (1+O_{z,e}(\epsilon)){\bf E}_{1-\epsilon}[\1{\mathcal{I}(\omega^{z,e})}G^{\omega^{z,e}}_{\delta}(0,z)] \\ \nonumber
=&{\bf E}_{1-\epsilon}\Bigl[\1{\mathcal{I}(\omega^{z,e})} \1{(z,2)=\emptyset}  \\ \nonumber
    &\qquad \times \Bigl[G^{\omega^{z,\emptyset}}_{\delta}(0,z)+\delta G^{\omega^{z,\emptyset}}_{\delta}(0,z) \sum_{e'\in \nu} (p^e(e')-p^{\emptyset}(e'))G^{\omega^{z,e}}_{\delta}(z+e',z)\\\nonumber
                     & \qquad +\delta G^{\omega^{z,\emptyset}}_{\delta}(0,z+e) \sum_{e'\in \nu} (p^{-e}(e')-p^{\emptyset}(e'))G^{\omega^{z,e}}_{\delta}(z+e+e',z) \Bigr]  \Bigr], 
 \end{align}
where $\sup_{z\in \Z^d ,e\in \nu} \abs{O_{z,e}(\epsilon)} \leq O(\epsilon)$.                     

We have managed to express the quantities in the environment $\omega^{z,e}$ with quantities in $\omega^{z,\emptyset}$.  Now we are led to look at quantities such as
\begin{equation}
\label{termterm1}
{\bf E}_{1-\epsilon}[\1{\mathcal{I}(\omega^{z,e})} \1{(z,2)=\emptyset} G_{\delta}^{\omega^{z,\emptyset}}(0,z) G_{\delta}^{\omega^{z,e}}(z+e',z)]
\end{equation}
and
\begin{equation}
\label{termterm2}
\ {\bf E}_{1-\epsilon}[\1{\mathcal{I}(\omega^{z,e})} \1{(z,2)=\emptyset} G_{\delta}^{\omega^{z,\emptyset}}(0,z+e) G_{\delta}^{\omega^{z,e}}(z+e+e',z)].
\end{equation}

From now on we fix $e\in \nu$. In order to handle the first type of terms (the proof is similar for the second term) we introduce the measure
\[
d\tilde{\mu}^z= \frac{\1{\mathcal{I}} \1{(z,2)=e}G_{\delta}^{\omega^{z,\emptyset}}(0,z)}{{\bf E}_{1-\epsilon}\left[\1{\mathcal{I}}\1{(z,2)=e}G_{\delta}^{\omega^{z,\emptyset}}(0,z)\right]} d{\bf P}_{1-\epsilon},
\]
and for $e_+\in \nu$ we introduce the Kalikow environment, corresponding to this measure on the environment and the point $z+e_+$, defined by
\begin{align*}
& \tilde{p}_{z,e,z+e_+}(y,y+e')\\
=&\frac{E_{\tilde{\mu}^z}[G_{\delta}^{\omega}(z+e_+,y)p^{\omega}(y,y+e')]}{E_{\tilde{\mu}^z}[G_{\delta}^{\omega}(z+e_+,y)]} \\
=&\frac{{\bf E}_{1-\epsilon}[\1{\mathcal{I}}\1{(z,2)=e} G_{\delta}^{\omega^{z,\emptyset}}(0,z)G_{\delta}^{\omega}(z+e_+,y)p^{\omega}(y,y+e')]}{{\bf E}_{1-\epsilon}[\1{\mathcal{I}}\1{(z,2)=e}G_{\delta}^{\omega^{z,\emptyset}}(0,z)G_{\delta}^{\omega}(z+e_+,y)]} \\ 
=&\frac{{\bf E}_{1-\epsilon}[\1{\mathcal{I}(\omega^{(z,2)=e})}G_{\delta}^{\omega^{(z,2)=\emptyset}}(0,z)G_{\delta}^{\omega^{(z,2)=e}}(z+e_+,y)p^{\omega^{(z,2)=e}}(y,y+e')]}{{\bf E}_{1-\epsilon}[\1{\mathcal{I}(\omega^{(z,2)=e})}G_{\delta}^{\omega^{(z,2)=\emptyset}}(0,z)G_{\delta}^{\omega^{(z,2)=e}}(z+e_+,y)]},
\end{align*}
where we used the notations from~(\ref{notation_omega}) and~(\ref{notation_omega1}).

Once again since Kalikow's property geometrically killed random walks does not use any properties on the measure of the environment, we have for any $z\in \Z^d$ and $e,e'\in \nu$ a property similar to Proposion~\ref{kalikow_0}, which allows us to relate the quantity in~(\ref{termterm1}) to
\begin{align}
\label{kalikow2}
G_{\delta}^{\tilde p_{z,e,z+e'}}(z+e',z)&=\frac{{\bf E}_{1-\epsilon}[\1{\mathcal{I}}\1{(z,2)=e} G_{\delta}^{\omega^{z,\emptyset}}(0,z) G_{\delta}^{\omega}(z+e',z)]}{{\bf E}_{1-\epsilon}[\1{\mathcal{I}}\1{(z,2)=e}G_{\delta}^{\omega^{z,\emptyset}}(0,z)]}\\ \nonumber
                           &=\frac{{\bf E}_{1-\epsilon}[\1{\mathcal{I}(\omega^{z,e})}\1{(z,2)=e} G_{\delta}^{\omega^{z,\emptyset}}(0,z) G_{\delta}^{\omega^{z,e}}(z+e',z)]}{{\bf E}_{1-\epsilon}[\1{\mathcal{I}(\omega^{z,e})}\1{(z,2)=e} G_{\delta}^{\omega^{z,\emptyset}}(0,z)]} \\ \nonumber
                           &=\frac{{\bf E}_{1-\epsilon}[\1{\mathcal{I}(\omega^{z,e})}\1{(z,2)=\emptyset} G_{\delta}^{\omega^{z,\emptyset}}(0,z) G_{\delta}^{\omega^{z,e}}(z+e',z)]}{{\bf E}_{1-\epsilon}[\1{\mathcal{I}}\1{(z,2)=\emptyset} G_{\delta}^{\omega}(0,z)]}
\end{align}
since on $\{(z,2)=\emptyset\}$ or $\{(z,2)=e\}$ we have $\1{\mathcal{I}} =\1{\mathcal{I}(\omega^{z,e})}$. 

The numerator of the previous display is exactly~(\ref{termterm1}). Hence for us it is sufficient to approximate $G_{\delta}^{\tilde p_{z,e,z+e'}}(z+e',z)$ to understand it and consequently to understand the derivative of the speed. We are now led to studying $\tilde p_{z,e,z+e'}(y,y+e')$. A similar reasoning could be made to understand~(\ref{termterm2}) .

Decomposing $\tilde{p}(y,y+e')$ according to the configurations at $y$, we get 
\begin{align}
\label{kali_2}
&\tilde{p}_{z,e,z+e_+}(y,y+e')=\sum_{A \subset \nu, A \neq \nu} {\bf P}_{1-\epsilon}[\C(y)=A] \\ \nonumber
 &\times \frac{{\bf E}_{1-\epsilon}[\1{\mathcal{I}(\omega^{(z,2)=e})}G_{\delta}^{\omega^{(z,2)=\emptyset}}(0,z)G_{\delta}^{\omega^{(z,2)=e}}(z+e_+,y)\mid \C(y)=A]}{{\bf E}_{1-\epsilon}[\1{\mathcal{I}(\omega^{(z,2)=e})}G_{\delta}^{\omega^{(z,2)=\emptyset}}(0,z)G_{\delta}^{\omega^{(z,2)=e}}(z+e_+,y)]} p^{\omega_{y,A}^{(z,2)=e}}(y,y+e').
\end{align}

Let us denote $a^+=0\vee a$ and from now on we will omit the subscript in $\tilde{p}_{z,e,z+e_+}$. The following proposition states that $\tilde{p}$ and $p_0^{z,e}$ are close in some sense. Using this, we will prove in the next section that it implies that $G_{\delta}^{\tilde p_{z,e,z+e'}}(z+e',z)$ and $G_{\delta}^{p_0^{z,e}}(z+e',z)$ are close.

\begin{proposition}
\label{approx_trans_prob}
For $\epsilon<\epsilon_7$ and $z,e,e_+ \in \Z^d\times \nu^2$ and $\delta \in (1/2,1)$, we have for $y\in \Z^d$, $e'\in \nu$ 
\[
\abs{\tilde{p}(y,y+e')-p_0^{z,e}(y,y+e')}\leq   (C_{14} e^{C_{15}((z-y) \cdot \vec \ell)^+}) \epsilon,
\]
where $\epsilon_7$, $C_{14}$ and $C_{15}$ depends on $\ell$ and $d$. We recall that $p_0^{z,e}$ is the environment where only the edge $[z,z+e]$ is closed.
\end{proposition}

This proposition will be used to link the Green function of $\tilde{p}_{z,e,z+e_+}$ to the one of $p_0^{z,e}$. In view of~(\ref{kali_2}) the previous proposition comes from the following.
\begin{proposition}
\label{quotient_ub}
For $0<\epsilon<\epsilon_8$, $y,z\in \Z^d$ and $A\subset \nu$, $A\neq \nu$,
\[
\frac{{\bf E}_{1-\epsilon}[\1{\mathcal{I}(\omega^{(z,2)=e})}G_{\delta}^{\omega^{(z,2)=\emptyset}}(0,z)G_{\delta}^{\omega^{(z,2)=e}}(z+e_+,y)\mid \C(y)=A]}{{\bf E}_{1-\epsilon}[\1{\mathcal{I}(\omega^{(z,2)=e})}G_{\delta}^{\omega^{(z,2)=\emptyset}}(0,z)G_{\delta}^{\omega^{(z,2)=e}}(z+e_+,y)]\mid \C(y)=\emptyset ]}\leq C_{16} e^{C_{17}((z-y) \cdot \vec \ell)^+},
\]
for $\epsilon_8$, $C_{16}$, $C_{17}$ depending only on $\ell$ and $d$ for $\delta\geq 1/2$.
\end{proposition}

In order to prove Proposition~\ref{approx_trans_prob}, once we have noticed that we have ${\bf P}[\C(y)=\emptyset]\geq \gamma_1$ and that 
\begin{align*}
&{\bf E}_{1-\epsilon}[\1{\mathcal{I}(\omega^{(z,2)=e})}G_{\delta}^{\omega^{(z,2)=\emptyset}}(0,z)G_{\delta}^{\omega^{(z,2)=e}}(z+e_+,y)]\\
\geq & {\bf P}_{1-\epsilon}[\C(y)=\emptyset] {\bf E}_{1-\epsilon}[\1{\mathcal{I}(\omega^{(z,2)=e})}G_{\delta}^{\omega^{(z,2)=\emptyset}}(0,z)G_{\delta}^{\omega^{(z,2)=e}}(z+e_+,y)]\mid \C(y)=\emptyset ],
\end{align*}
it suffices to substract $p_0^{z,e}(y,y+e')$ on both sides of~(\ref{kali_2}) and use the Proposition~\ref{quotient_ub}, to get Proposition~\ref{approx_trans_prob} with $C_{14}=2dC_{16}/ \gamma_1$ and $C_{15}=C_{17}$.

Obviously Proposition~\ref{quotient_ub} has strong similarities with Proposition~\ref{upperb}, since the only difference is that the upper bound is weaker, which is simply due to technical reasons. Moreover, since the proof is rather technical and independent of the rest of the argument, we prefer to defer it to Section~\ref{s_tech}.

\subsection{Expansion of Green functions}

Once Proposition~\ref{approx_trans_prob} is proved, we are able to approximate the Green functions appearing in~(\ref{kalikow2}) through the same type of arguments as given in~\cite{Sabot}. 

Heuristically, we may say that if environments are close then the Green functions should be close at least on short distance scales. 

Compared to~\cite{Sabot}, there is a twist due to the fact that we do not have uniform ellipticity and that our control on the environment in Proposition~\ref{approx_trans_prob} is only uniform in the direction of the drift. Moreover our \lq\lq limiting environment\rq\rq as $\epsilon$ goes to $0$ is not translation invariant (nor uniformly elliptic). Hence we need some extra work to adapt the methods of~\cite{Sabot}.
\begin{proposition}
\label{approx_green_function}
For any $z\in \Z^d$, $e,e_+ \in \nu$, $e',e''\in \nu \cup \{0\}$ we get 
\[
\sup_{\delta \in [1/2,1)}\abs{G_{\delta}^{\tilde{p}}(z+e'+e'',z)-G_{\delta}^{\omega_0^{z,e}}(z+e'+e'',z)}\leq o_{\epsilon}(1),
\]
where $o_{\epsilon}(\cdot)$ depends only on $\ell$ and $d$. We recall that $\tilde{p}$ represents $\tilde{p}_{z,e,z+e_+}$.
\end{proposition}

The proof of this proposition is independent of the rest of the argument and can be skipped on first lecture to see how it actually leads to the computation of the derivative.

\begin{proof}
The proof will be divided in two main steps:
\begin{enumerate}
\item prove that there exists transition probabilities $\overline{p}$ that are uniformly close to those corresponding to the environment $\omega_0^{z,e}$ on the whole lattice and which has a Green function close to the one of the transition probabilities $\tilde{p}$,
\item prove the same statement as in Proposition~\ref{approx_green_function} but for the environment $\overline{p}$ instead of $\tilde{p}$. Since the control on the environment is now uniform we can use arguments close to those of the proof of Lemma 3 in~\cite{Sabot}.
\end{enumerate}

{\it Step (1)}

For the first step, we will show that the random walk is unlikely to visit often $z$ and go far away in the direction  opposite to the drift, i.e.~we want to show that for any $\epsilon'>0$
\begin{equation}
\label{object}
G^{\tilde{p}}_{\delta}(z+e'+e'',z) - G^{\tilde{p}}_{\delta, \{x \in \Z^d,~ x \cdot \vec \ell <z\cdot \vec \ell-A\}}(z+e'+e'',z) <\epsilon',
\end{equation}
for $A$ large and $\epsilon$ small, where we used a notation of~(\ref{fct_green_gen}). This inequality follows from the fact that except at $z$ and $z+e$ the local drift under $\tilde{p}$ can be set to be uniformly positive in the direction $\vec{\ell}$ in any half-space $\{x\in \Z^d,~x\cdot \vec{\ell} >-A\}$ for $\epsilon$ small by Proposition~\ref{approx_trans_prob}.

{\it Step (1)-(a)}

In a first time, we show that the escape probabilities from $z$, $z+e$ and $z+e'+e''$ (to $\Delta$) are lower bounded in the environment $\tilde{p}$. This ensures that there cannot be many visits at those three points. 

For this we use the result of a classical super-martingale argument, see Lemma~1.1 in~\cite{SZ-Slow}. Without entering further into the details, this argument yields that for any $\eta>0$ there exists $f(\eta)>0$ such that for any random walk on $\Z^d$ defined by a Markov operator $P(x,y)$ such that $\bigl(\sum_{y \sim x} P(x,y) (y-x)\bigr) \cdot \vec \ell >\eta$, for $x$ such that $x\cdot \vec \ell \geq 0$, we have
\begin{equation}
\label{lb_esc_prob}
	P_0[X_n \cdot \vec \ell \geq 0, \text{ for all }n>0]>f(\eta).
\end{equation}

Now by Proposition~\ref{approx_trans_prob}, it is possible to fix a percolation parameter $1-\epsilon$ where $\epsilon$ is chosen small enough depending solely on $d$ and $\ell$ so that 
\begin{itemize}
\item The drift $d^{\tilde{p}}(x)=\sum_{e \in \nu} \tilde{p}(x,x+e) e$ is such that $d^{\tilde{p}}(x) \cdot \vec \ell > d_{\emptyset}\cdot \vec \ell /2$ for $x$ such that $x\cdot \vec \ell \geq (z+2de^{(1)})\cdot \vec \ell$ (this way we avoid the transitions probabilities at the vertices $z$ and $z+e$ which are special).
\item The transition probabilities $\tilde{p}$ on the shortest paths from $z$, $z+e$ and $z+e'+e''$ to $z+2de^{(1)}$ which does not use the edge $[z,z+e]$ (they have some length inferior to some $\gamma_1$ depending only on $d$) are greater than $\kappa_0/2$. 
\end{itemize}

Hence we can get a lower bound for the escape probability under $\tilde{p}$:
\begin{align}
\label{no_return_step1}
& \min_{y\in \{z,z+e,z+e'+e''\}} P_{y}^{\tilde{p}}[T_{\{z,z+e,z+e'+e''\}}^+=\infty] \\ \nonumber
 \geq & \min_{y\in \{z,z+e,z+e'+e''\}}  P_{y}^{\tilde{p}}[T_{\{z,z+e,z+e'+e''\}}^+>T_{z+2de^{(1)}}] \\ \nonumber
      & \qquad \times P_{z+2de^{(1)}}^{\tilde{p}}[(X_n-(z+2de^{(1)})) \cdot \vec \ell \geq 0,n >0] \\ \nonumber
 \geq & f(d_{\emptyset} \cdot \vec{\ell}/2) \Bigl(\frac{\kappa_0}2\Bigr)^{\gamma_1} =\gamma_2,
\end{align}
where $\gamma_2$ depends only on $d$ and $\ell$.

{\it Step (1)-(b)}
Now in a second time we will show that the walk is unlikely to go far in the direction opposite to the drift during an excursion from $z$, $z+e$ or $z+e'+e''$. Once this is done, this will imply with~(\ref{no_return_step1}) that the walker is unlikely to reach the half-plane $\{x \in \Z^d,~ x \cdot \vec \ell <z\cdot \vec \ell-A\}$ and~(\ref{no_return_step1}) also that when it does the expected number of returns to $z$ remains bounded. This will be made rigorous in Step (1)-(c) and will prove~(\ref{object}).
 
Consider any random walk on $\Z^d$ given by a transition operator $P(x,y)$ such that for all $x\in \Z^ d$ we have $d^P(x)\cdot \vec \ell:=\sum_{y \sim x} P(x,y) (y-x) \cdot \vec \ell > (d_{\emptyset}\cdot \vec \ell )/2=\gamma_3$. We know that
\[
M_n^P=X_n-X_0-\sum_{i=0}^{n-1} d^P(X_i),
\]
is a martingale with jumps bounded by 2. Hence since $d^P(x) \geq \gamma_3$, we can use Azuma's inequality, see~\cite{Alon}, to get
\begin{align*}
P_0[T_{\{x \in \Z^d,~ x \cdot \vec \ell <-A\}}<\infty]&\leq \sum_{n \geq 0} P_0\Bigl[M_n^P\cdot \vec \ell <-A-\gamma_3 n\Bigr]\\
                                                                                   &\leq \sum_{n\geq 0} \exp\Bigl(-\frac{(A+\gamma_3 n)^2}{8n}\Bigr)\leq \gamma_4\exp\Bigl(-\frac{\gamma_3 A}4\Bigr). 
\end{align*}

Set $\epsilon''>0$. Taking $A=A(\epsilon'')$ large enough, depending also on $d$ and $\ell$, we can make the right-hand side lower than $\epsilon''$, i.e.~$A(\epsilon'')\geq -(4/\gamma_3)\ln(\epsilon''/\gamma_4)$.

Now let us choose $\epsilon$ small enough so that for any $y\in \{x \in \Z^d,\ x \cdot \vec {\ell} \geq z\cdot \vec \ell-A-3 \text{ and } x\notin \{z,z+e\} \}$ we have $d^{\tilde{p}}(y) \cdot \vec \ell > (d_{\emptyset}\cdot \vec \ell) /2$. Let us introduce the environment $\tilde{p}_A$ such that
\begin{enumerate}
\item $\tilde{p}_A(y,y+f)=\tilde{p}(y,y+f)$ for all $f\in \nu$ and $y\in \{x \in \Z^d,\ x \cdot \vec {\ell} \geq z\cdot \vec \ell-A-3 \text{ and } x\notin \{z,z+e\} \}$,
\item $\tilde{p}_A(y,y+f)=p^{\emptyset}(f)$ for all $f\in \nu$ otherwise,
\end{enumerate}
the same formulas holding when the target point is $\Delta$.

Then, the previous computations, valid for $\tilde{p}_A$, imply
\begin{align}
\label{no_return_step2}
&\max_{y\in \{z,z+e,z+e'+e''\}}P_{y}^{\tilde{p}}[T_{\{x \in \Z^d,~ x \cdot \vec \ell <z\cdot \vec \ell -A-3\}}<T_{\{z,z+e,z+e'+e''\}}^+]\\ \nonumber
\leq & \max_{y \in\partial \{z,z+e,z+e'+e''\}}P_{y}^{\tilde{p}}[T_{\{x \in \Z^d,~ x \cdot \vec \ell <z\cdot \vec \ell -A-3\}}<T_{\{z,z+e,z+e'+e''\}}] \\ \nonumber
\leq & \max_{y \in\partial \{z,z+e,z+e'+e''\}}P_{y}^{\tilde{p}_A}[T_{\{x \in \Z^d,~ x \cdot \vec \ell <z\cdot \vec \ell -A-3\}}<T_{\{z,z+e,z+e'+e''\}}] \\ \nonumber
\leq & \max_{y \in\partial \{z,z+e,z+e'+e''\}}P_{y}^{\tilde{p}_A}[T_{\{x \in \Z^d,~ x \cdot \vec \ell <z\cdot \vec \ell -A-3\}}<\infty\}] \leq \epsilon'',
\end{align}
where we used that the event on the second line depends only on the transitions probabilities at the vertices of $\{x \in \Z^d,\ x \cdot \vec \ell \geq z\cdot \vec \ell-A-3 \text{ and } x\notin \{z,z+e\} \}$.

{\it Step (1)-(c)}

Let us now turn to the proof of~(\ref{object}). By~(\ref{no_return_step1}) we have
\begin{align}
\label{will_this_ever_stop_1}
 G^{\tilde{p}}_{\delta}(z+e'+e'',z+e'+e'')&=P^{\tilde{p}}_{z+e'+e''}[T^+_{z+e'+e''}>\tau_{\delta}]^{-1}\leq \frac 1{\gamma_2} \\ \nonumber
G^{\tilde{p}}_{\delta}(z,z)\leq \frac 1{\gamma_2} & \text{ and } G^{\tilde{p}}_{\delta}(z+e,z+e)\leq \frac 1{\gamma_2}.
\end{align}

Now decomposing the event $T_{\{x \in \Z^d,~ x \cdot \vec \ell <z\cdot \vec \ell-A-3\}}<\infty$ first with respect to the number of excursions to $z+e'+e''$ and then with respect to $z$ and $z+e$ in addition with~(\ref{no_return_step2}) and~(\ref{will_this_ever_stop_1}) yields
\begin{align}
\label{will_this_ever_stop_2}
& P_{z+e'+e''}^{\tilde{p}}[T_{\{x \in \Z^d,~ x \cdot \vec \ell <z\cdot \vec \ell-A-3\}}<\infty] \\ \nonumber
\leq&  G_{\delta}^{\tilde{p}}(z+e'+e'',z+e'+e'') P_{z+e'+e''}^{\tilde{p}}[T_{\{x \in \Z^d,~ x \cdot \vec \ell <z\cdot \vec \ell-A-3\}}<T_{z+e'+e''}^+]  \\ \nonumber
\leq& \frac 1 {\gamma_2} \bigl(G_{\delta}^{\tilde{p}}(z+e'+e'',z)+G_{\delta}^{\tilde{p}}(z+e'+e'',z+e)\bigr) \\ \nonumber
 & \times \max_{y \in\partial \{z,z+e,z+e'+e''\}}P_{y}^{\tilde{p}}[T_{\{x \in \Z^d,~ x \cdot \vec \ell <z\cdot \vec \ell -A-3\}}<T_{\{z,z+e,z+e'+e''\}}^+]\\ \nonumber
\leq& \frac{2\epsilon''}{\gamma_2^2}.
\end{align}

For $\epsilon$ small enough to verify the previous conditions we have using~(\ref{will_this_ever_stop_1}) and~(\ref{will_this_ever_stop_2})
\begin{align}
\label{no_return_step3}
&G^{\tilde{p}}_{\delta}(z+e'+e'',z) - G^{\tilde{p}}_{\delta, \{x \in \Z^d,~ x \cdot \vec \ell <z\cdot \vec \ell-A-3\}}(z+e'+e'',z)\\ \nonumber
\leq & P_{z+e'+e''}^{\tilde{p}}[T_{\{x \in \Z^d,~ x \cdot \vec \ell <z\cdot \vec \ell-A-3\}}<\infty] \max_{y\in \{x \in \Z^d,~ x \cdot \vec \ell <z\cdot \vec \ell-A-3\}} G^{\tilde{p}}_{\delta}(y,z) \\ \nonumber
\leq & \frac{2\epsilon''}{\gamma_2^2} G^{\tilde{p}}_{\delta}(z,z)=\frac{2\epsilon''}{\gamma_2^3}.
\end{align}

So that introducing $\overline{p}(y,f)$ so that for $f\in \nu $
\begin{align*}
\overline{p}(y,y+f)&=\tilde{p}(y,y+f) \text{ for $y$ such that $(y-z)\cdot\vec \ell \geq - A(\epsilon'')-1$} \\
\overline{p}(y,y+f)&=p^{\omega_0^{z,e}}(y,y+f) \text{ for $y$ such that $(y-z)\cdot\vec \ell < - A(\epsilon'')-1$},
\end{align*}
the same formulas holding when the target point is $\Delta$.

The equation~(\ref{no_return_step3}) is also valid for $G_{\delta}^ {\overline{p}}$ so that
\begin{equation}
\label{no_returns}
\abs{G^{\tilde{p}}_{\delta}(z+e'+e'',z)-G_{\delta}^{\overline{p}}(z+e'+e'',z)} \leq \gamma_5\epsilon'',
\end{equation}
where, by Proposition~\ref{approx_trans_prob}, $\overline{p}$ (depending on $\epsilon ''$) is such that
\begin{equation}
\label{approx_trans_prob1}
\max_{f\in \nu, y \in \Z^d } \abs{\overline{p}(y,f)-p^{\omega_0^{z,e}}(y,f)}\leq C_{14}e^{C_{15}A(\epsilon'')}\epsilon \leq \epsilon',
\end{equation}
for $\epsilon$ small enough (depending on $\epsilon'$ and $\epsilon''$) given any arbitrary $\epsilon'$. This completes step (1).

{\it Step (2)}

Since our control on the environment is now uniform through the environment $\overline{p}$, it turns out that we can use methods similar to those of~\cite{Sabot} to prove that there exists a $O(\epsilon)$ depending only on $d$ and $\ell$ such that
\begin{equation}
\label{resultat11}
 \sup_{\delta\in [1/2,1)} \abs{G^{\omega_0^{z,e}}_{\delta}(z+e'+e'',z)-G_{\delta}^{\overline{p}}(z+e'+e'',z)}\leq O(\epsilon),
\end{equation}
which in view of~(\ref{no_returns}) and~(\ref{approx_trans_prob1}) is enough to prove Proposition~\ref{approx_green_function}.

Let us define $M$ the operator of multiplication by $(\pi^{\omega_0^{z,e}})^{1/2}$ given for $f:\Z^d\to \R$, by
\[
M(f)(y)=(\pi^{\omega_0^{z,e}}(y))^{1/2}f(y).
\]

We consider a transition operator $P^{s,\delta}$ of a random walk on $\Z^d\cup \{\Delta\}$ given by
\begin{align}\label{def_op1}
P^{s,\delta}(x,x+e^{(i)})&=P^{s,\delta}(x+e^{(i)},x)\\ \nonumber
&=\delta(\pi^{\omega_0^{z,e}}(x))^{1/2}p^{\omega_0^{z,e}}(x,x+e^{(i)})(\pi^{\omega_0^{z,e}}(x+e^{(i)}))^{-1/2}\\ \nonumber
& =\delta(\pi^{\omega_0^{z,e}}(x+e^{(i)}))^{1/2}p^{\omega_0^{z,e}}(x+e^{(i)},x)(\pi^{\omega_0^{z,e}}(x))^{-1/2}\\ \nonumber
&=\delta (p^{\omega_0^{z,e}}(x+e^{(i)},x)p^{\omega_0^{z,e}}(x,x+e^{(i)}))^{1/2},
\end{align}
for any $i=1,\ldots,2d$ and for $x\in \Z^d$ 
\begin{align}\label{def_op2}
P^{s,\delta}(x,\Delta)&=(1-\delta)+\delta\Bigl(1-\sum_{e^{(i)}\in \nu}(p^{\omega_0^{z,e}}(x+e^{(i)},x)p^{\omega_0^{z,e}}(x,x+e^{(i)}))^{1/2}\Bigr)\\ \nonumber
 &=(1-\delta)+\frac{\delta}2\Bigl( \sum_{e^{(i)}\in \nu}(p^{\omega_0^{z,e}}(x+e^{(i)},x)^{1/2}-p^{\omega_0^{z,e}}(x,x+e^{(i)})^{1/2})^2\Bigr.\\ \nonumber
 &\qquad\qquad +2( p^{\omega_0}(x-e,x)-p^{\omega_0^{z,e}}(x-e,x)) \Bigr) \\ \nonumber
\text{and }P^{s,\delta}(\Delta,\Delta)&=1 .
\end{align}

Let us consider the following transformation appearing in~\cite{Sabot} which will simplify the proof. For $x,y \neq \Delta$,
\begin{align*}
G_{\delta}^{\omega_0^{z,e}}(x,y)=((I-\delta P^{\omega_0^{z,e}})^{-1})(x,y)&=(M^{-1}(I-P^{s,\delta})^{-1}M)(x,y)\\
                                                                         &=(M^{-1}G^{s,\delta} M)(x,y),
\end{align*}
where $G^{s,\delta}$ is the Green function of $P^{s,\delta}$. We define the operator $P^{\overline{s},\delta}$ the same way as in~(\ref{def_op1}) and~(\ref{def_op2}) using the environment $\overline{p}$ instead of $\omega_0^{z,e}$. Recalling~(\ref{approx_trans_prob1}) we have
\begin{align*}
 &P^{\overline{p},\delta}(x,x+e)=P^{\omega_0^{z,e}}(x,x+e)+\epsilon \xi_{\epsilon}(x,e)\\
 \text{ and } &P^{\overline{p},\delta}(x,\Delta)=P^{\omega_0^{z,e}}(x,\Delta)+\epsilon \xi_{\epsilon}(x,\Delta),
\end{align*}
where $\xi_{\epsilon}(\cdot,\cdot)$ are uniformly bounded (independently of $\delta$). 

Now, we use the following expansion of Green functions. For any $n\geq 0$ and $P$, $P'$ two Markov operators on $\Z^d\cup\{\Delta\}$ such that 
\[
\text{for $x\in \Z^d$,} \qquad P(x,\Delta)\geq c \text{ and } P'(x,\Delta) \geq c.
\]
we get
\[
G^{P'}=G^P+\sum_{k=1}^n (G^P(P'-P))^kG^P+(G^P(P'-P))^{n+1}G^{P'} \qquad \text{on $\Z^d$}.
\]

Since $P^{\overline{p},\delta}(x,\Delta)>c(\delta)>0$ and $P^{\omega_0^{z,e}}_{\delta}(x,\Delta)>1-\delta>0$, we can apply the previous formula to obtain for $x,x'\in\Z^d$,
\[
G^{\overline{p},\delta}(x,x')-G^{\omega_0^{z,e}}_{\delta}(x,x')= \sum_{i=1}^n \epsilon^kS_k(x,x')+\epsilon^{n+1}R_n(x,x'),
\]
where
\begin{align*}
&S_n(x,x')=\sum_{x_1,\ldots, x_n} \sum_{e_1,\ldots,e_n} G^{\omega_0^{z,e}}_{\delta}(x,x_1) \xi_{\epsilon}(x_1,e_1) G^{\omega_0^{z,e}}_{\delta}(x_1+e_1,x_2) \cdots \\
&\qquad \qquad \qquad \qquad \qquad \qquad \times \xi_{\epsilon}(x_n,e_n)G^{\omega_0^{z,e}}_{\delta}(x_n+e_n,x'),
\end{align*}
and
\[
R_n(x,x')=\sum_{x^*\in \Z^d}S_n(x,x^*) \sum_{e^*\in \nu} \xi_{\epsilon}(x^*,e^*) G^{\overline{p},\delta}(x^*+e^*,x').
\]

Consider the transformation
\begin{align*}
S_n(x,x')&=\Bigl(\frac{\pi^{\omega_0^{z,e}}(x')}{\pi^{\omega_0^{z,e}}(x)}\Bigr)^{1/2}\sum_{\substack{x_1,\ldots, x_n \\ e_1,\ldots, e_n}} G^{s,\delta}(x,x_1)\xi_{\epsilon}(x_1,e_1)\Bigl(\frac{\pi^{\omega_0^{z,e}}(x_1)}{\pi^{\omega_0^{z,e}}(x_1+e_1)}\Bigr)^{1/2} \\
         & \times G^{s,\delta}(x_1+e_1,x_2)\cdots \Bigl(\frac{\pi^{\omega_0^{z,e}}(x_n)}{\pi^{\omega_0^{z,e}}(x_n+e_n)}\Bigr)^{1/2}G^{s,\delta}(x_n+e_n,x'),
         \end{align*}
and
\begin{align*}
R_n(x,x')&=\Bigl(\frac{\pi^{\omega_0^{z,e}}(x')}{\pi^{\omega_0^{z,e}}(x)}\Bigr)^{1/2}\sum_{\substack{x_1,\ldots, x_n \\ e_1,\ldots, e_n}} G^{s,\delta}(x,x_1)\xi_{\epsilon}(x_1,e_1)\Bigl(\frac{\pi^{\omega_0^{z,e}}(x_1)}{\pi^{\omega_0^{z,e}}(x_1+e_1)}\Bigr)^{1/2} \\
         & \times G^{s,\delta}(x_1+e_1,x_2)\cdots \Bigl(\frac{\pi^{\omega_0^{z,e}}(x_n)}{\pi^{\omega_0^{z,e}}(x_n+e_n)}\Bigr)^{1/2}G^{\overline{s},\delta}(x_n+e_n,x').
         \end{align*}

Moreover for any $x\in \Z^d$ and $e_i\in \nu$ we get by~(\ref{kap1}) that
\begin{equation}
\label{abc1}
\frac{\pi^{\omega_0^{z,e}}(x)}{\pi^{\omega_0^{z,e}}(x+e_i)}\leq \kappa_1^2 e^{2\lambda},
\end{equation}
and for $x,x' \in \Z^d$ we obtain
\begin{align}
\label{abc2}
     & \sum_{\substack{x_1,\ldots, x_n \\ e_1,\ldots, e_n}} G^{s,\delta}(x,x_1)G^{s,\delta}(x_1+e_1,x_2)\cdots G^{s,\delta}(x_n+e_n,x') \\ \nonumber
\leq & \Bigl(\sum_{x_1} G^{s,\delta}(x,x_1) (2d)\Bigr) \max_{x_*\in \Z^d} \sum_{\substack{x_2,\ldots, x_n \\ e_2,\ldots, e_n}} G^{s,\delta}(x_*,x_2)\cdots G^{s,\delta}(x_n+e_n,x') \\ \nonumber
\leq & \frac {2d} {\min_{x} P^{s,\delta}(x,\Delta)} \max_{x_*\in \Z^d} \sum_{\substack{x_2,\ldots, x_n \\ e_2,\ldots, e_n}} G^{s,\delta}(x_*,x_2)\cdots G^{s,\delta}(x_n+e_n,x') \\ \nonumber
\leq & \cdots \leq \Bigl( \frac{2d} {\gamma_6} \Bigr)^n,
\end{align}
where we used an easy recursion to obtain the last inequality and the fact that $\min_{x} P^{s,\delta}(x,\Delta)\geq \gamma_6$ for $\delta\geq 1/2$ where by~(\ref{def_op1})
\[
\gamma_6= \frac 14 \Bigl(\min_{i} \sum_{f^{(j)}\in\nu\setminus \{e^{(i)}\}}(p^{e^{(i)}}(f^{(j)})^{1/2}-p^{e^{(i)}}(-f^{(j)})^{1/2})^{2}\Bigr)
\]

 Finally using~(\ref{abc1}) and~(\ref{abc2}) in the definition of $S_n(x,x')$ we get
\begin{align*}
     \abs{ S_n(x,x')}&\leq\Bigl(\frac{\pi^{\omega_0^{z,e}}(x')}{\pi^{\omega_0^{z,e}}(x)}\Bigr)^{1/2}\Bigl( \kappa_1^2 e^{2\lambda}\Bigl(\sup_{y,e}\abs{\xi_{\epsilon}(y,e)}\Bigr) \frac {2d} {\gamma_6} \Bigr)^{n+1} \\
               &\leq \Bigl(\frac{\pi^{\omega_0^{z,e}}(x')}{\pi^{\omega_0^{z,e}}(x)}\Bigr)^{1/2} \gamma_7^{n+1},
\end{align*}
for some positive constant $\gamma_7$, depending only on $d$ and $\ell$. We can get a similar estimate for the remaining term $R_n(x,x')$ considering that $P^{s,\delta}(x,\Delta) \sim P^{\overline{s},\delta}(x,\Delta)$. This implies that for $\epsilon<\gamma_7^{-1}/2$ small enough, the series $\sum_{k=0}^{\infty} \epsilon^k\abs{S_k(x,x')}$ is convergent and upper bounded by a constant independent of $\delta$ and that
\begin{align*}
\text{for any $\delta\in[1/2,1)$,}\qquad \abs{G^{\overline{p},\delta}(x,x')-G^{\omega_0^{z,e}}_{\delta}(x,x')}&\leq\sum_{k=1}^{\infty} \epsilon^k\abs{S_k(x,x')} \\
                                                &=\Bigl(\frac{\pi^{\omega_0^{z,e}}(x')}{\pi^{\omega_0^{z,e}}(x)}\Bigr)^{1/2}O(\epsilon),
\end{align*}
where $O(\cdot)$ depends only on $d$ and $\ell$.

Applying this last result for all cases $x=z+e'+e''$ and $x'=z$ yields~(\ref{resultat11}) and thus the result.
\end{proof}

\subsection{First order expansion of the asymptotic speed}

We have now all the necessary tools to compute the asymptotic speed. Applying Proposition~\ref{approx_green_function}, we get 
\[
G_{\delta}^{\tilde p_{z,e,z+e'}}(z+e',z)=G_{\delta}^{\omega_0^{z,e}}(z+e',z)+o_{\delta,z,e,e'}(1),
\]
where the $o_{\delta,z,e,e'}(1)$ verifies,
\begin{equation}\label{def_o1}
\text{for all $\delta\geq 1/2$,}\qquad \qquad \abs{o_{\delta,z,e,e'}(1)}\leq \abs{o_{\epsilon}(1)},
\end{equation}
where the $o_{\epsilon}(1)$ depends only on $d$ and $\ell$ and vanishes as $\epsilon$ goes to $0$.

 Hence putting the previous equation together with~(\ref{kalikow2}) we obtain
\begin{align}
\label{expan_1}
 &{\bf E}_{1-\epsilon}[\1{\mathcal{I}(\omega^{z,e})}\1{(z,2)=\emptyset} G_{\delta}^{\omega^{z,\emptyset}}(0,z) G_{\delta}^{\omega^{z,e}}(z+e',z)]\\ \nonumber
=&(1+o_{\delta,z,e,e'}(1)){\bf E}_{1-\epsilon}[\1{\mathcal{I}}\1{(z,2)=\emptyset} G_{\delta}^{\omega}(0,z)]G_{\delta}^{\omega_0^{0,e}}(e',0) \\ \nonumber 
=&(1+o_{\delta,z,e,e'}(1)){\bf E}[\1{\mathcal{I}} G_{\delta}^{\omega}(0,z)]G_{\delta}^{\omega_0^{0,e}}(e',0),
\end{align}
where we used the following consequence of the first part of Lemma~\ref{yet_another_technicality}
\[
\text{for all $z\in \Z^d$ and $e\in \nu$,} \qquad \abs{\frac{{\bf E}_{1-\epsilon}[\1{\mathcal{I}}G_{\delta}^{\omega}(0,z)]}{{\bf E}_{1-\epsilon}[\1{\mathcal{I}}\1{(z,2)=\emptyset}G_{\delta}^{\omega}(0,z) ]}-1}\leq O(\epsilon).
\]

Adapting the same methods for $z+e$ yields
\begin{align}
\label{expan_2}
 &{\bf E}[\1{\mathcal{I}(\omega^{z,e})} \1{(z,2)=e}G_{\delta}^{\omega^{z,\emptyset}}(0,z+e) G_{\delta}^{\omega^{z,e}}(z+e+e',z)]\\ \nonumber
&=(1+o_{\delta,z,e,e'}(1)){\bf E}[\1{\mathcal{I}}G_{\delta}^{\omega}(0,z+e)]G_{\delta}^{\omega_0^{0,e}}(e+e',0),
\end{align}
where the $o_{\delta,z,e,e'}(\cdot)$ verified~(\ref{def_o1}).

Let us denote
\begin{equation}
\label{def_phi}
\phi(e)=\sum_{e'\in \nu} (p^e(e')-p^{\emptyset}(e')) G^{\omega_0^{0,e}}(e',0),
\end{equation}
and
\begin{equation}
\label{def_psi}
\psi(e)=\sum_{e'\in \nu} (p^{-e}(e')-p^{\emptyset}(e')) G^{\omega_0^{0,e}}(e+e',0).
\end{equation}

Hence inserting the estimates~(\ref{expan_1}) and~(\ref{expan_2}) into the expression of~(\ref{ugly_ass_ub}) we get
\begin{align*}
&{\bf E}_{1-\epsilon}[\1{\mathcal{I}(\omega^{z,e})}G^{\omega^{z,e}}_{\delta}(0,z)]\\
= & (1+o_{\delta,z,e}(1))\Bigl[ {\bf E}_{1-\epsilon}[\1{\mathcal{I}}G_{\delta}^{\omega}(0,z)] (1+\delta\phi(e)) +  {\bf E}_{1-\epsilon}[\1{\mathcal{I}}G_{\delta}^{\omega}(0,z+e)]\delta\psi(e) \Bigr].
 \end{align*} 
 
 Inserting the previous equation in~(\ref{first_order_0}) yields
\begin{align*}
& d_{\delta}^{\widehat{\omega}}(z)-d_{\emptyset}\\
=& \frac{\epsilon(1+o_{\delta,z}(1))}{{\bf E}_{1-\epsilon}[\1{\mathcal{I}}G_{\delta}^{\omega}(0,z)]}\Bigl[{\bf E}_{1-\epsilon}[\1{\mathcal{I}}G_{\delta}^{\omega}(0,z)] \Bigl(\sum_{e\in \nu} (1+\delta\phi(e))(d_{e}-d_{\emptyset}) \Bigr) \\ 
&+{\bf E}_{1-\epsilon}[\1{\mathcal{I}}G_{\delta}^{\omega}(0,z+e)]\Bigl(\sum_{e\in \nu} \delta\psi(e)(d_{e}-d_{\emptyset}) \Bigr)\Bigr]+O_z(\epsilon^2).
\end{align*}

We are not able to derive uniform estimates for $d_{\delta}^{\widehat{\omega}}(z)$, nevertheless we are still able to estimate the asymptotic speed. Recalling Proposition~\ref{kalikow_0}, the previous equation yields
\begin{align}
\label{this_is_long}
& \frac{\sum_{z\in \Z^d} G_{\delta}^{\widehat{\omega}_{\delta}^{\epsilon}}(0,z)\widehat{d}_{\delta}^{\epsilon}(z)}{\sum_{z\in \Z^d} G_{\delta}^{\widehat{\omega}_{\delta}^{\epsilon}}(0,z)} -d_{\emptyset} \\ \nonumber 
=&\epsilon(1+o_{\delta,z}(1))\frac{\sum_{z\in \Z^d} \sum_{e\in \nu}{\bf E}[\1{\mathcal{I}}G_{\delta}^{\omega}(0,z)]  (1+\delta\phi(e))(d_e-d_{\emptyset})}{\sum_{z\in \Z^d} {\bf E}[\1{\mathcal{I}}G_{\delta}^{\omega}(0,z)]}\\ \nonumber
& + \epsilon(1+o_{\delta,z}(1)) \frac{\sum_{z\in \Z^d} \sum_{e\in \nu}{\bf E}[\1{\mathcal{I}}G_{\delta}^{\omega}(0,z+e)] \delta \psi(e)(d_e-d_{\emptyset})}{\sum_{z\in \Z^d} {\bf E}[\1{\mathcal{I}}G_{\delta}^{\omega}(0,z)]}+o_{\delta}(\epsilon)\\ \nonumber
=&\epsilon \sum_{e\in \nu}(1+\delta(\phi(e)+\psi(e)))(d_e-d_{\emptyset}) +o_{\delta}(\epsilon),
\end{align}
since $\sum_z {\bf E}[\1{\mathcal{I}}G_{\delta}^{\omega}(0,z)] = \sum_z {\bf E}[\1{\mathcal{I}}G_{\delta}^{\omega}(0,z+e)]= {\bf P}[\mathcal{I}]/(1-\delta)$. We emphasize that we do actually get a $o_{\delta}(\epsilon)$ such that 
\begin{equation}
\label{def_o12}
\text{for $\delta \geq 1/2$,} \qquad \abs{o_{\delta}(\epsilon)}\leq \abs{o_{\epsilon}(\epsilon)},
\end{equation}
 since it is the sum of $2(2d)^2$ barycenters of all $(o_{\delta,z,e,e'}(\epsilon))_{z\in \Z^d,e,e'\in \nu}$ which verify the bound of~(\ref{def_o1}) and a barycenter of $(O_{z}(\epsilon^2))_{z\in \Z^d}$ verifying~(\ref{def_oo1}).

We can then obtain an expression of the speed using Proposition~\ref{speed_green} by letting $\delta$ go to $1$ in~(\ref{this_is_long})
\begin{equation}
\label{result_0}
v_{\ell}(1-\epsilon)=d_{\emptyset}+\epsilon \sum_{e\in \nu}(1+\delta(\phi(e)+\psi(e)))(d_e-d_{\emptyset})+o(\epsilon),
\end{equation}
since by~(\ref{def_o12}) all $o_{\delta}(\epsilon)$ are smaller than some $o_{\epsilon}(\epsilon)$ uniformly in $\delta\in[1/2,1)$.

\subsection{Simplifying the expression of the limiting velocity}

In order to simplify the expression of the limiting velocity we prove
\begin{lemma}
\label{simplify_tech}
We have
\begin{align*}
&\sum_{e'\in \nu} (p^e(e')-p^{\emptyset}(e')) G^{\omega_0^{0,e}}(e',0) +\sum_{e'\in \nu} (p^{-e}(e')-p^{\emptyset}(e')) G^{\omega_0^{0,e}}(e+e',0) \\
&=(p^{\emptyset}(e)-p^{\emptyset}(-e))(G^{\omega^{0,e}_0}(0,0)-G^{\omega^{0,e}_0}(e,0))-p^{\emptyset}(e).
\end{align*}
\end{lemma}

\begin{proof}
Recalling the notations~(\ref{shorter}), we get
\[
p^e(e')-p^{\emptyset}(e')=\begin{cases} \frac{c(e')c(e)}{\pi^{\emptyset}\pi^{e}} & \text{ if $e\neq e'$,}\\
                                                  -\frac{c(e')}{\pi^{\emptyset}} & \text{ if $e=e'$.}
                                    \end{cases}
\]

Hence we get that,
\begin{equation}
\label{simplify_tech1}
\sum_{e'\neq e} \frac{c(e')c(e)}{\pi^{\emptyset}\pi^{e}} G^{\omega_0^{0,e}}(e',0)= \frac{c(e)}{\pi^{\emptyset}} (G^{\omega_0^{0,e}}(0,0)-1),
\end{equation}
and
\begin{equation}
\label{simplify_tech2}
\sum_{e'\neq e} \frac{c(e')c(-e)}{\pi^{\emptyset}\pi^{-e}} G^{\omega_0^{0,e}}(e+e',0)= \frac{c(-e)}{\pi^{\emptyset}} G^{\omega_0^{0,e}}(e,0).
\end{equation}

Finally using $\frac{c(e)}{\pi^{\emptyset}}=p^{\emptyset}(e)$ and the previous equations, the computations are straightforward.
\end{proof}

Recalling that $p^{\emptyset}(e)-p^{\emptyset}(-e)=d_{\emptyset}\cdot e$ and $1-p^{\emptyset}(e)=\pi^e/\pi^{\emptyset}$, we see that the previous lemma means that
\[
\alpha(e)=\phi(e)+\psi(e)=\frac{\pi^{e}}{\pi^{\emptyset}}+(d_{\emptyset}\cdot e)(G^{\omega^{0,e}_0}(0,0)-G^{\omega^{0,e}_0}(e,0)),
\]
where we used notations from~(\ref{def_phi}) and~(\ref{def_psi}). So~(\ref{result_0}) becomes
\begin{equation}
\label{result_1}
v_{\ell}(1-\epsilon)=d_{\emptyset}+\epsilon \sum_{e\in \nu} \alpha(e)(d_e-d_{\emptyset})+o(\epsilon).
\end{equation}

We still may simplify slightly the expression of the speed we obtained using the following 
\[
\sum_{e\in \nu} \pi^{e} d_e=\sum_{i=1}^{2d} \sum_{e \neq e^{(i)}} c(e)e =(2d-1)\sum_{e\in \nu} c(e)e=(2d-1) \pi^{\emptyset} d_{\emptyset}=\sum_{e \in \nu} \pi^{e} d_{\emptyset},
\]

Inserting this last equation into~(\ref{result_1}) yields 
\[
v_{\ell}(1-\epsilon)=d_{\emptyset}+\epsilon \sum_{e\in \nu} (d_{\emptyset}\cdot e)(G^{\omega^{0,e}_0}(0,0)-G^{\omega^{0,e}_0}(e,0))(d_e-d_{\emptyset})+o(\epsilon),
\]
which proves Theorem~\ref{asymptoticspeed}. \qed

\section{Estimate on Kalikow's environment}
\label{s_tech}

The section is devoted to the proof of Proposition~\ref{quotient_ub} in which we assumed to have fixed $y,z\in \Z^d$, $A\subset \nu$, $A\neq \nu$, $e \in \nu$ and $e_+,e_-\in \nu \cup \{0\}$.  Before entering into the details let us present the main steps of the proof of the previous proposition which are rather similar to the ones in the proof of Proposition~\ref{upperb}. Let us study the numerator of the quotient of Proposition~\ref{quotient_ub}.
\begin{enumerate}
\item The Green functions behave essentially as a hitting probability multiplied by a resistance (normalized by the invariant measure). See~(\ref{demo_step1}).
\item  In order to transform the conditioning around $\C(y)=A$ into $\C(y)=\emptyset$ we use the estimates on resistances of Proposition~\ref{insertedge}. This procedure will essentially give an upper bound on the numerator in Proposition~\ref{quotient_ub} as a finite sum of terms which ressemble the denominator but with a local correlation around $y$ due to the presence of random variables $Z_y^{(i)}$ (see~(\ref{definition_Z})) reminiscent of the random variable \lq\lq $L_z$\rq\rq which appeared in the proof of Proposition~\ref{upperb}. 

Moreover we will need some extra work to get an expression with some sort of independence property between our local correlation term and the other terms appearing in the upper-bound. This is necessary for step (3) of the proof. See~(\ref{finite_sum}) and~(\ref{finite_sum1}) for the upper-bound.

\item We finish the proof by decorrelation lemmas similar to Lemma~\ref{decorrelation} to show that the local correlation terms have a limited effect. This will imply that the numerator and the denominator of Proposition~\ref{quotient_ub} are of the same order. See sub-section~\ref{s_deco}
\end{enumerate}

Compared to Proposition~\ref{upperb} there is an extra difficulty added by the fact that we need to handle two Green functions instead of only one (in some sense we will even have three), hence we will apply Proposition~\ref{insertedge} recursively, this is done in Proposition~\ref{Q_upper_bound}.

Before actually starting the proof, we point out that in addition, we cannot prove directly a decorrelation lemma. Indeed one of the hitting probabilities coming from the Green functions appearing in Proposition~\ref{quotient_ub} behaves badly when a local modification of the environment is made at $y$. Hence we need to transform this hitting probability into an expression which we will be able to decorrelate from a local modification of the environment and this will change slightly the outline of the proof given above. The aim of the next sub-section is to take care of this problem.

\subsection{The perturbed hitting probabilites}

We want to understand the effect of the change of configuration around $y$ on the hitting probabilities $P_{z+e_+}^{\omega_{y,A}^{(z,2)=e}}[T_{y}<\tau_{\delta}]$ and $P_0^{\omega_{y,A}^{(z,2)=\emptyset}}[T_z\leq \tau_{\delta}]$. The former term can be estimated easily. If we denote the (deterministic) set 
\begin{equation}
\label{def_b_star}
B^*(y,k)=\Bigl\{ t\in B(y,k),~t \text{ is connected to }y \text{ in }B^E(y,k)\setminus \{[z,z+e]\} \Bigr\},
\end{equation}
 and
\begin{equation}
\label{notation_PZ}
p_z^{\omega}(y,k)=\begin{cases} \displaystyle{\max_{u \in \partial B^{*}(y,k)}} P_{z+e_+}^{\omega}[T_u=T_{\partial B^{*}(y,k)}<\tau_{\delta}] & \text{ if } z+e_+ \notin B^*(y,k) , \\
                                1 & \text{ otherwise.} \end{cases}
\end{equation}
then for any $k\geq 1$ such that $z+e_+\notin B^*(y,k)$, we have
\begin{equation}
\label{hitting_time1}
 P_{z+e_+}^{\omega_{y,A}^{(z,2)=e}}[T_{y}<\tau_{\delta}] \leq  \rho_d k^d p_z^{\omega^{(z,2)=e}_{y,A}}(y,k),
\end{equation}
the special notation $B^*(y,k)$ is useful because in the configuration $\omega^{(z,2)=e}$ the walker can only reach $B(y,k)\setminus \partial B(y,k)$ (and hence $y$) from $z+e_+$ by entering the ball $B(y,k)$ through $B^*(y,k)$. The technical reason will only appear in the proof of Lemma~\ref{deco_2a}

As we announced previously, the second hitting probability is more difficult to treat. Let us introduce the following notations
\begin{equation}
\label{notation_P1}
p_1^{\omega}(y,k)=\begin{cases} \displaystyle{\max_{u\in \partial B(y,k)}} P_0^{\omega}[T_u=T_{\partial B(y,k)}<\tau_{\delta}] & \text{ if } 0\notin B(y,k), \\
                                1 & \text{ otherwise,} \end{cases}
\end{equation}
\begin{equation}
\label{notation_P2}
p_2^{\omega}(y,k)=\begin{cases} \displaystyle{\max_{u \in \partial B(y,k)}} P_u^{\omega}[T_z<\tau_{\delta}\wedge T_{\partial B(y,k)}^+] & \text{ if } z \notin B(y,k), \\
                                1 & \text{ otherwise.} \end{cases}
\end{equation}

To make notations lighter we also set
\begin{equation}
\label{notation_R1}
\text{for any $x\in \Z^d$,} \qquad R_*^{\omega}(x)=e^{2\lambda x \cdot \vec \ell} R^{\omega}[x\leftrightarrow \Delta],
\end{equation}
and moreover we introduce
\begin{equation}
\label{notation_R5}
\overline{R}_*^{\omega}(y,k)=\begin{cases} \displaystyle{\max_{u\in \partial B(y,k)}} R^{\omega}_*[u \leftrightarrow z \cup \Delta] & \text{ if } z \notin B(y,k),\\
                                1 & \text{ otherwise,} \end{cases}
\end{equation}
where
\begin{equation}
\label{notation_R4}
\text{for any $u\in \Z^d$,} \qquad R^{\omega}_*[u \leftrightarrow z \cup \Delta]=e^{2\lambda u\cdot \vec \ell} R^{\omega}[u\leftrightarrow z \cup \Delta].
\end{equation}

We can obtain an upper-bound on $P_0^{\omega}[T_z< \tau_{\delta}]$ through the following proposition.

\begin{proposition}
\label{hitting_time2}
Take any configuration $\omega$ and set $y,z \in \Z^d$ and $B=B(y,r)$ with $r\geq 1$ and $\delta \geq 1/2$. If $0,z \notin B$ and $P_0^{\omega}[T_z< \tau_{\delta}] >2P_0^{\omega}[T_z< T_{\partial B}\wedge \tau_{\delta}]$, then we have
\[
P_0^{\omega}[T_z< \tau_{\delta}]\leq C_{19} r^{2d} p_1^{\omega}(y,k) p_2^{\omega}(y,k) \overline{R}^{\omega}_*(y,k).
\]

If $0\in B$, $z\notin B$ and $P_0^{\omega}[T_z< \tau_{\delta}] >2P_0^{\omega}[T_z< T_{\partial B}\wedge \tau_{\delta}]$, then
\[
P_0^{\omega}[T_z< \tau_{\delta}]\leq C_{19} r^{2d}p_2^{\omega}(y,k) \overline{R}^{\omega}_*(y,k).
\]

Finally if $0\notin B$ and $z\in B$,
\[
P_0^{\omega}[T_z< \tau_{\delta}]\leq C_{19}r^{2d} p_1^{\omega}(y,k).
\]
\end{proposition}

Thanks to this lemma we can say that $P_0^{\omega}[T_z< \tau_{\delta}]$ is either not influenced much by a local modification around $y$ (in the case where typically the walk will not visit $y$ when it goes from $0$ to $z$), or upper bounded by a product of at most three random variables. Two of them behave as hitting probabilities which are well suited for our future decorrelation purposes, the third random variable is essentially a resistance for which we have estimates as well.

In the case where $P_0^{\omega}[T_z< \tau_{\delta}] >2P_0^{\omega}[T_z< T_{\partial B}\wedge \tau_{\delta}]$, we will not have any issues for the decorrelation lemma.

\begin{proof}
We will only consider the case $0,z \notin B$, the other being similar but simpler. Our hypothesis implies
\[
P_0^{\omega}[T_z< \tau_{\delta}]\leq 2 P_0^{\omega}[T_{\partial B}\leq T_z < \tau_{\delta}],
\]
we can get an upper bound on the right-hand term by Markov's property 
\begin{align}
\label{hitting_decompo}
P_0^{\omega}[T_{\partial B}\leq T_z < \tau_{\delta}]&= \sum_{u\in \partial B} P_0^{\omega}[T_u=T_{\partial B}<\tau_{\delta}]P_u^{\omega}[T_z<\tau_{\delta}] \\ \nonumber
                                         &\leq \abs{\partial B} \max_{u\in \partial B} P_0^{\omega}[T_u=T_{\partial B}<\tau_{\delta}] \max_{u\in \partial B}P_u^{\omega}[T_z<\tau_{\delta}].
\end{align}

Denoting $z_1\rightarrow \cdots \rightarrow z_n$ the event that the $n$ first vertices of $\partial B \cup z \cup \Delta$ visited are, in order, $z_1$, $z_2, \ldots, z_n$, we can write for $u\in \partial B$
\begin{align}
\label{hit_dec_1}
P_u^{\omega}[T_z<\tau_{\delta}]&=E_u^{\omega}\Bigl[\sum_{n}\sum_{z_1,\ldots, z_n \in \partial B } \1{z_1 \rightarrow  \cdots \rightarrow z_n \rightarrow z} \Bigr] \\ \nonumber
 & =\sum_{n}\sum_{z_1,\ldots, z_n \in \partial B } E_u^ {\omega} [ \1{z_1 \rightarrow  \cdots \rightarrow z_n} ]P_{z_n}^ {\omega}[T_z<T_{\partial B}^ +\wedge \tau_{\delta}] \\ \nonumber
 & \leq \max_{v \in \partial B} P_v^{\omega}[T_z<\tau_{\delta}\wedge T_{\partial B}^+] E_u^ {\omega} \Bigl[\sum_{n}\sum_{z_1,\ldots, z_n \in \partial B } \1{z_1 \rightarrow  \cdots \rightarrow z_n} \Bigr] \\ \nonumber
 & = \max_{v \in \partial B} P_v^{\omega}[T_z<\tau_{\delta}\wedge T_{\partial B}^+] G_{\delta,\{ z\}}^{\omega}(u, \partial B),
\end{align}
where 
\begin{equation}
\label{hit_dec_2}
G_{\delta, \{z\}}^{\omega}(u, \partial B)=E_u^{\omega}\Bigl[\sum_{n=0}^{\tau_{\delta} \wedge T_z} \1{X_n \in \partial B} \Bigr] \leq \abs{\partial B} \max_{v\in \partial B} G_{\delta, \{z\}}^{\omega}(v, v).
\end{equation}

Since by Lemma~\ref{lemgreen2},~(\ref{kap1}) and~(\ref{resdelta}) we have for $\delta \geq 1/2$ and any $v\in \partial B$
\begin{equation}
\label{hit_dec_3}
G_{\delta,\{z\}}^{\omega}(v, v) = \pi^{\omega(\delta)}(v)R^ {\omega}(v \leftrightarrow z\cup \Delta) \leq  \gamma_1 \max_{u\in \partial B} R^{\omega}_*[u \leftrightarrow z \cup \Delta].
\end{equation}

 Since $\abs{\partial B} \leq \rho_d r^d$ adding up~(\ref{hit_dec_1}),~(\ref{hit_dec_2}) and~(\ref{hit_dec_3}) we get 
\[
\max_{u \in \partial B} P_u^{\omega}[T_z<\tau_{\delta}] \leq \gamma_2 r^d \max_{u\in \partial B} R^{\omega}_*[u \leftrightarrow z \cup \Delta] \max_{u \in \partial B} P_u^{\omega}[T_z<\tau_{\delta}\wedge T_{\partial B}^+].
\]

Using the previous equation with~(\ref{hitting_decompo}) concludes the proof of the Proposition.
\end{proof}

Recalling the notation from~(\ref{notation_R4}) , let us introduce
\begin{equation}
\label{notation_R6}
\underline{R}_*^{\omega}(y,k)=\begin{cases} \displaystyle{\min_{u\in \partial B(y,k)}} R_*^{\omega}[u \leftrightarrow z \cup \Delta] & \text{ if } z \notin B(y,k),\\
                                1 & \text{ otherwise.} \end{cases}
                                \end{equation}

We do not yet have for the random variable $\overline{R}_*^{\omega_{(y,r),1}}(y,r) $ a property similar to Proposition~\ref{insertedge}. For the future decorrelation part it is in fact better to rewrite $\overline{R}_*^{\omega_{(y,r),1}}(y,r)$ in terms of $\underline{R}^{\omega}_*(y,r')$ and local quantities. This is done in the following lemma.

\begin{lemma}
\label{mult_lemma1}
For any $B=B_E(y,r)$ and $r'\geq r$. Suppose that $y\in K_{\infty}(\omega_{(y,r),1})$ and $\partial B \cap K_{\infty}(\omega) \neq \emptyset$, we have
\[
\overline{R}_*^{\omega_{(y,r),1}}(y,r) \leq 4e^{4\lambda r'}\underline{R}^{\omega_{(y,r),1}}_*(y,r') + C_{20}L_{y,r'}^{C_{21}} e^{C_{22}L_{y,r'}}.
\]
\end{lemma}

\begin{proof}
Let us denote $v\in \partial B(y,r)$ such that
\begin{equation}
\label{subdemo_part1}
\overline{R}_*^{\omega_{(y,r),1}}(y,r)=\max_{u\in \partial B(y,k)}R_*^{\omega_{(y,r),1}}(u \leftrightarrow z \cup \Delta)=R^{\omega_{(y,r),1}}(v \leftrightarrow z \cup \Delta)e^{2\lambda v\cdot \vec \ell},
\end{equation}
applying Proposition~\ref{insertedge1} we get for any $r'\geq r$
\begin{align}
\label{subdemo_part2}
R_*^{\omega_{(y,r),1}}(v \leftrightarrow z \cup \Delta) & \leq 4R^{\omega_{(y,r'),1}}(v \leftrightarrow z \cup \Delta)+C_1L_{y,r'}^{C_2} e^{2\lambda (-y \cdot \vec \ell+ L_{y,r'})} \\ \nonumber
 & \leq 4R^{\omega_{(y,r'),1}}(v \leftrightarrow z \cup \Delta)+C_1L_{y,r'}^{C_2} e^{2\lambda (-v \cdot \vec \ell+2 L_{y,r'})} ,
\end{align}
where we used that $y \cdot \vec \ell \geq v \cdot \vec \ell -r$ and that $L_{y,r'}\geq r'\geq r$ by the third property of Proposition~\ref{threeprop}.

For any $u \in \partial B(y,r')$, let us denote $i_0(\cdot)$ the unit current from $u$ to $z\cup\{\Delta\}$ in $\omega_{(y,r'),1}$ and $\vec{\mathcal Q}$ one of the shortest directed path from $v$ to $u$ included in $ B(y,r')$. Let $\omega_n$ be an increasing exhaustion of subgraphs of $\omega$. Consider the unit flow from $v$ to $z\cup\{\Delta\}$ given by $\theta(e)=i_0(e)+(\1{e\in \vec{\mathcal Q}}-\1{-e \in \vec{\mathcal Q}})$. By taking the trace on $\omega_n \cup \{\delta\}$, $\theta(\cdot)$ induces naturally a familly of unit flows $\theta_n(\cdot)$  from $v$ to $z \cup \{\Delta\} \cup\{\omega\setminus \omega_n\}$  on $\omega_n$, for $n$ large enough. Applying Thompson's principle for $\theta_n$ and taking the limit as $n$ goes to infinity yields 
\begin{equation}
\label{subdemo_part3}
R^{\omega_{(y,r'),1}}(v \leftrightarrow z \cup \Delta)\leq R^{\omega_{(y,r'),1}}(u \leftrightarrow z \cup \Delta)+8r'e^{2\lambda (-y\cdot\vec \ell+ r')}.
\end{equation}

Hence adding up~(\ref{subdemo_part2}) and~(\ref{subdemo_part3}), we get
\begin{align*}
R^{\omega_{(y,r),1}}(v \leftrightarrow z \cup \Delta) \leq & 4\min_{u\in \partial B(y,r')} R^{\omega_{(y,r'),1}}(u \leftrightarrow z \cup \Delta)\\
        &  \qquad \qquad  +\gamma_1(L_{y,r'})^{\gamma_2}\gamma_3^{L_{y,r'}}e^{-2\lambda y\cdot\vec \ell},
\end{align*}
since $L_{y,r'}\geq r'\geq r$.

We get multiplying the left side by $e^{2\lambda v \cdot \vec \ell}$ and the right one by $e^{2\lambda r'}e^{2 \lambda y \cdot \vec \ell}$ (which is greater than $e^{2\lambda v \cdot \vec \ell}$) that
\[
R^{\omega_{(y,r),1}}(v\leftrightarrow z \cup \Delta) e^{2\lambda v \cdot \vec \ell}\leq 4 e^{4\lambda r'}\underline{R}^{\omega_{(y,r),1}}_*(y,r') + \gamma_4(L_{y,r'})^{\gamma_5}e^{\gamma_6 L_{y,r'}}
\]
where we used that $\max_{u \in \partial B(y,r')} e^{2\lambda u \cdot \vec \ell} \leq e^{2\lambda r'} e^{2\lambda y \cdot \vec \ell}$. So by~(\ref{subdemo_part1}) we obtain the lemma.
\end{proof}

\subsection{Quenched estimates on perturbed Green functions}

The aim of this subsection is to complete the first two steps of the sketch of proof at the beginning of Section~\ref{s_tech}. Let us introduce
\begin{equation}
\label{notation_R2}
R_*(z)=R_*^{\omega_{y,A}^{(z,2)=\emptyset}}(z)\text{ and } R_*(y)=R_*^{\omega_{y,A}^{(z,2)=e}}(y),
\end{equation}
where we emphasize those are not functions of $y$ and $z$ which are fixed vertices in this section.

{\it Step (1)}

We reduce our problem of studying Green functions to studying resistances, indeed using Lemma~\ref{lemgreen} and~(\ref{res_green}) we get for $\delta\geq 1/2$,
\begin{equation}
\label{equiv2_green_res}
\frac 1 {\kappa_1} G_{\delta}^{\omega_{y,A}^{(z,2)=\emptyset}}(z,z) \leq R_*(z) \leq 2\kappa_1 G_{\delta}^{\omega_{y,A}^{(z,2)=\emptyset}}(z,z),
\end{equation}
and
\begin{equation}
\label{equiv2_green_res_2}
\frac 1 {\kappa_1} G_{\delta}^{\omega_{y,A}^{(z,2)=e}}(y,y) \leq R_*(y) \leq 2\kappa_1 G_{\delta}^{\omega_{y,A}^{(z,2)=e}}(y,y).
\end{equation}

Moreover we can now easily obtain the first step of our proof since
\begin{equation}
\label{demo_step1}
G_{\delta}^{\omega_{y,A}^{(z,2)=\emptyset}}(0,z)G_{\delta}^{\omega^{(z,2)=e}_{y,A}}(z+e_+,y)
\leq  4\kappa_1^2 P_0^{\omega_{y,A}^{(z,2)=\emptyset}}[T_z\leq \tau_{\delta}]P_{z+e_+}^{\omega_{y,A}^{(z,2)=e}}[T_{y}\leq \tau_{\delta}]R_*(z)R_*(y).
\end{equation}

{\it Step (2)-(a):} Notations 

Now our aim is to remove the condition appearing for the configuration at $y$. This is done in way pretty similar to the first part of the proof of Proposition~\ref{upperb}. As mentionned before, we will apply recursively the resistance estimates of Proposition~\ref{insertedge}, for this we introduce
\[
l^{(0)}_{y}=1,~l^{(1)}_{y}=L_{y,1},~l^{(2)}_{y}=L_{y,l^{(1)}_{y}} \text{ and } l^{(3)}_{y}=L_{y,l^{(2)}_{y}},
\]
$L^{(i)}_{y}(\omega)=l^{(i)}_{y}(\omega^{(z,2)=\emptyset})\vee l^{(i)}_{y}(\omega^{(z,2)=e})$ and $B_{y}^{(i)}=B^E(y,L_{y}^{(i)})$. Moreover we set
\begin{equation}
\label{definition_Z}
Z_{y,k}=C_{23} k^{C_{24}} e^{C_{25}k}e^{2\lambda((z-y)\cdot\vec{\ell})} \text{ and } Z_{y}^{(i)}=Z_{y,L_{y}^{(i)}},
\end{equation}
where and $C_{23}=64 \vee C_1\vee C_{19} \vee C_{20}$, $C_{24}=C_2 \vee C_{21} \vee 2d$ and $C_{25}=4\lambda \vee C_{22}$. Moreover set, for $i=0,\ldots,3$
\begin{equation}
\label{notation_R3}
R_*^{(i)}(y)=R_*^{\omega_{B_{y}^ {(i)},1}^{(z,2)=e}}(y) \text{ and } R_*^{(i)}(z)=R_*^{\omega_{B_{y}^ {(i)},1}^{(z,2)=\emptyset}}(z).
\end{equation}

Also recalling~(\ref{notation_R6}), we set 
\[
\underline{R}_*^{(i)}=\underline{R}_*^{\omega^{(z,2)=\emptyset}_{B_{y}^{(i)},1}}(y,L_{y}^{(i)}) \text{ and } \overline{R}_*^{(i)}=\overline{R}_*^{\omega^{(z,2)=\emptyset}_{B_{y}^{(i)},1}}(y,L_{y}^{(i)}).
\]

 Finally we denote for $i=1,2$ and $j=0,1,2$ 
\[
 p_i^{(j)}=p_i^{\omega^{(z,2)=\emptyset}}(y,L_{y}^{(j)}) \text{ and } p_z^{(j)}=p_z^{\omega^{(z,2)=e}}(y,L_{y}^{(j)}).
\]

Let us state how the inequality previously proved are expressed in terms of $Z_{y}^{(i)}$. From our choice of $Z_{y}^{(i)}$, we can write Proposition~\ref{hitting_time2} as follows: for any $z\in \Z^d$ and $i\in \{0,1,2\}$, 
\begin{equation}
\label{demo_step4}
P_0^{\omega_{y,A}^{(z,2)=\emptyset}}[T_z< \tau_{\delta}]\leq Z_{y}^{(i)} p_1^{(i)}p_2^{(i)}\overline{R}^{(i)}_* +2P_0^{\omega_{y,A}^{(z,2)=\emptyset}}[T_z< T_{\partial B_{y}^ {(i)}}\wedge \tau_{\delta}],
\end{equation}
which is a way to get rid of the conditioning around $y$ for the hitting probabilities.

Also from Lemma~\ref{mult_lemma1} we obtain that for any $y\in K_{\infty}(\omega_{y,A}^{(z,2)=\emptyset})$, we have for any $i\leq j$
\begin{equation}
\label{demo_step8}
\overline{R}_*^{(i)} \leq Z_{y}^{(j)}\underline{R}_*^{(j)} + Z_{y}^{(j+1)}.
\end{equation}

Moreover for $y,z \in K_{\infty}(\omega_{y,A}^{(z,2)=e})=K_{\infty}(\omega_{y,A}^{(z,2)=\emptyset})$, Proposition~\ref{insertedge} imply that for $i\leq j \in \{0,1,2\}$
\begin{equation}
\label{res_maj_1}
R_*(z) \leq R_*^{(i)}(z) \leq  64R_*^{(j)}(z) +  Z_{y}^{(j+1)} \text{ and } R_*(y) \leq R_*^{(i)}(y) \leq  64R_*^{(j)}(y) +Z_{y}^{(j+1)}.
\end{equation}

{\it Step (2)-(b):} Upper-bounding 
\[
P_0^{\omega_{y,A}^{(z,2)=\emptyset}}[T_z< \tau_{\delta}] R_*(z)R_*(y).
\]

Those three inequalities and are enough to study~(\ref{demo_step1}). We recall that the term $P_{z+e_+}^{\omega_{y,A}^{(z,2)=e}}[T_{y}\leq \tau_{\delta}]$ appearing in~(\ref{demo_step1}) has already been treated at~(\ref{hitting_time1}). The equations~(\ref{demo_step4}) and~(\ref{demo_step8}) yield that for any $y,z\in K_{\infty}(\omega_{y,A}^{(z,2)=e})$
\begin{align}
\label{step_init}
&P_0^{\omega_{y,A}^{(z,2)=\emptyset}}[T_z< \tau_{\delta}] R_*(z)R_*(y) \\ \nonumber
\leq &\Bigl(p_1^{(0)}p_2^{(0)}(Z_{y}^{(0)}\underline{R}^{(0)}_*+Z_{y}^{(1)}) +2P_0^{\omega_{y,A}^{(z,2)=\emptyset}}[T_z< T_{\partial B_{y}^ {(0)}}\wedge \tau_{\delta}] \Bigr)R_*(z)R_*(y), 
\end{align}

The idea is now to use recursively~(\ref{res_maj_1}) and~(\ref{demo_step8}) to obtain the following proposition
\begin{proposition}
\label{Q_upper_bound}
For any $\omega$ such that $y,z\in K_{\infty}(\omega_{y,A}^{(z,2)=e})$,
\begin{align*}
\underline{R}^{(0)}_*R_*(z)R_*(y) \leq & C_{26}\Bigl[\underline{R}_*^{(0)}R_*^{(0)}(z)R_*^{(0)}(y) \\
&+ (Z_{y}^{(1)})^2(\underline{R}_*^{(1)}R_*^{(1)}(z)+\underline{R}_*^{(1)}R_*^{(1)}(y)+R_*^{(1)}(z)R_*^{(1)}(y)) \\
      &+  (Z_{y}^{(2)})^4 (\underline{R}_*^{(2)}+R_*^{(2)}(z)+R_*^{(2)}(y))+(Z_{y}^{(3)})^4\Bigr],
\end{align*}
and
\[
R_*(z)R_*(y)
\leq C_{27}\Bigl[R_*^{(1)}(z)R_*^{(1)}(y)+Z_{y}^{(2)}(R_*^{(2)}(z)+R_*^{(2)}(y))+(Z_{y}^{(3)})^2\Bigr].
\]
\end{proposition}

This is an interesting upper bound since the resistances are only multiplied with some local quantities which are in some sense independent of those resistances. More precisely, for example the local quantity~$Z_{y}^{(2)}$ is independent of $R_*^{(2)}(y)$ conditionnaly on $\{L_{y}^{(2)}=k\}$ under the mesure ${\bf P}[\;\cdot\;]$, since $Z_{y}^{(2)}$ depends only on the \lq\lq stopping time\rq\rq $L_{y}^{(2)}$, i.e.~only on the edges of $B^E(y,L_{y}^{(2)})$ by the second property of Proposition~\ref{threeprop}.

\begin{proof}
Let us prove the first upper bound, we use~(\ref{res_maj_1}) to get
\begin{align*}
       &\underline{R}_*^{(0)}R_*(z)R_*(y) \leq \underline{R}_*^{(0)}(64R_*^{(0)}(z)+Z_{y}^{(1)})(64R_*^{(0)}(y)+Z_{y}^{(1)}) \\
\leq & 64^2 \underline{R}_*^{(0)}R_*^{(0)}(z)R_*^{(0)}(y) + 64Z_{y}^{(1)}(\underline{R}_*^{(0)}R_*^{(0)}(z)+\underline{R}_*^{(0)}R_*^{(0)}(y))+(Z_{y}^{(1)})^2\underline{R}_*^{(0)}.
\end{align*}

The first term of the right-hand side is of the form announced in the proposition. We need to simplify the remaining terms, we will continue the expansion for $\underline{R}_*^{(0)}R_*^{(0)}(z)$ (the method is similar for $\underline{R}_*^{(0)}R_*^{(0)}(y)$). Emphasizing that 
\[
\text{for $i=0,1,2$,} \qquad \underline{R}_*^{(i)} \leq \overline{R}_*^{(i)},
\]
where we used Rayleigh's monotonicity principle. We may now use~(\ref{res_maj_1}) and~(\ref{demo_step8}) to get
\begin{align*}
& \underline{R}_*^{(0)}R_*^{(0)}(z) \\
 \leq & (Z_{y}^{(1)}\underline{R}_*^{(1)}+Z_{y}^{(2)})(64R_*^{(1)}(z)+Z_{y}^{(2)}) \\
 \leq & 64\bigl[Z_{y}^{(1)}\underline{R}_*^{(1)}R_*^{(1)}(z)+Z_{y}^{(2)}(Z_{y}^{(1)}\underline{R}_*^{(1)}+R_*^{(1)}(z))+(Z_{y}^{(2)})^2\bigr] \\
 \leq & 64\bigl[Z_{y}^{(1)}\underline{R}_*^{(1)}R_*^{(1)}(z) +(Z_{y}^{(2)} )^2(Z_{y}^{(2)}\underline{R}_*^{(2)}+64R_*^{(2)}(z)+2Z_{y}^{(3)})+(Z_{y}^{(2)})^2\bigr] \\
 \leq & (64)^2 \bigl[ Z_{y}^{(1)}\underline{R}_*^{(1)}R_*^{(1)}(z) +(Z_{y}^{(2)} )^3(\underline{R}_*^{(2)}+R_*^{(2)}(z))+3(Z_{y}^{(3)})^3\bigr],
\end{align*}
where we used that for any $i\leq j$ we have $1\leq Z_{y}^{(i)} \leq Z_{y}^{(j)}$. All terms here are of the same type as in the proposition.

The expansion for the term $(Z_{y}^{(1)})^2\underline{R}_*^{(0)}$ is handled by applying~(\ref{demo_step8}) for $i=0$ and $j=1$. Once again our upper bound is correct.

The second upper bound is similar and simpler since it uses only~(\ref{res_maj_1}), so we skip the details.
\end{proof}

On the event $\{y,z \in K_{\infty}(\omega_{y,A}^{(z,2)=e})\}$ (which will turn out to be verified) we want to give an upper bound of~(\ref{step_init}) with a finite sum of terms of the form
\begin{equation}
\label{grrr_1}
(Z_{y}^{(i)})^{\gamma_1}p_1^{(i)}p_2^{(i)} \underline{R}_*^{(i)}R_*^{(i)}(z)R_*^{(i)}(y),
\end{equation}
and
\begin{equation}
\label{grrr_2}
(Z_{y}^{(i)})^{\gamma_1}P_0^{\omega_{y,A}^{(z,2)=\emptyset}}[T_z< T_{\partial B_{y}^ {(i)}}\wedge \tau_{\delta}]R_*^{(i)}(y)R_*^{(i)}(z),
\end{equation}
for $i\leq 3$ and also similar terms where $\underline{R}_*^{(i)}$, $R_*^{(i)}(z)$ or $R_*^{(i)}(y)$ are possibly replaced by $1$.

Recalling the notations~(\ref{notation_P1}),~(\ref{notation_P2}) and~(\ref{notation_PZ}) we have for $j\in \{z,1,2\}$
\begin{equation}
\label{demo_step6}
\text{for $y\in \Z^d$ and $k_1<k_2$, }\qquad p_j^{\omega}(y,k_1)\leq \rho_d k_2^{d-1} p_j^{\omega}(y,k_2),
\end{equation}
so that for $j\in \{z,1,2\}$ and $k_1<k_2 \in \{0,1,2,3\}$,
\begin{equation}
\label{demo_step7}
p_j^{(k_1)} \leq \rho_d Z_{y}^{(k_2)} p_j^{(k_2)}.
\end{equation}
 
Using the inequalities~(\ref{demo_step6}) and~(\ref{demo_step7}) and Proposition~\ref{Q_upper_bound} we can give an upper-bound of $p_1^{(0)}p_2^{(0)}(Z_{y}^{(0)}\underline{R}^{(0)}_*+Z_{y}^{(1)})R_*(z)R_*(y)$ in term of elements described in~(\ref{grrr_1}). We recall here that Proposition~\ref{Q_upper_bound} can be applied since $y,z \in K_{\infty}(\omega_{y,A}^{(z,2)=e})$ by the hypothesis made just above~(\ref{grrr_1}).

For the second term appearing in~(\ref{step_init}), let us take notice that
\[
P_0^{\omega_{y,A}^{(z,2)=\emptyset}}[T_z< T_{\partial B_{y}^ {(0)}}\wedge \tau_{\delta}]R_*(z)R_*(y)=P_0^{\omega_{y,A}^{(z,2)=\emptyset}}[T_z< T_{\partial B_{y}^ {(0)}}\wedge \tau_{\delta}]R_*^{(0)}(y)R_*^{(0)}(z),
\]
which proves that the left-hand side can be upper-bounded using the terms described in~(\ref{grrr_2}).
 
{\it Step (2)-(c):}  Upper-bounding 
 \[
 \1{\mathcal{I}(\omega^{(z,2)=e}_{y,A})} G_{\delta}^{\omega_{y,A}^{(z,2)=\emptyset}}(0,z)G_{\delta}^{\omega_{y,A}^{(z,2)=e}}(z+e_+,y).
 \]
 
  If this term is positive then
 \begin{enumerate}
 \item $\1{\mathcal{I}(\omega^{(z,2)=e}_{y,A})}>0$ implies that $0 \in  K_{\infty}(\omega_{y,A}^{(z,2)=e})$,
 \item $G_{\delta}^{\omega_{y,A}^{(z,2)=\emptyset}}(0,z)>0$ implies that $0$ is connected to $z$ in $\omega_{y,A}^{(z,2)=\emptyset}$,
 \item $G_{\delta}^{\omega_{y,A}^{(z,2)=e}}(z+e_+,y)>0$ implies that $z+e_+$ is connected to $y$ in $\omega_{y,A}^{(z,2)=e}$,
 \end{enumerate}
which means that $y,z \in K_{\infty}(\omega_{y,A}^{(z,2)=e})=K_{\infty}(\omega_{y,A}^{(z,2)=\emptyset})$.
 
Hence we can use the upper-bound of~(\ref{step_init}) obtained at~(\ref{grrr_1}) and~(\ref{grrr_2}) and insert it into~(\ref{demo_step1}). Using also~(\ref{hitting_time1}) we can show that it is possible to give an upper bound on $\1{\mathcal{I}(\omega^{(z,2)=e}_{y,A})}G_{\delta}^{\omega_{y,A}^{(z,2)=\emptyset}}(0,z)G_{\delta}^{\omega_{y,A}^{(z,2)=e}}(z+e_+,y)$ with a finite sum of terms of the form
\begin{equation}
\label{finite_sum}
\1{\mathcal{I}(\omega^{(z,2)=e}_{y,A})}(Z_{y}^{(i)})^{C_{28}}p_z^{(i)}p_1^{(i)}p_2^{(i)} \underline{R}_*^{(i)}R_*^{(i)}(z)R_*^{(i)}(y),
\end{equation}
and
\begin{equation}
\label{finite_sum1}
\1{\mathcal{I}(\omega^{(z,2)=e}_{y,A})}(Z_{y}^{(i)})^{C_{28}}p_z^{(i)}P_0^{\omega_{y,A}^{(z,2)=\emptyset}}[T_z< T_{\partial B_{y}^ {(i)}}\wedge \tau_{\delta}]R_*^{(i)}(y)R_*^{(i)}(z),
\end{equation}
for $i\in \{0,1,2,3\}$ and also similar terms where $\underline{R}_*^{(i)}$, $R_*^{(i)}(z)$ or $R_*^{(i)}(y)$ are possibly replaced by $1$.

This completes step (2) of the proof of Proposition~\ref{quotient_ub}: the correlation term $Z_{y}^{(i)}$ is associated only with terms with which it has some independence property. Indeed except for $\1{\mathcal{I}(\omega^{(z,2)=e}_{y,A})}$ which is only a minor detail, conditionaly on $\{L_{y}^{(i)}=k\}$
\begin{enumerate}
\item $Z_{y}^{(i)}$ depends only on the \lq\lq stopping time\rq\rq $L_{y}^{(i)}$, i.e.~only on the edges of $B^E(y,k)$ by the second property of Proposition~\ref{threeprop},
\item all the other terms depend only on the edges of $E(\Z^d)\setminus B^E(y,L_{y}^{(i)})$,
\end{enumerate}
so these terms are in fact independent conditionaly on $\{L_{y}^{(i)}=k\}$.

Now we can use those independence properties to prove the third step of our proof that is the decorrelation part. We want to give an upper bound of ${\bf E}\Bigl[\1{\mathcal{I}(\omega^{(z,2)=e}_{y,A})}G_{\delta}^{\omega_{y,A}^{(z,2)=\emptyset}}(0,z)G_{\delta}^{\omega_{y,A}^{(z,2)=e}}(z+e_+,y)\Bigr]$, so we shall look for an upper bound on the expectations of~(\ref{finite_sum}) and~(\ref{finite_sum1}), which is the subject of the next sub-section.

\subsection{Decorrelation part}
\label{s_deco}

Recall the definition of $Z_{y}^{(i)}$ at~(\ref{definition_Z}). Let us prove the first decorrelation lemma.
\begin{lemma}
\label{deco_2a}
We have for $i\in \{0,1,2,3\}$ and $\delta\geq 1/2$
\begin{align*}
&{\bf E}\Bigl[\1{\mathcal{I}(\omega^{(z,2)=e}_{y,A})}(Z_{y}^{(i)})^{C_{28}}p_z^{(i)}p_1^{(i)}p_2^{(i)}\underline{R}_*^{(i)}R_*^{(i)}(z)R_*^{(i)}(y)\Bigr] \\
\leq & C_{29}{\bf E}\Bigl[(L_{y}^{(i)})^{C_{30}}e^{C_{31}L_{y}^{(i)}}\Bigr] {\bf E}\Bigl[\1{\mathcal{I}(\omega^{(z,2)=e}_{y,\emptyset})}G_{\delta}^{\omega^{(z,2)=\emptyset}_{y,\emptyset}}(0,z)G_{\delta}^{\omega^{(z,2)=e}_{y,\emptyset}}(z+e_+,y)\Bigr] e^{C_{32}((y-z) \cdot \vec \ell)^+},
\end{align*}
where $C_{29}$, $C_{30}$, $C_{31}$  and $C_{32}$ depend only on $d$ and $\ell$.
\end{lemma}

This lemma is essentially similar to Lemma~\ref{decorrelation}, since $Z_{y}^{(i)}$ is in fact a function of $L_{y}^{(i)}$. Notice that the second expectation on the right-hand side is equal to the numerator of Proposition~\ref{quotient_ub}.

The same lemma holds, with different constants, if we replace $\underline{R}_*^{(i)}$, $R_*^{(i)}(z)$ or $R_*^{(i)}(y)$ by $1$. Indeed it can be seen using Rayleigh's monotonicity principle that for $\delta\geq 1/2$, these three quantities are lower bounded by 
\[
R^{\omega_0}(0\leftrightarrow \Delta ) \wedge \min_{k \in \N} \min_{u\in \partial B(0,k),~z\notin B(0,k)} R_*^{\omega_0}(u\leftrightarrow z \cup \Delta) \geq \gamma_1,
\]
where $\gamma_1$ can be chosen independent of $y$, $i$, $z$ and $A$. Indeed by Lemma~\ref{lemgreen2}, 
\[
R_*^{\omega_0}(u\leftrightarrow z \cup \Delta)\geq \gamma_2 G^{\omega_0}_{\delta, \{z\}}(u,u)\geq \gamma _2.
\]

\begin{proof}
We recall $L_{y}^{(i)}<\infty$ by Proposition~\ref{threeprop}.

Let us condition on the event $\{L_{y}^{(i)}=k\}$ for $k<\infty$. First suppose that $0 \notin B(y,k)$, $z\notin B(y,k)$ and $z+e_+\notin B^*(y,k)$, where we used a notation appearing above~(\ref{notation_PZ}). Recalling the notations~(\ref{notation_PZ}),~(\ref{notation_P1}) and~(\ref{notation_P2}), we may denote $x_0 \in \partial B^*(y,k)$ and $x_1,~x_2\in \partial B(y,k)$ such that
\begin{align*}
p_z^{\omega}(y,k)&=P_{z+e_+}^{\omega^{(z,2)=e}}[T_{x_0}=T_{\partial B^*(y,k)}<\tau_{\delta}], \\
p_1^{\omega}(y,k)&=P_0^{\omega^{(z,2)=\emptyset}}[T_{x_1}=T_{B(y,k)}<\tau_{\delta}], \\
p_2^{\omega}(y,k)&=P_{x_2}^{\omega^{(z,2)=\emptyset}}[T_{z}<\tau_{\delta}\wedge T_{B(y,k)}^+],
\end{align*}
where $x_0$ is connected to $y$ in $B^E(y,k)\setminus [z,z+e]$ and we denote $\mathcal{P}_0$ one of the corresponding shortest such paths (hence of length $\leq k+2$). This is possible by the definition of $\partial B^*(y,k)$ at~(\ref{def_b_star}).

  We also introduce the event 
  \[
  \{0\Leftrightarrow y \Leftrightarrow \infty\}= \Bigl\{ 0 \leftrightarrow \partial B(y,k),\ \partial B(y,k)\stackrel{\omega^{(y,k),0}}{\leftrightarrow} \infty \Bigr\} ,
  \]
   which is true when $ \1{\mathcal{I}(\omega^{y,A})}p_1^{(i)}$ is positive. Moreover let us set $x_3$ connecting $\partial B(y,k)$ to infinity without edges of $B(y,k)$. Thus
\begin{align*}
&{\bf E}\Bigl[ \1{\mathcal{I}(\omega^{y,A})}(Z_{y}^{(i)})^{C_{28}}p_z^{(i)}p_1^{(i)}p_2^{(i)}\underline{R}_*^{(i)}R_*^{(i)}(y)R_*^{(i)}(z)\mid L_{y}^{(i)}=k\Bigr] \\
\leq & \gamma_1 k^{\gamma_2}e^{\gamma_3 k} e^{\gamma_4((z-y)\cdot\vec \ell)^+ } {\bf E}\Bigl[\1{0\Leftrightarrow y \Leftrightarrow \infty}p_z^{(i)}p_1^{(i)}p_2^{(i)}  \underline{R}_*^{(i)}R_*^{(i)}(y)R_*^{(i)}(z)\mid L_{y}^{(i)}=k\Bigr],
\end{align*}
where the integrand of the right-hand side depends only on the edges of $E(\Z^d) \setminus B^E(y,k)$, so that the conditionning inside the corresponding ball can be modified. 

We emphasize that seemingly $p_z^{(i)}$ may depend on the state of the edges in $B^E(y,k)$, but the walk cannot reach $B(y,k)\setminus \partial B(y,k) $ without going through $\partial B^*(y,k)$. Hence $p_z^{(i)}$ can only depend on the edges of $B^E(y,k)$ through the transition probabilities in $\omega^{(z,2)=e}$ of a vertex in $\partial B(y,k) \setminus  \partial B^*(y,k)$. But if such a vertex exists it is unique and the only edge adjacent
to this vertex which lies in $B^E(y,k)$ is necessarily $[z,z+e]$ and is closed in $\omega^{(z,2)=e}$. Hence there is in fact no dependence.

Let us denote  $\mathcal{P}_1$, $\mathcal{P}_2$ and $\mathcal{P}_3$ one of the shortest paths from respectively $x_1$, $x_2$ and $x_3$ to $y$ and $\mathcal{P}=\mathcal{P}_0 \cup\mathcal{P}_1 \cup \mathcal{P}_2\cup \mathcal{P}_3\cup \{y+e,~e\in \nu\}$. Hence we need to control
\begin{align*}
 &{\bf E}\Bigl[ \1{0\Leftrightarrow y \Leftrightarrow \infty}p_z^{(i)}p_1^{(i)}p_2^{(i)} \underline{R}_*^{(i)}R_*^{(i)}(z)R_*^{(i)}(y)\mid L_{y}^{(i)}=k\Bigr]\\
\leq &{\bf E}\Bigl[\1{0\Leftrightarrow y \Leftrightarrow \infty}P_{z+e_+}^{\omega^{(z,2)=e}}[T_{x_0}=T_{\partial B^*(y,k)}<\tau_{\delta}]
P_0^{\omega^{(z,2)=\emptyset}}[T_{x_1}=T_{\partial B(y,k)}<\tau_{\delta}] \\
 & \times P_{x_2}^{\omega^{(z,2)=\emptyset}}[T_{z}<\tau_{\delta}\wedge T_{\partial B(y,k)}^+] R_*^{\omega^{(y,k),1}}[x_2 \leftrightarrow z\cup \Delta]R_*^{(0)}(z)R_*^{(0)}(y) \mid \mathcal{P} \in \omega \Bigr] \\
\leq & 2^{4k+2d+2}{\bf E}\Bigl[\1{0\Leftrightarrow y \Leftrightarrow \infty}\1{\mathcal{P} \in \omega}P_0^{\omega_{y,\emptyset}^{(z,2)=\emptyset}}[T_{x_1}=T_{\partial B(y,k)}<\tau_{\delta}]R_*^{(0)}(z)  R_*^{(0)}(y)  \\
 & \times P_{x_2}^{\omega_{y,\emptyset}^{(z,2)=\emptyset}}[T_{z}<\tau_{\delta}\wedge T_{\partial B(y,k)}^+]  P_{z+e_+}^{\omega_{y,\emptyset}^{(z,2)=e}}[T_{x_0}=T_{\partial B^*(y,k)}<\tau_{\delta}] R_*^{\omega_{y,\emptyset}^{(z,2)=\emptyset}}[x_2 \leftrightarrow z\cup \Delta] \Bigr] ,
\end{align*}
where we used that
\begin{enumerate}
\item ${\bf P}[\mathcal{P}\in \omega] \geq 2^{-(4k+2d+2)}$,
\item equalities such as $P_{z+e_+}^{\omega^{(z,2)=e}}[T_{x_0}=T_{\partial B^*(y,k)}<\tau_{\delta}]=P_{z+e_+}^{\omega_{y,\emptyset}^{(z,2)=e}}[T_{x_0}=T_{\partial B^*(y,k)}<\tau_{\delta}]$,
\item Rayleigh's monotonicity principle to say, for example, that 
\[
R_*^{(i)}(y)\leq R_*^{(0)}(y) \text{ and } R_*^{\omega_{(y,k),1}^{(z,2)=\emptyset}}[x_2 \leftrightarrow z\cup \Delta]\leq R_*^{\omega_{y,\emptyset}^{(z,2)=\emptyset}}[x_2 \leftrightarrow z\cup \Delta].
\]
\end{enumerate}

Using Lemma~\ref{lemgreen} and $x_2\in B(y,k)$, we get
\begin{align*}
& P_{x_2}^{\omega_{y,\emptyset}^{(z,2)=\emptyset}}[T_{z}<\tau_{\delta}\wedge T_{\partial B(y,k)}^+] R_*^{\omega_{y,\emptyset}^{(z,2)=\emptyset}}[x_2 \leftrightarrow z\cup \Delta]\\
\leq &\gamma_5 P_{x_2}^{\omega_{y,\emptyset}^{(z,2)=\emptyset}}[T_z<\tau_{\delta}\wedge T_{x_2}^+]  \Bigl(P_{x_2}^{\omega_{y,\emptyset}^{(z,2)=\emptyset}}[T_z \wedge \tau_{\delta} < T_{x_2}^+]\Bigr)^{-1} \\
=& \gamma_5P_{x_2}^{\omega_{y,\emptyset}^{(z,2)=\emptyset}}[T_z<\tau_{\delta}\wedge T_{x_2}^+ \mid T_z \wedge \tau_{\delta} < T_{x_2}^+] \\
=& \gamma_5 P_{x_2}^{\omega_{y,\emptyset}^{(z,2)=\emptyset}}[T_{z}<\tau_{\delta}],
\end{align*}
where for the last equality we simply notice that the probability of the event $\{T_{z}<\tau_{\delta}\}$ can be computed on the last excursion from $x_2$ before reaching $\Delta$ or $z$. Moreover on $\omega$ such that $\1{\mathcal{P} \in \omega}$,
\[
P^{\omega_{y,\emptyset}^{(z,2)=\emptyset}}_{x_1}[T_{x_2}<\tau_{\delta}]\geq (\delta \kappa_0)^{2k},
\]
and putting these last two equations together we get
\begin{align*}
& P_0^{\omega_{y,\emptyset}^{(z,2)=\emptyset}}[T_{x_1}=T_{\partial B(y,k)}<\tau_{\delta}]P_{x_2}^{\omega_{y,\emptyset}^{(z,2)=\emptyset}}[T_{z}<\tau_{\delta}\wedge T_{\partial B(y,k)}^+] R_*^{\omega_{y,\emptyset}^{(z,2)=\emptyset}}[x_2 \leftrightarrow z\cup \Delta] \\
\leq & \gamma_5(\delta \kappa_0)^{-2k}P_0^{\omega_{y,\emptyset}^{(z,2)=\emptyset}}[T_{x_1}<\tau_{\delta}]P^{\omega_{y,\emptyset}^{(z,2)=\emptyset}}_{x_1}[T_{x_2}<\tau_{\delta}]P_{x_2}^{\omega_{y,\emptyset}^{(z,2)=\emptyset}}[T_{z}<\tau_{\delta}] \\
\leq & \gamma_5(\delta \kappa_0)^{-2k} P_0^{\omega_{y,\emptyset}^{(z,2)=\emptyset}}[T_z<\delta].
\end{align*}

In a similar way, we get by Markov's property that
\[
P_{z+e_+}^{\omega_{y,\emptyset}^{(z,2)=e}}[T_{x_0}=T_{\partial B^*(y,k)}<\tau_{\delta}]\leq (\delta\kappa_0)^{-(k+2)}P_{z+e_+}^{\omega_{y,\emptyset}^{(z,2)=e}}[T_{y}< \tau_{\delta}].
\]

Finally
\[
\1{0\Leftrightarrow y \Leftrightarrow \infty}\1{\mathcal{P} \in \omega}\leq \1{\mathcal{I}}.
\]

Hence for $\omega$ such that $\mathcal{P}\in \omega$ we have, using $\delta \geq 1/2$
\begin{align*}
& \1{0\Leftrightarrow y \Leftrightarrow \infty}\1{\mathcal{P} \in \omega} P_0^{\omega_{y,\emptyset}^{(z,2)=\emptyset}}[T_{x_0}=T_{\partial B(y,k)}<\tau_{\delta}]
P_{z+e_+}^{\omega_{y,\emptyset}^{(z,2)=e}}[T_{x_0}=T_{\partial B^*(y,k)}<\tau_{\delta}] \\
 & \qquad \qquad \times P_{x_2}^{\omega_{y,\emptyset}^{(z,2)=\emptyset}}[T_{z}<\tau_{\delta}\wedge T_{\partial B(y,k)}^+] R^{\omega_{y,\emptyset}^{(z,2)=\emptyset}}_*[x_2 \leftrightarrow z\cup \Delta]R_*^{(0)}(z)R_*^{(0)}(y)\\
\leq & \gamma_6 (2/\kappa_0)^{3k} \1{\mathcal{I}}P_0^{\omega_{y,\emptyset}^{(z,2)=\emptyset}}[T_z< \tau_{\delta}]P_{z+e_+}^{\omega_{y,\emptyset}^{(z,2)=e}}[T_{y}<\delta]R_*^{(0)}(z)R_*^{(0)}(y) \\
\leq &  \gamma_7 e^{\gamma_8 k}  \1{\mathcal{I}} G_{\delta}^{\omega_{y,\emptyset}^{(z,2)=\emptyset}}(0,z)G_{\delta}^{\omega_{y,\emptyset}^{(z,2)=e}}(z+e_+,y),
\end{align*}
where we used that $R_*^{(0)}(y)= R^{\omega_{y,\emptyset}^{(z,2)=e}}[y \leftrightarrow \Delta]$,~(\ref{equiv2_green_res}) and~(\ref{equiv2_green_res_2}).

The result follows by integrating over all possible values of $L_{y}^{(i)}$, since we have just proved that
 \begin{align*}
&{\bf E}\Bigl[ \1{\mathcal{I}(\omega^{y,A})}(Z_{y}^{(i)})^{C_{28}}p_z^{(i)}p_1^{(i)}p_2^{(i)}\underline{R}_*^{(i)}R_*^{(i)}(z)R_*^{(i)}(y)\mid L_{y}^{(i)}=k\Bigr] \\
\leq & \gamma_9 k^{\gamma_{10}}e^{\gamma_{11}k} {\bf E}\Bigl[ \1{\mathcal{I}(\omega^{y,\emptyset})}G_{\delta}^{\omega^{(z,2)=\emptyset}_{y,\emptyset}}(0,z)G_{\delta}^{\omega^{(z,2)=e}_{y,\emptyset}}(z+e_+,y)\Bigr] e^{\gamma_{12}((y-z) \cdot \vec \ell)^+}.
\end{align*}

For the remaining cases, we proceed as follows
\begin{enumerate}
\item if $0 \in B(y,k)$, then we formally replace $P_0^{\omega^{(z,2)=\emptyset}}[T_x=T_{\partial B(z,k)}<\tau_{\delta}]$ by $1$ for any $x\in \partial B(z,k)$ and $x_1$ by $0$,
\item if $z+e_+\notin B^*(y,k)$, then we formally replace $P_{z+e^+}^{\omega^{(z,2)=e}}[T_x=T_{\partial B^*(z,k)}<\tau_{\delta}]$ by $1$ for any $x\in \partial B^*(z,k)$ and $x_0$ by $z+e^+$,
\item if $z\in B(y,k)$, then we formally replace $P_{x_2}^{\omega^{(z,2)=\emptyset}}[T_z=T_{\partial B(z,k)}^+\wedge \tau_{\delta}]$ by $1$ for any $x\in \partial B(z,k)$, $R_*^{\omega_{y,\emptyset}^{(z,2)=\emptyset}}[x_2 \leftrightarrow z\cup \Delta]$ by $1$ and $x_2$ by $z$,
\end{enumerate}
and the previous proof carries over easily.
\end{proof}

We need another decorrelation lemma, which is essentially similar to the previous one but simpler to prove.

\begin{lemma}
\label{deco_2b}
We have for $i\in \{0,1,2,3\}$ and $\delta \geq 1/2$,
\begin{align*}
&{\bf E}\Bigl[ \1{\mathcal{I}(\omega^{(z,2)=e}_{y,A})}(Z_{y}^{(i)})^{C_{28}}p_z^{(i)}P_0^{\omega_{y,A}^{(z,2)=\emptyset}}[T_z< T_{\partial B_{y}^ {(i)}}\wedge \tau_{\delta}]R_*^{(i)}(y)R_*^{(i)}(z)\Bigr] \\
\leq &C_{33}{\bf E}\Bigl[(L_{y}^{(i)})^{C_{34}}e^{C_{35}L_{y}^{(i)}}\Bigr] {\bf E}\Bigl[\1{\mathcal{I}(\omega^{(z,2)=e}_{y,\emptyset})}G_{\delta}^{\omega^{(z,2)=\emptyset}_{y,\emptyset}}(0,z)G_{\delta}^{\omega^{(z,2)=e}_{y,\emptyset}}(z+e_+,y)\Bigr] e^{C_{36}((y-z) \cdot \vec \ell)^+},
\end{align*}
where the constants depend only on $d$ and $\ell$.
\end{lemma}

\begin{proof}
Once again we condition on $\{L_{y}^{(i)}=k\}$ for $k<\infty$ and suppose that $0\notin B(y,k)$ and $z\notin B(y,k)$, the other cases can be handled in the same way as before. We see that
\[
\1{\mathcal{I}(\omega^{y,A})}\leq \1{0\Leftrightarrow y \Leftrightarrow \infty},
\]
and we denote $x_0,x_1 \in \partial B(y,k)$ such that
\[
p_z^{(i)}= P_z^{\omega^{(z,2)=e}}[T_{x_0}=T_{\partial B(y,k)}<\tau_{\delta}],
\]
and $x_1$ is connected to $\infty$ without edges from $B(y,k)$. Moreover denote $\mathcal{P}_0$ one of the shortest paths connecting $x_0$ to $y$ and $\mathcal{P}_1$ one of the shortest paths connecting $x_1$ to $y$.

Then, using the same type of arguments as in the proof of Lemma~\ref{deco_2a}, we get for $\mathcal{P}=\mathcal{P}_0 \cup \mathcal{P}_1\cup \{y+e,~e\in \nu\}$, on $\omega$ such that $\{L_{y}^{(i)}=k\}$,
\begin{align*}
&{\bf E}\Bigl[\1{\mathcal{I}(\omega^{y,A})} (Z_{y}^{(i)})^{C_{28}}p_z^{(i)}P_0^{\omega_{y,A}^{(z,2)=\emptyset}}[T_z< T_{\partial B_{y}^ {(i)}}\wedge \tau_{\delta}] R_*^{(i)}(y)R_*^{(i)}(z) \mid L_{y}^{i}=k \Bigr]\\
\leq & \gamma_1 k^{\gamma_2} e^{\gamma_3 k} e^{\gamma_4((y-z)\cdot \vec \ell)^+} {\bf E}\Bigl[\1{0\Leftrightarrow y \Leftrightarrow \infty} P_z^{\omega^{(z,2)=e}}[T_{x_0}=T_{\partial B^*(y,k)}<\tau_{\delta}]\\ 
& \qquad \qquad\qquad \qquad \times P_0^{\omega_{y,A}^{(z,2)=\emptyset}}[T_z< T_{\partial B(y,k)}\wedge \tau_{\delta}] R_*^{(0)}(y)R_*^{(0)}(z)\mid \mathcal{P} \in \omega \Bigr] \\
\leq & \gamma_1 k^{\gamma_2} 2^{2k+2d+2} e^{\gamma_3 k} e^{\gamma_4((y-z)\cdot \vec \ell)^+}{\bf E}\Bigl[\1{ \mathcal{P} \in \omega}\1{0\Leftrightarrow y \Leftrightarrow \infty} \\ 
& \qquad \times P_z^{\omega^{(z,2)=e}}[T_{x_0}=T_{\partial B^*(y,k)}<\tau_{\delta}] P_0^{\omega_{y,A}^{(z,2)=\emptyset}}[T_z< T_{\partial B(y,k)}\wedge \tau_{\delta}] R_*^{(0)}(y)R_*^{(0)}(z) \Bigr].\end{align*}

Now on $\omega$ such that $\{\mathcal{P} \in \omega\}$, we have
\[
P_0^{\omega_{y,A}^{(z,2)=\emptyset}}[T_z< T_{\partial B_{y}^ {(i)}}\wedge \tau_{\delta}]= P_0^{\omega_{y,\emptyset}^{(z,2)=\emptyset}}[T_z< T_{\partial B_{y}^ {(i)}}\wedge \tau_{\delta}] \leq P_0^{\omega_{y,\emptyset}^{(z,2)=\emptyset}}[T_z<\tau_{\delta}],
\]
and
\[
P_z^{\omega^{(z,2)=e}}[T_{x_0}=T_{\partial B^*(y,k)}<\tau_{\delta}](\delta \kappa_0)^{k}\leq P_{z+e_+}^{\omega^{(z,2)=\emptyset}}[T_{y}<\tau_{\delta}].
\]

Since we also have $\1{ \mathcal{P} \in \omega}\1{0\Leftrightarrow y \Leftrightarrow \infty}\leq\1{\mathcal{I}(\omega^{(z,2)=\emptyset})}$ and $\delta \geq 1/2$ so that,
\begin{align*}
&\1{\mathcal{P} \in \omega}\1{0\Leftrightarrow y \Leftrightarrow \infty} p_z^{\omega^{(z,2)=e}}(y,k)P_0^{\omega_{y,A}^{(z,2)=\emptyset}}[T_z< T_{\partial B(y,k)}\wedge \tau_{\delta}]\\ 
& \qquad \qquad \qquad \qquad \qquad \qquad \times R_*^{(0)}(y)R_*^{(0)}(z) \\
\leq & \gamma_5 k^{\gamma_6}e^{\gamma_7 k} \1{\mathcal{I}(\omega^{y,\emptyset})}G_{\delta}^{\omega^{(z,2)=\emptyset}_{y,\emptyset}}(0,z)G_{\delta}^{\omega^{(z,2)=e}_{y,\emptyset}}(z+e_+,y) ,
\end{align*}
and the results follows by integration over the values of $L_{y}^{(i)}$.
\end{proof}

Now, as we did to obtain the continuity of the speed, we need to show that the contribution due to the local modifications of the environment has a limited effect. Hence we want to prove that the expectations appearing in Lemma~\ref{deco_2a} and Lemma~\ref{deco_2b} are finite for $\epsilon$ small enough. This is proved using the following lemma.

\begin{lemma}
\label{integrability}
For $\epsilon_9$ small enough and any $\epsilon<\epsilon_9$ we have
\[
{\bf E}[(L_{y}^{(i)})^{C_{30}+C_{34}}e^{(C_{31}+C_{35})L_{y}^{(i)}}]<C_{37},
\]
where $C_{37}$ depends only on $d$ and $\ell$.
\end{lemma}
\begin{proof}
Since $L_{y}^{(i)} \leq L_{y}^{(3)}$, it is enough to give an upper bound on the tail of $L_{y}^{(3)}$, we have 
\begin{align*}
{\bf P}[L_{y}^{(3)}\geq n] & \leq {\bf P}[L_{y}^{(3)}\geq n ,~L_{y}^{(2)}\leq n/(2C_8)] \\
 &  + {\bf P}[ L_{y}^{(2)} \geq n/(2C_8),~L_{y}^{(1)}\leq n/(2C_8)^2] \\
                         & + {\bf P}[L_{y}^{(1)} \geq n/(2C_8)^2 ],
\end{align*}
and recalling Proposition~\ref{perco2} and Proposition~\ref{perco3} we get for $A=B(x,r)$
\[
{\bf P}_{1-\epsilon}[ L_A(\omega^{(z,2)=\emptyset})\vee L_A(\omega^{(z,2)=e}) \geq n+C_8r ] \leq 2C_9r^d n \alpha(\epsilon)^n,
\]
so that we may use the second property of Proposition~\ref{threeprop}
\[
{\bf P}[L_{y}^{(3)}\geq n] \leq 6C_9 \Bigl(\frac n {2C_8}\Bigr)^d n \alpha(\epsilon)^{f(n)},
\]
where $f(n)=(n/(2C_8)^2-C_8)$ and $\alpha(\epsilon)$ can be arbitrarily small if we take $\epsilon$ small enough. The result follows easily.
\end{proof}

Now, Proposition~\ref{quotient_ub} follows from the decomposition obtained at~(\ref{finite_sum}) and~(\ref{finite_sum1}), the decorrelation part being handled by Lemma~\ref{deco_2a}, Lemma~\ref{deco_2b} where the multiplicative terms appearing in these lemmas are finite by Lemma~\ref{integrability} for $\epsilon$ small enough.

\section{An increasing speed}
\label{s_simplification}

We want to prove Proposition~\ref{incr_speed} and show that the walk slows down when we percolate, i.e.~$v_{\ell}(1)\cdot v_{\ell}'(1)>0$ under certain conditions. We recall $J^e=G^{\omega_0}(0,0)-G^{\omega_0}(e,0)>0$ and we introduce $J^e_e=G^{\omega^{0,e}_0}(0,0)-G^{\omega^{0,e}_0}(e,0)>0$.

We use~(\ref{green_exp}) to prove that
\begin{align*}
G^{\omega_0^{0,e}}(0,0)=&G^{\omega_0}(0,0)+G^{\omega_0}(0,0)\sum_{e'\in  \nu}(p^e(e')-p^{\emptyset}(e'))G^{\omega_0^{0,e}}(e',0)\\ & +G^{\omega_0}(0,e)\sum_{e'\in \nu}(p^{-e}(e')-p^{\emptyset}(e'))G^{\omega_0^{0,e}}(e+e',0),
\end{align*}
and
\begin{align*}
G^{\omega_0^{0,e}}(e,0)=&G^{\omega_0}(e,0)+G^{\omega_0}(e,0)\sum_{e'\in  \nu}(p^e(e')-p^{\emptyset}(e'))G^{\omega_0^{0,e}}(e',0)\\ &+G^{\omega_0}(e,e)\sum_{e'\in \nu}(p^{-e}(e')-p^{\emptyset}(e'))G^{\omega_0^{0,e}}(e+e',0).
\end{align*}

Now, recalling the proof of Lemma~\ref{simplify_tech} (in particular~(\ref{simplify_tech1}) and~(\ref{simplify_tech2})), noticing the relations, $G^{\omega_0}(e,e)=G^{\omega_0}(0,0)$ and by reversibility $G^{\omega_0}(e,0)=(\pi^{\omega_0}(0)/\pi^{\omega_0}(e))G^{\omega_0}(0,e)=(c(e)/c(-e))G^{\omega_0}(0,e)$, we get
\begin{align*}
J^e_e=&J^e+G^{\omega_0}(0,0)\bigl[p(e)(G^{\omega_0^{0,e}}(0,0)-1)-p(e)G^{\omega_0^{0,e}}(e,0)\\
 & \qquad \qquad \qquad \qquad-(p(-e)G^{\omega_0^{0,e}}(e,0)-p(-e)G^{\omega_0^{0,e}}(0,0))\bigr] \\
            &+ G^{\omega_0}(e,0)\bigl[(c(e)/c(-e))(p(-e)G^{\omega_0^{0,e}}(e,0)-p(-e)G^{\omega_0^{0,e}}(0,0))\\
 & \qquad \qquad \qquad \qquad -(p(e)(G^{\omega_0^{0,e}}(0,0)-1)-p(e)G^{\omega_0^{0,e}}(e,0))\bigr],
\end{align*}
which, recalling $p(e)c(-e)=p(-e)c(e)$, means that
\[
J_e^e=J^e+G^{\omega_0}(0,0)((p(e)+p(-e))J^e_e-p(e))+G^{\omega_0}(e,0)(-2p(e)J^e_e+p(e)).
\]

Now rewriting, using reversibility $p(e)G^{\omega_0}(e,0)=p(-e)G^{\omega_0}(0,e)=p(-e)G^{\omega_0}(-e,0)$, we get
\[
J_e^e=J^e+p(e)J^eJ^e_e+p(-e)J^{-e}J^e_e-p(e)J_e,
\]
i.e.
\begin{equation}
\label{Jee}
J_e^e=\frac{(1-p(e))J^e}{1-p(e)J^e-p(-e)J^{-e}}.
\end{equation}

In order to obtain the alternative form of the derivative we only need to rewrite the term $1-p(e)=\pi^e/\pi^{\emptyset}$ using
\[
\pi^e(d_{\emptyset}-d_e)=\pi^e\Bigl(-\sum_{e'\neq e} \frac{c(e')c(e)}{\pi^{\emptyset}\pi^e} e' +\frac{c(e)}{\pi^{\emptyset}} e\Bigr)=c(e)(e-d_{\emptyset}),
\]
hence recalling~(\ref{Jee}) we get
\[
J_e^e (v_{\ell}(1)-d_e)=\frac{p(e)J^e}{1-p(e)J^e-p(-e)J^{-e}}(e-d_{\emptyset}),
\]
which proves the first part of Proposition~\ref{incr_speed}. 

Now, we need to show that this derivative is in the same direction as $v_{\ell}(1)$, for this let us first notice that
\begin{align*}
& 1-p(e)J^e-p(-e)J^{-e} \\=&1-G^{\omega_0}(0,0)(p(e)P_e^{\omega_0}[T_0^+=\infty]+p(-e)P_{-e}^{\omega_0}[T_0^+=\infty])>0,
\end{align*}
since $G^{\omega_0}(0,0)^{-1}=P_0^{\omega_0}[T_0^+=\infty]=\sum_{e'\in \nu}p(e')P_{e'}^{\omega_0}[T_0^+=\infty]$.

Notice that the quantity in the previous display is the same for $e$ and $-e$. 

Now, fix $e\in \nu$ such that $e\cdot d_{\emptyset}>0$, we will show that the common contribution of the terms corresponding to $e$ and $-e$ in the derivative have a positive scalar product with $d_{\emptyset}$ under our assumptions $v_{\ell}(1)$. In fact it is
\[
H(\abs{e}):=(d_{\emptyset}\cdot e)\Bigl[\frac{p(e)J^e+p(-e)J^{-e}}{1-p(e)J^e-p(-e)J^{-e}}e-\frac{p(e)J^e-p(-e)J^{-e}}{1-p(e)J^e-p(-e)J^{-e}}d_{\emptyset}\Bigr],
\]
and since $\beta(\abs{e})=:(d_{\emptyset}\cdot e)/(1-p(e)J^e-p(-e)J^{-e})>0$ we get
\[
H(\abs{e})\cdot d_{\emptyset}=\beta(\abs{e})\bigl[(p(e)J^e+p(-e)J^{-e})(d_{\emptyset}\cdot e)-(p(e)J^e-p(-e)J^{-e})(d_{\emptyset}\cdot d_{\emptyset})\bigr]>0,
\]
if we suppose that 
\[
\text{for $i=1,\ldots,d$ such that $d_{\emptyset} \cdot e^{(i)}>0$,} \qquad d_{\emptyset} \cdot e^{(i)}\geq\abs{\abs{d_{\emptyset}}}^2.
\]

Finally $v_{\ell}(1)\cdot v_{\ell}'(1)=\sum_{i=1}^d H(\abs{e^{(i)}})\cdot d_{\emptyset}>0$, so that Proposition~\ref{incr_speed} is proved.
\qed

\section*{Acknowledgements}
I would like to thank my advisor Christophe Sabot for suggesting this problem and for his support. 

I also would like to thank Pierre Mathieu for useful discussions and comments.

Finally I am grateful to the anonymous referee for his careful reading and for all his remarks.

My research was supported by the A.N.R.~\lq\lq MEMEMO\rq\rq.

\end{document}